\documentclass[11pt,a4paper]{amsproc}
\usepackage{tgtermes}
\usepackage{amsthm}
\usepackage{amsmath}
\usepackage{amssymb}
\usepackage{amscd}
\usepackage[latin2]{inputenc}
\usepackage{t1enc}
\usepackage[mathscr]{eucal}
\usepackage{indentfirst}
\usepackage{graphicx}
\usepackage{graphics}
\usepackage{pict2e}
\usepackage{epic}
\numberwithin{equation}{section}
\usepackage[margin=1.1in]{geometry}
\usepackage{epstopdf} 
\usepackage{tikz-cd}
\usepackage{mathtools}
\usepackage{cases}
\usepackage{hyperref}
\usepackage{mathrsfs}
\usepackage{verbatim,amsmath,amsthm,amsfonts,amssymb,latexsym,graphicx,mathtools,extpfeil,color,mathabx}
\usepackage{epstopdf,pinlabel}
\usepackage{amsfonts}
\usepackage[all]{xy}
\usepackage{graphicx}
\usepackage{caption}
\usepackage{subcaption}
\usepackage{tikz} 
\usetikzlibrary{calc,knots}
\DeclareMathAlphabet{\mathpzc}{OT1}{pzc}{m}{it}

\theoremstyle{plain}
\newtheorem{Th}{Theorem}[section]
\newtheorem{Lemma}[Th]{Lemma}
\newtheorem{Cor}[Th]{Corollary}
\newtheorem{Prop}[Th]{Proposition}
\newtheorem{Problem}[Th]{Problem}

 \theoremstyle{definition}
\newtheorem{Def}[Th]{Definition}

\newtheorem{Rem}[Th]{Remark}
\newtheorem{?}[Th]{Question}

\newcommand{\RN}{\mathbb{R}}
\newcommand{\CN}{\mathbb{C}}

\newcommand{\ZN}{\mathbb{Z}}

\newcommand{\la}{\langle}
\newcommand{\ra}{\rangle}

\newcommand\norm[1]{\left\lVert#1\right\rVert}

\makeatletter
\def\l@subsection{\@tocline{2}{0pt}{2.5pc}{5pc}{}}
\makeatother

\title[The three-dimensional Seiberg-Witten equations for $3/2-$spinors]{The three-dimensional Seiberg-Witten equations for $3/2-$spinors: a compactness theorem}

\author{Ahmad Reza Haj Saeedi Sadegh}
\address{Northeastern University}
\email{a.hajsaeedisadegh@northeastern.edu}

\author{Minh Lam Nguyen}
\address{Washington University in St. Louis}
\email{minhn@wustl.edu}


\begin{document}

	\maketitle
	
	\begin{abstract}
        The Rarita-Schwinger-Seiberg-Witten (RS-SW) equations are defined similarly to the classical Seiberg-Witten equations, where a geometric non-Dirac-type operator replaces the Dirac operator called the Rarita-Schwinger operator. In dimension four, the RS-SW equation was first considered by the second named author \cite{nguyen2023pin}. The variational approach will also give us a three-dimensional version of the equations. The RS-SW equations share some features with the multiple-spinor Seiberg-Witten equations, where the moduli space of solutions could be non-compact. In this note, we prove a compactness theorem regarding the moduli space of solutions of the RS-SW equations defined on 3-manifolds.


	\end{abstract}

	\tableofcontents

 \newpage

\section{Introduction}

Let $Y$ be a closed Riemannian 3-manifold with fixed spin structure $\mathfrak{s}_{1/2}$ with $W_{\mathfrak{s}_{1/2}}$ denoting the spinor bundle. We denote by $D$ the Dirac operator on $W_{\mathfrak{s}_{1/2}}$ and denote by $D^T$ the twisted Dirac operator on the twisted vector bundle $W_{\mathfrak{s}_{1/2}}\otimes TY$, where $TY$ is equipped with the Levi-Civita connection. The bundle of 3/2-spinors, $W_{\mathfrak{s}_{3/2}}$, is given as the kernel of the Clifford multiplication
morphism $\gamma:W_{\mathfrak{s}_{1/2}}\otimes TY\to W_{\mathfrak{s}_{1/2}}$. We have an orthogonal decomposition
of vector bundles $W_{\mathfrak{s}_{1/2}}\otimes TY=W_{\mathfrak{s}_{1/2}}\oplus W_{\mathfrak{s}_{3/2}}$, with respect to which the twisted Dirac operator has the following matrix form:
\[D^T=\left(
    \begin{array}{c|c}
       D & P^*\\
      \hline
      P & Q
    \end{array}
    \right)\]
where $P:\Gamma(W_{\mathfrak{s}_{1/2}})\to\Gamma(W_{\mathfrak{s}_{1/2}}\otimes TY)$ is the \emph{Penrose operator} and $Q:\Gamma(W_{\mathfrak{s}_{3/2}})\to \Gamma(W_{\mathfrak{s}_{3/2}})$ 
is the \textit{Rarita-Schwinger operator}. Similar to Dirac operators, the operator $Q$ is a first-order elliptic differential operator that is essentially self-adjoint. The bundle of $3/2$-spinors is not a Clifford module so the Rarita-Schwinger operator is not a Dirac-type operator.

For the gauge-theoretic purposes, fix a $U(1)$-line bundle $\mathscr{L}$ over $Y$ and denote by $\mathcal{A}(\mathscr{L})$ its affine space of all $U(1)$-connections.
For any $A\in\mathcal{A}(\mathscr{L})$, one then obtains the twisted Rarita-Scwhinger operator
\[Q_{A}:\Gamma(W_{\mathfrak{s}_{3/2}}\otimes\mathscr{L})\to \Gamma(W_{\mathfrak{s}_{3/2}}\otimes\mathscr{L}).\]
The Rarita-Schwinger-Seiberg-Witten equation (RS-SW) is the following system
 for the pairs $(A,\phi)\in\mathcal{A}(\mathscr{L})\times \Gamma(W_{\mathfrak{s}_{3/2}}\otimes\mathscr{L})$:
 \begin{equation}\label{eq:SWRS}
    \begin{cases}
      Q_A\phi=0\\
      F_A=\mu(\phi)
    \end{cases} \end{equation}
where $F_A$ is curvature of the connection $A$, and $\mu$ is the quadratic function that gives the traceless part of $\phi\phi^*$. Fix $A_0$ to be the referenced $U(1)$-connection in $\mathcal{A}(\mathscr{L})$. The solutions of this equation correspond to critical points of the modified Chern-Simons-Dirac functional 
\begin{equation}\label{eq:csrsfunctional}
    \mathcal{L}^{RS}(A,\phi)=-\int_Y(A-A_0)\wedge(F_{A}+F_{A_0})+\frac{1}{2}\int_Y\langle Q_A\phi,\phi\rangle
\end{equation}
where the Dirac operator is replaced by the Rarita-Schwinger operator.

A solution $(A,\phi)$ of \eqref{eq:SWRS} is called \textit{reducible} if $\phi=0$. By blowing up along the locus of the reducible solutions \cite{kronheimer_mrowka_2007}, we then obtain the blown-up equations 
\begin{equation}\label{eq:SWRSblowup}
    \begin{cases}
    \norm{\psi}_{L^4}=1\\
      Q_A\psi=0\\
      \epsilon^2F_A=\mu(\psi)
    \end{cases}\end{equation}
for triples $(A,\psi,\epsilon)\in \mathcal{A}(\mathscr{L})\times\Gamma(W_{\mathfrak{s}_{3/2}}\otimes \mathscr{L})\times (0,\infty)$. In this paper, we will study the moduli space of solutions of \eqref{eq:SWRSblowup} up to gauge transformations. 

There have been several generalizations to the classical Seiberg-Witten equations, one of which is the \textit{multiple-spinor Seiberg-Witten equations}. The four-dimensional version of the equations was first considered in \cite{MR1392667} and later on was systematically studied by Taubes in \cite{Taubes:2016voz}. Haydys and Walpuski \cite{haydys2015compactness} and several other people also study the three-dimensional variant of the multiple-spinor Seiberg-Witten equations. Broadly speaking, each variant of the multiple-spinor Seiberg-Witten equations associates with a choice of a compact Lie group $G$ equipped with a quaternionic representation $G\to Sp(S)$, where $Sp(S)$ is the group of quaternion-linear isometries on a quaternionic vector space $S$. For example, when $G=U(1)$, then one obtains the classical Seiberg-Witten equations. Note that Dirac (or Dirac-type) operators are among the key ingredients in defining these equations.

One could consider \eqref{eq:SWRS} as a kind of generalization to the classical Seiberg-Witten equations where the Dirac operator is replaced with the Rarita-Schwinger operator. The Rarita-Schwinger operator was first defined by Rarita and Schwinger in 1941 to study wave functions in supergravity and superstring theory \cite{PhysRev.60.61}. It is studied extensively in theoretic Physics. In contrast to its Dirac cousin, the Rarita-Schwinger operator has not been studied as much in Mathematics. However, there are some recent interests in the applications of the Rarita-Schwinger operator in geometry and topology by several people, see \cite{MR1129331}, \cite{homma2019kernel} and \cite{MR4252883}.

\subsection{Main results}
Similar to some other variants of the generalized Seiberg-Witten equations, the RS-SW equations could have non-compact moduli space of solutions. In dimension four, the non-compactness of the moduli space of the RS-SW equations is exhibited by some topological condition on the manifold via $Pin(2)-$equivariant degree theory \cite{nguyen2023pin}. In dimension three, to the best of our knowledge, there is no such analogous statement. As a result, one must take on a more analytical approach to understand the behavior of solutions of the three-dimensional RS-SW equations. 

Since the quadratic map $\mu$ in the curvature equation of the system fails to be proper, we do not have an a priori universal bound on the $3/2-$spinors. However, if a hypothetical universal bound does exist for the $3/2-$spinor solutions, then roughly speaking, via integration-by-parts argument, we should expect that the moduli space should be (sequentially) compact in some appropriate sense. Note that the first equation in the system is homogeneous. This suggests that we should re-scale the equations so that, at the very least, we have a uniform finite norm of the $3/2-$spinor solutions at the expense of introducing an extra unknown to the system. Hence, the hope is that some sort of integration-by-parts argument will also apply to us.

With that being said, regarding the moduli space of the RS-SW equations, the main featured theorem of the paper is the following.

\begin{Th}[cf. Theorem \ref{first compactness theorem}, Theorem \ref{second compactness theorem}]\label{first main theorem} Fix a Riemannian metric $g$ on a smooth, oriented, closed three-dimensional manifold $Y$.
\begin{enumerate}
    \item Let $\{(A_n, \psi_n, \epsilon_n)\} \subset \mathcal{A}(\mathscr{L}) \times \Gamma(W_{\mathfrak{s}_{3/2}} \otimes \mathscr{L}) \times (0, \infty)$ be a sequence of solutions of the blown-up RS-SW equations (cf. \eqref{eq:first compactness.2}). There exists a constant $c$ depending only on the Riemannian metric $g$ on $Y$, $A_0$ such that if $\limsup \epsilon_n > c$, then after passing through a subsequence and up to gauge transformations $(A_n, \psi_n, \epsilon_n)$ converges to a solution $(A,\psi, \epsilon)$ in the $C^\infty$ topology.
    \item Let $\{(A_n, \psi_n, \epsilon_n)\} \subset \mathcal{A}(\mathscr{L}) \times \Gamma(W_{\mathfrak{s}_{3/2}}\otimes \mathscr{L}) \times (0,\infty)$ be a sequence of solutions of \eqref{eq:first compactness.2} where $\{F_{A_n}\}$ is uniformly bounded in $L^6-$norm. If $\limsup \epsilon_n = 0$, then
    \begin{enumerate}
        \item There is a closed nowhere-dense subset $Z \subset Y$, a connection $A \in \mathcal{A}(\mathscr{L}|_{Y\setminus Z})$ and a $3/2-$spinor $\psi \in \Gamma(Y\setminus Z, W_{\mathfrak{s}_{3/2}}\otimes \mathscr{L})$ such that $(A,\psi, 0)$ solves \eqref{eq:first compactness.2}. Furthermore, after to passing through a subsequence, $|\psi_n|$ converges to $|\psi|$ in $C^{0,\alpha}-$topology. Specifically, $Z = |\psi|^{-1}(0)$.
        \item On $Y\setminus Z$, up to gauge transformations and after passing through a subsequence, $A_n$ converges weakly to $A$ in $L^{2}_{1,loc}$ and $\psi_n$ converges weakly to $\psi$ in $L^2_{2,loc}$.
    \end{enumerate}
\end{enumerate}
\end{Th}

As we vary the Riemannian metric on the three-manifold, a version of Theorem \ref{first main theorem} is still true. Specifically, we have

\begin{Th}[cf. Theorem \ref{varying compactness theorem}]\label{second main theorem} 
    Denote $\mathfrak{M}$ by the space of all Riemannian metrics on $Y$. Let $\{g_n\}$ be a sequence of metrics on $Y$ converging to $g \in \mathfrak{M}$. Let $\{(A_n, \psi_n, \epsilon_n)\}$ be a sequence of solutions of the $g_n-$\eqref{eq:first compactness.2} equations (i.e, the blown-up $g_n-$RS-SW equations) such that the $L^6-$norms of $F_{A_n}$ are all uniformly bounded.
    \begin{enumerate}
        \item If $\limsup \epsilon_n > c >0$, then after passing through a subsequence and up to gauge transformations $\{(g_n, (A_n, \psi_n, \epsilon_n))\}$ converges to $(g, (A, \psi, \epsilon))$ in the $C^\infty$ topology. 
        \item If $\limsup \epsilon_n = 0$, then there exists a closed nowhere-dense subset $Z \subset Y$, a connection $A$ on $Y\setminus Z$, a $g-3/2-$spinor $\psi$ on $Y\setminus Z$ such that
            \begin{enumerate}
                \item $Q_A \psi = 0$ and $\mu(\psi) = 0$
                \item $\displaystyle \int_{Y\setminus Z} |\psi|^4 = 1$ and $\displaystyle \int_{Y\setminus Z} |\nabla_A \psi|^2 < \infty$
                \item $|\psi|$ extends to a $C^{0,\alpha}-$H\"older continuous function on $Y$ where $Z= |\psi|^{-1}(0)$;
            \end{enumerate}
            furthermore, $A_n$ converges weakly to $A$ in $L^2_{1,loc}$ and $\psi_n$ converges weakly to $\psi$ in $L^2_{2,loc}$ on $Y\setminus Z$.
    \end{enumerate}
\end{Th}

\begin{Rem}
    One should compare our main theorem (cf. Theorem \ref{first main theorem}) with the compactness results of Haydys-Walpuski in \cite{haydys2015compactness}, and of Taubes in \cite{Taubes:2016voz}. However, there are two notable differences in our statement. Firstly, we have to blow up our equations by rescaling the $3/2-$spinor unknown by the reciprocal of its global $L^4-$norm. Secondly, we have to impose a universal bound on the curvatures of the connections. These two differences will addressed in a later subsection. It would be interesting to know if one has an analogous result without assuming the universal bound on the curvatures.
\end{Rem}

Follow \cite{haydys2015compactness}, we call the potential limiting solutions that appear in the moduli space of the RS-SW equations $3/2-$\textit{Fueter sections} (cf. Definition \ref{3/2 fueter section}). It turns out that the existence of $3/2-$Fueter sections is an obstruction for the moduli space to be compact.

\begin{Cor}[cf. Corollary \ref{obstruction for compactness}]\label{Corollary obstruction}
    Let $g \in \mathfrak{M}$. If there is no $g-3/2-$Fueter section, then the charged-$k$ moduli space (cf. Definition \ref{charged moduli space}) is always compact on some neighborhood of $g$ in $\mathfrak{M}$.
\end{Cor}

\begin{Rem}
    In the multiple-monopole Seiberg-Witten equations setting, the limiting objects of the associated moduli space correspond to solutions of a certain non-linear Dirac operator defined on a hyper-K\"ahler fiber bundle over the three-manifold \cite{MR2980921}. We suspect that an analogous correspondence also holds for our $3/2-$Fueter sections. In particular, we believe that the $3/2-$Fueter sections should appear as solutions to a certain non-linear Rarita-Schwinger operator defined on some hyper-K\"ahler fiber bundle over $Y$. This circle of ideas will be addressed elsewhere in our future work.
\end{Rem}

\begin{Rem}
    The existence of Fueter sections (or $\ZN/2-$harmonic spinors) have been shown by Doan-Walpuski in \cite{doan2021existence}. To the best of our knowledge, there is no such analogous statement for the $3/2-$Fueter sections.
\end{Rem}

Lastly, just as the study of Rarita-Schwinger fields is intimately tied with the deformation problem of Einstein metrics, the non-existence of $g-3/2-$Fueter sections gives us some interesting information about the Einstein geometry of the underlying three-manifold. Recall that to define the RS-SW equations, besides picking a Riemannian metric $g$, we also have to pick a $\text{spin}^c$ structure.

\begin{Cor}[cf. Corollary \ref{geometry of y}]
    Let $\mathfrak{s}_{1/2}$ be a $\text{spin}^c$ structure on $Y$ such that $c_1(\mathfrak{s}_{1/2})$ is torsion in $H^2(Y;\ZN)$. Let $g$ be a Riemannian metric on $Y$. If there is no $g-3/2-$Fueter section associated to the RS-SW equations corresponding to $g$ and $\mathfrak{s}_{1/2}$, then
    \begin{enumerate}
        \item The charged moduli space $\mathcal{M}_k(g)$ is compact for some charge $k$.
        \item Either $g$ is not an Einstein metric with non-negative scalar curvature or $(A,\psi)$ is never a solution where $A$ is a flat $U(1)-$connection on $det\,(\mathfrak{s}_{1/2})$.
    \end{enumerate}
\end{Cor}

\subsection{Outline of the argument}

As a convention, all geodesic balls we consider throughout this paper have radius $r \in (0,\ r_0]$, where $r_0$ is the injective radius of $Y$ that is taken to be much lesser than 1.
Sobolev completions of the space of connections and spinors play an essential role in the analysis of various mathematical gauge-theoretic elliptic PDEs, so they will be of importance in our setup. Firstly, we describe how these spaces are defined in the setting of the RS-SW equations. We use the notation $L^p_k$ to describe the space of $L^p-$integrable objects with their derivatives up to order $k$.

Fix a referenced $U(1)-$connection $A_0$ on a complex line bundle $\mathscr{L}$ over $Y$. Every other connection $A$ can be written as $A= A_0 + a$, where $a \in \Gamma(Y, T^*Y \otimes i\RN) := i\Omega^1(Y)$. Denote $\mathcal{A}(\mathscr{L})$ by the space of all $U(1)-$connections on $\mathscr{L}$. Then $\mathcal{A}(\mathscr{L})$ can be viewed as an affine space over $i\Omega^1(Y)$. The $L^p_k-$Sobolev completion of $\mathcal{A}(\mathscr{L})$ is taken to be the Sobolev completion of $i\Omega^1(Y)$ with respect to the fixed referenced connection $A_0$. In particular, $\norm{A}^p_{L^p_k}$ is taken to by $ \sum_{j=0}^{k}\norm{\nabla^{\otimes j}_{A_0}a}^p_{L^p}$. The topology on $\mathcal{A}(\mathscr{L})$ induced by the $L^p_k-$norm does not change as one changes the referenced connection.

The $3/2-$spinors are sections of the bundle $T^*Y \otimes W_{\mathfrak{s}_{1/2}}\otimes \mathscr{L}$. A covariant derivative for the $3/2-$spinors is constructed from the Levi-Civita connection and a covariant derivative $\nabla_A$, where $A \in \mathcal{A}(\mathscr{L})$. We consider the $L^p_k-$Sobolev completion of the space of $3/2-$spinors by using $\nabla_{A_0}$. That is, $\norm{\psi}^{p}_{L^p_k} := \sum_{j=1}^{k} \norm{\nabla^{\otimes j}_{A_0}\psi}^{p}_{L^p}$. Once again, changing the fixed referenced connection does change the topology induced by the Sobolev norm on $\Gamma(T^*Y \otimes W_{\mathfrak{s}_{1/2}}\otimes \mathscr{L})$.

With the notations explained, we give an outline for the proof of Theorem \ref{first main theorem}. 

Following \cite[Chapter 6]{kronheimer_mrowka_2007}, and similar to \cite{haydys2015compactness}, we study a blow-up version of the RS-SW equations. The first part of the theorem (cf. Theorem \ref{first compactness theorem}) follows from a combination of the Weitzenb\"ock formula of the Rarita-Schwinger operator (cf. Lemma \ref{RS clifford lemma 2.12} or Lemma \ref{weitzenbock formula}) and Green's integration by parts formula (cf. Lemma \ref{first compactness lemma 1.6}). This is a standard technique that is also used to prove similar compactness results for the multiple-spinor Seiberg-Witten equations \cite{haydys2015compactness} and the equations of flat $PSL(2,\CN)-$connections \cite{taubes2012psl}. However, there is one notable technical difference that we now describe. Unlike the Dirac operator, the square of the Rarita-Schwinger operator contains an additional second-order term contributed by the \textit{Penrose operator} $P_A$ (a more detailed definition is given in Subsection 2.1). As a result, in our version of Green's integration by parts formula, there is an appearance of $\norm{P^*_A \psi}_{L^2}$ which we have to control. Essentially, we can control the divergence of $\psi$ by the curvature $F_A$ (cf. Lemma \ref{first compactness lemma 1.5}) via Uhlenbeck's gauge fixing lemma (cf. Lemma \ref{first compactness lemma 1.4}) at the expense of $\psi$ having unit $L^4-$norm. In short, the main reason to use the $L^4-$norm has to do with the requirement for higher regularity conditions due to the appearance of the Penrose operator in the Weitzenb\"ock formula. Thus in our set-up, we blow up the equations differently compared to other variants of generalized Seiberg-Witten equations \cite{haydys2015compactness} by re-scaling the $3/2-$spinor by the reciprocal of its global $L^4-$norm. Since the $L^4-$norm is continuous with respect to the topology of the configuration space, this is enough for the blow-up construction.  

The second part of the theorem (cf. Theorem \ref{second compactness theorem}) is more involved. We prove that, for a solution $(A,\psi,\epsilon)$ of the blown-up RS-SW equations \eqref{eq:first compactness.2}, we can uniformly control the $L^2_{2,A}-$norm of $\psi$ on a geodesic ball $B(x,r)$ on $Y$ as long as $r$ is less than the \textit{critical radius} (cf. Definition \ref{first compactness defintion 2.8}) $\rho(x)$ at $x$. To control $\rho$, we use a \textit{frequency function} $N(r)$, which measures how much $\psi$ vanishes near $x$. Roughly speaking, if the vanishing order of $\psi$ near $x$ is not too large, then we always have a lower bound for the critical radius. Hence by Uhlenbeck's local slice theorem (cf. Lemma \ref{local slice theorem}), connections can be put in appropriate gauge-slice so that we can establish convergence away from the vanishing set of the $3/2-$spinors.

The frequency function is an important tool for studying the nodal set of eigenfunctions and the growth rate of solutions of elliptic PDEs. It was first developed by Almgren \cite{MR0574247} and Agmon \cite{MR0252808}. The gauge-theoretic version of the frequency function was first considered by Taubes \cite{taubes2012psl} and then Haydys-Walpuski \cite{haydys2015compactness}. In spirit, we follow closely the approach of \cite{taubes2012psl}, \cite{taubes2014zero}, \cite{Taubes:2016voz}, \cite{haydys2015compactness} in our setup with a technical difference. For us to obtain a uniform control of the $L^2_{2,A}-$norm of the rescaled $3/2-$spinor solution $\psi$ on $B(x,r)$ and a control the critical radius $\rho(x)$, we must establish some local estimates in terms of $\epsilon$ for $\psi$. Just as in the case of the first part of the main theorem, to achieve this, we have to rely on the control of divergence of $\psi$ by the curvature $F_A$. Therefore, it seems natural for us to impose a universal for bound for the curvatures $F_A$ and study the behavior of such a sequence of solutions.

\subsection{Organization of the paper}

In this section, we briefly survey the current literature about the Rarita-Schwinger operator. A motivation to why we study the RS-SW equations as a generalization of the classical Seiberg-Witten equations is also given. Although the main analytical problem of the RS-SW equations is similar to other variants of generalized Seiberg-Witten equations, some of the stark differences are addressed. The main theorem (cf. Theorem \ref{first main theorem}) regarding the compactness phenomenon of the moduli space of the RS-SW equations is stated. Lastly, we give a brief outline of the strategy employed in proving the main theorem and discuss the motivation of the approach we take.

In Section 2, we present some preliminary material about the construction of Rarita-Schwinger operators. A general definition of the Rarita-Schwinger operator on a Clifford module is given in Subsection 2.1 (cf. Definition \ref{the equations defintion 2.1}). In Subsection 2.2, we construct the Rarita-Schwinger operator associated with a $\text{spin}^c$ structure on a three or four-dimensional manifold. Various Weitzenb\"ock-type formulas associated with the Rarita-Schwinger operators will be derived in Section 2 (cf. Lemma \ref{RS clifford lemma 2.8}, Lemma \ref{RS clifford lemma 2.12}). These formulas play an important role in many computations and estimates which are employed throughout the paper. Although the results of the paper mainly concern the three-dimensional setting, we think that it is also appropriate to include the four-dimensional construction so that we have the terminologies to discuss our future works in the discussion section. Having established the background, in Subsection 2.3, we define the three-dimensional RS-SW equations (cf. \eqref{eq:RSSW.1}, \eqref{eq:first compactness.1}).

The proof of the main theorem (cf. Theorem \ref{first main theorem}) is divided into two parts. In Section 3, we give proof of the first part of our main theorem (cf. Theorem \ref{first compactness theorem}). The rest of the paper, namely Section 4--6, will be devoted to the proof of the second part of the main theorem (cf. Theorem \ref{second compactness theorem}). The technical tool in the analysis of PDEs we use to prove the second part of the main theorem is the \textit{frequency function} method. To apply this method effectively to our gauge-theoretic equations, we must exhibit some sort of control on the $3/2-$spinors by the curvatures. The main content of Section 4 is to establish this control. In section 5, we define our gauge-theoretic frequency function (cf. Definition \ref{frequency function definition 3.1}) associated with the RS-SW equations (cf. \eqref{eq:RSSW.1}, \eqref{eq:first compactness.1}). Various necessary analyses of our frequency function will be established to set up for Section 6 (cf. Proposition \ref{frequency function proposition 3.14}, Corollary \ref{frequency function corollary 3.11}, Corollary \ref{frequency function corollary 3.13}, Proposition \ref{frequency function proposition 3.16}). In Section 6, we give a proof of the second part of our main theorem (cf. Theorem \ref{second compactness theorem}).

Section 7 is reserved for the discussion of some of the conjectural pictures regarding the RS-SW equations. Several of the future works and problems will also be stated in this section.

\section{The equations}

\subsection{Rarita-Schwinger operator on a Clifford module}

Suppose $M$ is a smooth, oriented, closed Riemannian manifold of dimension $n$ and $S$ is a Clifford module over $M$. Let $\nabla$ be a compatible connection on  $S\to M$, i.e., it is compatible with a fixed hermitian metric on $S$ and the Levi-Civita connection $\nabla^{LC}$ of $M$ associated with some fixed Riemannian metric $g$ on $M$ (e.g, see \cite{MR0960889} or \cite{MR1215720}). Together with the Clifford multiplication $\gamma$, we have $D = \gamma \circ \nabla: \Gamma(S) \to \Gamma(S)$, which is called the \textit{Dirac operator} of $S$. When $n$ is even, there is a chirality of $S$ induced by $\gamma$ extended to top complexified volume form $\omega_{\CN}$ of $M$. Then one has the following orthogonal decomposition $S = S^+ \oplus S^-$, where $S^{\pm}$ corresponds to the $\pm 1-$eigenspace of $\gamma(\omega_{\CN})$. With respect to the above decomposition, we have 
\[D = \begin{pmatrix} 0 & D^- \\ D^{+} & 0 \end{pmatrix},\text{i.e,} \quad \quad \quad D^{\pm}: \Gamma(S^{\pm}) \to \Gamma(S^{\mp}).\]

Consider $S\otimes TM$. One can show that $S \otimes TM$ is also a Clifford module over $M$. The Clifford multiplication on such bundle is defined by tensoring $\gamma$ with the identity on $TM$. We are also going to use $\gamma$ to refer to this new Clifford multiplication. A compatible connection on $S \otimes TM$ is given by $\nabla := \nabla \otimes 1 + 1 \otimes \nabla^{LC}$. Then one obtains another Dirac operator $D^T: \Gamma(S\otimes TM) \to \Gamma(S \otimes TM)$. Again, when $n$ is even, $D^T$ exchanges the chirality of $S \otimes TM$,
\[D = \begin{pmatrix} 0 & D^{T-} \\ D^{T+} & 0 \end{pmatrix},\text{i.e,} \quad \quad \quad D^{T\pm }: \Gamma(S^{\pm}\otimes TM) \to \Gamma(S^{\mp}\otimes TM).\]

We want to think of $\gamma$ as follows. It is a linear map $\gamma : TM \otimes S \to S$. Denote $ker\, \gamma:= S_{3/2}$, this is a sub-bundle of $TM\otimes S$. One can show that we have the following orthogonal decomposition $TM \otimes S \cong S_{3/2} \oplus \iota (S)$, where $\iota : S \to S \otimes T^*M \cong S \otimes TM$ is an embedding given by
\[\iota(\phi)(v) = -\frac{1}{n} \gamma(v \otimes \phi) = -\frac{1}{n} \gamma(v) \phi, \quad \quad \phi \in \Gamma(S), v \in \Gamma(TM).\]
Let $\pi$ be the orthogonal projection from $TM \otimes S$ onto $S_{3/2}$. With this orthogonal projection, we have the \textit{Penrose operator} $P = \pi \circ \nabla$, and another first order operator $Q = \pi \circ D^T|_{S_{3/2}}$. In fact, with respect to the orthogonal decomposition $TM \otimes S \cong S_{3/2} \oplus \iota(S)$, one can write $D^T$ in matrix from as following
\begin{equation*}
   \displaystyle D^{T} = \begin{pmatrix} \frac{2-n}{n} \iota D \iota^{-1} & 2\iota P^* \\ \frac{2}{n}P\iota^{-1} & Q \end{pmatrix}.
\end{equation*}

\begin{Def}\label{the equations defintion 2.1}
$Q : \Gamma(S_{3/2}) \to \Gamma(S_{3/2})$ is called the \textit{Rarita-Schwinger operator} of a Clifford module $S$ over $M$.
\end{Def}

\begin{Def}
$\psi \in \Gamma(S_{3/2})$ is called a \textit{Rarita-Schwinger field} if and only if $Q\psi = 0$ and $P^* \psi = 0$. Since $P^* = \nabla^*$, $\psi$ is a Rarita-Schwinger field if and only if $\psi \in ker\, Q$ and divergence-free. Equivalently, one can view Rarita-Schwinger field $\psi$ as a harmonic $S-$valued $1-$form such that $\gamma(\psi) = 0$.
\end{Def}

\begin{Rem}
    One can check that an explicit formula of the projection $\pi$ is given by $\pi = 1 - \iota \circ \gamma$.
\end{Rem}

\begin{Rem}
    Just like $D$, $D^T$, it can be shown that $Q$ is a first-order elliptic operator that is formally self-adjoint, see \cite{MR4395331}. 
\end{Rem}

\begin{Rem}
    When $n$ is even, $Q$ also exchanges the chirality of $S_{3/2}$. We put in the superscript decorations $\pm$ appropriately in the context where we restrict our attention to positive (negative) Rarita-Schwinger fields.
\end{Rem}

\begin{Rem}
    Since $Q$ is a first-order elliptic operator, upon appropriate Sobolev completion of $\Gamma(S_{3/2})$, $Q$ extends to a bounded Fredholm operator.
\end{Rem}

\begin{Rem} 
    Via the canonical isomorphism induced by $g$, we have $\psi \in \Gamma(TM\otimes S)\cong \Gamma(T^*M\otimes S)$. Thus, we can view $\psi \in \Gamma(Hom(TM, S))$. 
\end{Rem}  

Next, we establish some Weitzenb\"ock-type formulas related to the Rarita-Schwinger operator $Q$. Consider the covariant exterior derivative $d^{\nabla}_k : \Lambda^k T^* M \otimes S \to \Lambda^{k+1}T^* M \otimes S$. Let $\pi_k: \Lambda^k T^* M \otimes S \to \Lambda^{k-1} T^* Y \otimes S$ be a map defined (locally) by
$$\pi_k(\psi)(v_1,\cdots, v_{k-1}) = \sum_{j}\gamma(e_j)\psi(e_j, v_1, \cdots, v_{k-1}).$$
As a convention, we set $\pi_0 = 0$. Note that $\pi_1(\psi) = \gamma(\psi)$. With the maps $\pi_k$ and $d^{\nabla}_k$, we define $D^{k,T} = \pi_{k+1} d^{\nabla}_k + d^{\nabla}_{k-1}\pi_k: \Lambda^k T^* M \otimes S \to \Lambda^k T^* M \otimes S$, which is called the \textit{higher-spin Dirac operator}. For example, consider $D^{0,T} \equiv D$ and $D^{1,T} \equiv D^T$. These are the Dirac operators on $S$ and $S \otimes TM$ that we defined above, respectively.

\begin{Lemma}\label{RS clifford lemma 2.7}
We have the following formulas
\begin{enumerate}
    \item $D^T \iota + \iota D = \displaystyle \frac{2}{n} d^{\nabla}$.
    \item $D^{T}d^{\nabla}_0 - d^{\nabla}_0 D = \displaystyle \pi_2(F^{\nabla})$, where $F^{\nabla}$ is the curvature of the connection $\nabla$ on a Clifford module $S$.
\end{enumerate}
\end{Lemma}

\begin{proof}
Let $\{e_j\}$ be a local orthonormal frame on $M$ and $\phi \in \Gamma(S)$. We have 
$$D^T\iota(\phi)(e_j) = \sum_i \gamma(e_i) \nabla_{e_i}\left( -\frac{1}{n}\gamma(e_j)\phi \right)= -\frac{1}{n} \sum_i \gamma(e_i)\gamma(e_j)\nabla_{e_i}\phi.$$
On the other hand,
$$\iota (D\phi)(e_j) = -\frac{1}{n}\gamma(e_j)D\phi = -\frac{1}{n}  \sum_{i}\gamma(e_j)\gamma(e_i) \nabla_{e_i} \phi.$$
Thus $D^T\iota(\phi)(e_j) + \iota(D\phi)(e_j) = \displaystyle -\frac{2}{n} \nabla_{e_j} \phi$. This proves the first part. To show the second part, we appeal to the definition of the higher-spin Dirac operator described above. We have
\begin{align}
    D^T d^{\nabla}_0 - d^{\nabla}_0 D & = (\pi_2 d^{\nabla}_1 + d^{\nabla}_0 \pi_1) d^{\nabla}_0 - d^{\nabla}_0 D \nonumber \\
    & = \pi_2 d^{\nabla}_1 d^{\nabla}_0 + d^{\nabla}_0 D - d^{\nabla}_0 D \nonumber  = \pi_2 F^{\nabla}
\end{align}
This is exactly the formula in the second part.
\end{proof}

\begin{Lemma}\label{RS clifford lemma 2.8}
We have the following formulas
\begin{enumerate}
    \item $\displaystyle \frac{2(2-n)}{n^2}PD + \frac{2}{n}QP = \frac{2}{n}\pi \pi_2(F^{\nabla})$.
    \item $\displaystyle Q^2 + \frac{4}{n}PP^*=\nabla^*\nabla + \frac{s}{4} + \pi(\gamma \otimes F^{\nabla})$, where $s$ is the scalar curvature of $g$ and $\gamma \otimes F^{\nabla}$ is understood as the Clifford contraction of the twisting curvature $F^{\nabla}$ on the Clifford module $S$.
\end{enumerate}
\end{Lemma}

\begin{proof}
    We apply $D^T$ to both sides of Lemma \ref{RS clifford lemma 2.7} part one and part two to obtain
    \begin{align}\label{eq:RS clifford.1}
        (D^T)^2\iota & = -D^T\iota D + \frac{2}{n}D^T d^{\nabla} = \left(\iota D - \frac{2}{n}d^{\nabla}\right)D + \frac{2}{n}D^T d^{\nabla}\nonumber \\
        &= \iota D^2 + \frac{2}{n}(D^T d^{\nabla}- d^{\nabla}D) = \iota D^2 + \frac{2}{n}\pi_2(F^{\nabla})
    \end{align}
    Note that in the matrix form of $D^T$, we have
    $$(D^T)^2 = \begin{pmatrix} \left(\frac{2-n}{n}\right)^2\iota D^2 \iota^{-1} + \frac{4}{n} \iota P^*P \iota^{-1} & \frac{2(2-n)}{n}\iota DP^* + 2\iota P^*Q \\ \frac{2(2-n)}{n^2} PD\iota^{-1} + \frac{2}{n} QP \iota^{-1} & \frac{4}{n} PP^* + Q^2 \end{pmatrix}.$$
    So the lower left block of $(D^T)^2$ is determined exactly by $\pi(D^T)^2\iota$. Thus, when combined with \eqref{eq:RS clifford.1}, we have
    $$\displaystyle \frac{2(2-n)}{n^2}PD + \frac{2}{n}QP = \frac{2}{n}\pi \pi_2(F^{\nabla}).$$
    This proves the first part. To see the second part, recall the Weitzenb\"ock formula for $(D^T)^2$,
    $$(D^T)^2= \nabla^* \nabla + \frac{s}{4} + \gamma \otimes F^{\nabla}.$$
    One simply applies $\pi$ to both sides of the above formula and arrives at the Weitzenb\"ock-type formula for the Rarita-Schwinger operator $Q$.
\end{proof}

\begin{Rem}
    Lemma \ref{RS clifford lemma 2.7} and Lemma \ref{RS clifford lemma 2.8} both have versions when we restrict our attention to either only $S^{\pm}_{3/2}$ in the case, $n$ is even. This will become more apparent when we expand these formulas on a $ 4-$ dimensional manifold.
\end{Rem}

Since $P^*P$ is an elliptic operator, we can orthogonally decompose $\Gamma(S_{3/2}) = ker\, P^* \oplus im\, P$. In general, it is not true that $Q$ should preserve this decomposition. However, in an ideal situation, we do have the following.

\begin{Lemma}\label{RS clifford lemma 2.10}
    If $\nabla$ is a flat connection of $S \to M$, then $Q$ preserves the orthogonal decomposition $ker\, P^* \oplus im\, P$.
\end{Lemma}

\begin{proof}
    If $\nabla$ is a flat connection, by the first part of Lemma \ref{RS clifford lemma 2.8}, we have 
    $$\frac{2(2-n)}{n^2}PD + \frac{2}{n}QP = 0 \Rightarrow \frac{2(2-n)}{n^2}DP^* + \frac{2}{n}P^*Q = 0.$$
    Clearly, if $\psi \in ker\, P^*$, then $P^*Q\psi = 0$ by the above formula. Thus, $Q\psi \in ker\, P^*$. On the other hand, if $\psi \in im\, P$, then there is a $\phi \in \Gamma(S)$ such that $P\phi = \psi$. Thus, $\displaystyle -\frac{2}{n}Q\psi = \frac{2(n-2)}{n^2}PD\phi$. Immediately, this implies that $Q\psi \in im\, P$. 
\end{proof}

\begin{Rem}
    Lemma \ref{RS clifford lemma 2.7} and Lemma \ref{RS clifford lemma 2.8} should be compared with similar results in \cite{MR1129331}. Our formulas are more general and hold for \textit{any} Clifford module.
\end{Rem}

\begin{Rem}
    Lemma \ref{RS clifford lemma 2.10} is a general version of a result in \cite{homma2019kernel} applied to the setting of \textit{any} Clifford module over the base manifold.
\end{Rem}

\subsection{Rarita-Schwinger operator of a $\text{spin}^c$ structure}
Let $Y$ be a $3-$manifold that is smooth, oriented, and closed. It is a standard fact that any such $3-$manifold is spin and thus $\text{spin}^c$. Consider $\mathfrak{s}_{1/2}$ to be any $\text{spin}^c$ structure over $Y$. The Levi-Civita connection combined with a choice of a unitary connection $2A$ on $det\,\mathfrak{s}_{1/2}$ gives us a compatible connection $\nabla_A$ on $W_{\mathfrak{s}_{1/2}}$. Following the general construction in the previous subsection, one obtains the following operators
$$D_A: \Gamma(W_{\mathfrak{s}_{1/2}}) \to \Gamma(W_{\mathfrak{s}_{1/2}}), \quad D^T_A: \Gamma(TY \otimes W_{\mathfrak{s}_{1/2}}) \to \Gamma(TY \otimes W_{\mathfrak{s}_{1/2}}),$$
$$Q_A : \Gamma(W_{\mathfrak{s}_{3/2}}) \to \Gamma(W_{\mathfrak{s}_{3/2}}),$$
where $W_{\mathfrak{s}_{3/2}}$ notation-wise refers to the kernel of $\gamma: TY \otimes W_{\mathfrak{s}_{1/2}} \to W_{\mathfrak{s}_{1/2}}$. Respectively, they are the spinor Dirac operator on the spinor bundle, the twisted Dirac operator on the twisted spinor bundle, and the Rarita-Schwinger operator associated with the $\text{spin}^c$ structure $\mathfrak{s}_{1/2}$. We want to emphasize that the construction of these various first-order elliptic operators depends not only on the Riemannian metric of $Y$ but also on a choice of a unitary connection on $det\,\mathfrak{s}_{1/2}$. Let $L = det\,\mathfrak{s}_{1/2}$. Then we can view $W_{\mathfrak{s}_{1/2}} = W_{\mathfrak{s}} \otimes_{\CN} L^{1/2}$, where $W_{\mathfrak{s}}$ is some spinor bundle over $Y$. The connection on the spinor bundle is determined by the Levi-Civita connection $\nabla^{LC}$ of $g$. From this perspective, a unitary connection $2A$ on $L$ corresponds to a unitary connection $A$ on $L^{1/2}$. Thus, the compatible connection $\nabla_A$ on $W_{\mathfrak{s}_{1/2}}$ maybe regarded as $\nabla^{LC} \otimes 1 + 1 \otimes \nabla_A$. Hence, $F_A = F^{\nabla^{LC}} \otimes 1 + 1 \otimes F_A$, where $F_A$ genuinely is a curvature on $L^{1/2}$. Note that if $\phi \otimes s \in \Gamma(W_{\mathfrak{s}}\otimes L^{1/2})$ where $\phi \in \Gamma(W_{\mathfrak{s}})$ and $s\in \Gamma(L^{1/2})$, then for $v \in \Gamma(TY)$ locally
$$\pi_2(F^{\nabla^{LC}}(\phi) \otimes S)(v) = \left(\sum_{i} \gamma(e_i)F^{\nabla^{LC}}(e_i,v)\phi\right)\otimes s = \frac{1}{2}Ric(v)\phi \otimes s,$$
where $Ric$ is the Ricci curvature of $g$. Furthermore, the Weitzenb\"ock formula for $D^T_A$ will look like the following in this setting
$$(D^T_A)^2 = \nabla^*_A \nabla_A + \frac{s}{4} + \gamma(F_A) - 1 \otimes Ric.$$ 
Then, one can rewrite Lemma \ref{RS clifford lemma 2.8} in this setting as follows.

\begin{Lemma} \label{RS clifford lemma 2.12}
    We have the following formulas
    \begin{enumerate}
        \item $\displaystyle -\frac{2}{9}P_A D_A + \frac{2}{3} Q_A P_A = \frac{1}{3}\left(Ric - \frac{s}{3}\right) + \frac{2}{3} \pi \pi_2(1\otimes F_A).$
        \item $\displaystyle Q^2_A + \frac{4}{3}P_A P^*_A = \nabla^*_A \nabla_A + \frac{s}{4} + \pi(\gamma(F_A)) - \pi(1\otimes Ric).$ 
    \end{enumerate}
\end{Lemma}

\begin{Cor} \label{RS clifford lemma 2.13}
    If $g$ is an Einstein metric on $Y$, then we have the following:
    \begin{enumerate}
        \item $\displaystyle -\frac{2}{9}P_A D_A + \frac{2}{3} Q_A P_A =  \frac{2}{3} \pi \pi_2(1\otimes F_A).$
        \item $\displaystyle Q^2_A + \frac{4}{3}P_A P^*_A = \nabla^*_A \nabla_A - \frac{s}{12} + \pi(\gamma(F_A)).$ 
    \end{enumerate}
\end{Cor}

\begin{proof}
    If $g$ is an Eistein metric, then $\displaystyle Ric = \frac{s}{3}$ and $\displaystyle \pi(1\otimes Ric) = \frac{s}{3}$. Then, the results immediately follow. 
\end{proof}

Similarly, on a smooth, oriented, closed $4-$manifold $X$, one can also consider a Rarita-Schwinger operator associated to a $\text{spin}^c$ structure $\mathfrak{s}_{1/2}$ on $X$. Again, the construction of the Rarita-Schwinger operator $Q^{\pm}_A$ associated to $\mathfrak{s}_{1/2}$ depends on the Riemannian metric of $X$ and a choice of a unitary connection $2A$ on $det\, \mathfrak{s}_{1/2}$. Following the discussion above for $3-$manifold, we obtain the following lemmas regarding various Weitzenbock-type formulas for $Q^{\pm}_A$.

\begin{Lemma}
    We have the following formulas
    \begin{enumerate}
        \item $\displaystyle -\frac{1}{4}P^{\pm}_A D^{\mp}_A \iota^{-1} + \frac{1}{2} Q^{\mp}_A P^{\mp}_A \iota^{-1} = \frac{1}{4}\left(Ric - \frac{s}{4}\right) + \frac{1}{2}\pi\pi_2(1\otimes F^{\pm}_A)$.
        \item $\displaystyle Q^-_A Q^+_A + P^+_A P^{+*}_A = \nabla^*_A \nabla_A + \frac{s}{4} + \pi(\gamma(F^+_A)) - \pi(1\otimes Ric)^+.$
    \end{enumerate}
\end{Lemma}

\begin{Cor}
    If $g$ is an Einstein metric on $X$, then we have the following
    \begin{enumerate}
        \item $\displaystyle -\frac{1}{4}P^{\pm}_A D^{\mp}_A \iota^{-1} + \frac{1}{2} Q^{\mp}_A P^{\mp}_A \iota^{-1} = \frac{1}{2}\pi\pi_2(1\otimes F^{\pm}_A). $
        \item $\displaystyle Q^-_A Q^+_A + P^+_A P^{+*}_A = \nabla^*_A \nabla_A +  \pi(\gamma(F^+_A)).$
    \end{enumerate}
\end{Cor}

\subsection{The Seiberg-Witten equations for $3/2-$spinors}

Let $\mathfrak{s}_{1/2}$ be a $\text{spin}^c$ structure on a $3$-manifold $Y$. We consider the following system
\begin{equation}\label{eq:RSSW.1}
    \begin{cases}
        Q_A \psi = 0,\\
        F_A = \gamma^{-1}(\mu(\psi)).
    \end{cases}
\end{equation}
The unknowns of the system \eqref{eq:RSSW.1}  are pairs $(A, \psi)$, where $2A \in \mathcal{A}(det\, \mathfrak{s}_{1/2}),$ which is the space of all unitary connections on the determinant line bundle, and $\psi \in \Gamma(W_{\mathfrak{s}_{3/2}})$. We have that the Clifford multiplication $\gamma : i\Lambda^2 T^*Y \to i\mathfrak{su}(W_{\mathfrak{s}_{1/2}})$ is an isometry and $\mu$ is a quadratic map defined on $\Gamma(W_{\mathfrak{s}_{3/2}})$ as following
$$\psi \in \Gamma(W_{\mathfrak{s}_{3/2}}) \subset \Gamma(W_{\mathfrak{s}_{1/2}} \otimes TY) \cong \Gamma(W_{\mathfrak{s}_{1/2}}\otimes T^*Y).$$
This means that we can view $\psi \in \Gamma(End\,(TY, W_{\mathfrak{s}_{1/2}}))$. Then $\psi \psi^*$ as composition of operators is a section of the bundle $End\,(W_{\mathfrak{s}_{1/2}})$ that is self-adjoint. And $\mu(\psi)$ is taken to be the traceless part of $\psi\psi^*$ so that $\mu(\psi) \in i\mathfrak{su}(W_{\mathfrak{s}_{1/2}})$. Unlike the quadratic map defined in the classical Seiberg-Witten equations, this map $\mu$ is not proper, i.e., there might be a non-zero $\psi$ such that $\mu(\psi) = 0$. This feature is one of the many difficulties in analyzing solutions of \eqref{eq:RSSW.1}.

The equations \eqref{eq:RSSW.1} fall under the umbrella of abelian gauge theory. The gauge group is given by $\mathcal{G} = Maps(Y \to U(1))$ and it acts on the configuration space $\mathcal{C} = \mathcal{A}(det\, \mathfrak{s}_{1/2}) \times \Gamma(W_{\mathfrak{s}_{3/2}})$ of \eqref{eq:RSSW.1} by the following: For $u \in \mathcal{G}, 2A \in \mathcal{A}, \psi \in \Gamma(W_{\mathfrak{s}_{3/2}})$,
$$u\cdot A = A - u^{-1}du, \quad u \cdot \psi = u\psi.$$
Let $\mathcal{F}$ be the defining map of \eqref{eq:RSSW.1},
$$\mathcal{F}(A,\psi) = (Q_A \psi, F_A - \gamma^{-1}(\mu(\psi))).$$
$\mathcal{F}$ can be shown to be $\mathcal{G}-$equivariant. As a result, solutions of \eqref{eq:RSSW.1} are preserved under $\mathcal{G}-$symmetry.

It turns out that \eqref{eq:RSSW.1} can be viewed as the equations of motion of the following modified Chern-Simon-Dirac functional that was introduced in the Introduction section. Let $A_0$ be a fixed referenced connection. We recall the definition of the functional (cf. \eqref{eq:csrsfunctional}). 
\begin{equation*}
    \mathcal{L}^{RS}(A,\psi) = - \int_Y (A-A_0)\wedge (F_A + F_{A_0}) + \dfrac{1}{2}\int_{Y} \la Q_A \psi, \psi\ra.
\end{equation*}
Consider a small change in $(A,\psi)$ given by $(A + tb, \psi + t\phi)$, where $t \in (-\epsilon, \epsilon)$. Note that 
\begin{align*}
(A+tb - A_0) \wedge (F_A + tdb + F_{A_0}) &= (A-A_0)\wedge(F_A + F_{A_0}) + \\
&+ t(A-A_0)\wedge db + tb\wedge (F_A + F_{A_0}) + O(t^2) \\
\la Q_{A+tb}(\psi+t\phi) , \psi + t\phi\ra &= \la Q_A \psi, \psi\ra + \\
&+2t\,Re\la \phi, Q_A \psi\ra + t \la b\cdot \psi, \psi\ra + O(t^2).
\end{align*}
As a result, by Stoke's theorem, we have
\begin{align*}
    d_{(A,\psi)}\mathcal{L}^{RS}(b,\phi) &= \int_{Y} b \wedge F_A + \int_{Y} \la b\cdot \psi, \psi \ra + \int_{Y} \,Re\la \phi, Q_A \psi \ra
\end{align*}
Recall that $b$ can be viewed as a traceless skew-adjoint endomorphism of the spinor bundle, so it can be easily seen that $\la b\cdot \psi, \psi \ra = \la b, \gamma^{-1}(\mu(\psi))\ra$. Since $\star_3^2 = 1$, we can rewrite the above formula as
\begin{align*}
    d_{(A,\psi)}\mathcal{L}^{RS}(b,\phi) = \la b, \star_3 F_A + \gamma^{-1}(\mu(\psi)) \ra_{L^2} + Re\la\phi, Q_A \psi \ra_{L^2}.
\end{align*}
Thus, the gradient of $\mathcal{L}^{RS}$ at $(A,\psi)$ is given as
\[ grad\,\mathcal{L}^{RS}(A,\psi) = (\star_3 F_A + \gamma^{-1}(\mu(\psi)), Q_A \psi). \]
Since $\gamma(\star_3 \alpha) = -\gamma(\alpha)$ for a $1-$form $\alpha$, the critical points of $\mathcal{L}^{RS}$ correspond to the solutions of \eqref{eq:RSSW.1}. We summarize the above discussion in the following proposition.

\begin{Prop}\label{chern-simon-dirac}
The RS-SW equations \eqref{eq:RSSW.1} give the minimizing condition of the modified Chern-Simon-Dirac functional associated with the Rarita-Schwinger operator
\begin{equation*}
    \mathcal{L}^{RS}(A,\psi) = - \int_Y (A-A_0)\wedge (F_A + F_{A_0}) + \dfrac{1}{2}\int_{Y} \la Q_A \psi, \psi\ra.
\end{equation*}
\end{Prop}

\begin{Prop}\label{RSSW proposition 3.1}
Suppose $g$ is an Einstein metric on $Y$ such that it has non-negative scalar curvature. Let $(A,\psi)$ be a solution to \eqref{eq:RSSW.1}. If $A$ is flat, correspondingly, $c_1(\mathfrak{s}_{1/2})$ is torsion in $H^2(Y;\ZN)$, then $\psi$ must be a Rarita-Schwinger field.
\end{Prop}

\begin{proof}
    By Corollary \ref{RS clifford lemma 2.13} part one, since $A$ is flat, we have $\displaystyle Q_A P_A = \frac{1}{3}P_A D_A$. We then apply $Q_A$ to both sides of this identity and apply Corollary \ref{RS clifford lemma 2.13} part one again to obtain $Q_A^2 P_A = \frac{1}{3}Q_A P_A D_A = \frac{1}{9}P_A D_A^2= \frac{1}{9} P_A \left( \nabla^*_A \nabla_A + \frac{s}{4}\right)$.
    One can easily check that the Penrose operator commutes with the B\"ochner Laplacian $\nabla^*_A \nabla_A$. As a result, we see that $Q_A^2 = \displaystyle \frac{1}{9}\left(\nabla^*_A \nabla_A + \frac{s}{4}\right)$ on $im\, P_A$. Furthermore, by taking the adjoint of both sides of $\displaystyle Q_A P_A = \frac{1}{3}P_A D_A$, one also has $\displaystyle P_A^*Q_A = \frac{1}{3} D_A P_A^*$. Hence, a similar argument as in the proof of Lemma \ref{RS clifford lemma 2.10} will give us 
    $$Q_A : ker\, P_A^* \to ker\, P_A^*, \quad Q_A: im\, P_A \to im\, P_A.$$
    Thus if we decompose $\psi = \psi_1 \oplus \psi_2$, where $\psi_1 \in ker\, P^*_A$ and $\psi_2 \in im\, P_A$, then $Q_A \psi = 0 \Rightarrow Q_A \psi_1 = 0, \quad Q_A \psi_2 = 0$. We focus on the $\psi_2-$part. Because we also have $Q_A^2 \psi_2 = 0$, we are going to consider an $L^2-$paring with $\psi_2$ and apply the Weitzenb\"ock formula of $Q_A^2$ restricted to $im\, P_A$ to expand
    $$0 = \la Q_A^2\psi_2, \psi_2 \ra_{L^2} = \frac{1}{9}|\nabla_A \psi_2|^2 + \frac{s}{36} |\psi_2|^2 \geq \frac{s}{36} |\psi_2|^2.$$
    As a result, if $s \geq 0$, then $\psi_2 \equiv 0$. Therefore, $\psi = \psi_1$ must be divergence-free. In other words, $\psi$ is a Rarita-Schwinger field as claimed.
\end{proof}

\begin{Rem}
    Proposition \ref{RSSW proposition 3.1} should be considered to be the twisted version of Proposition $4.1$ in \cite{homma2019kernel}.
\end{Rem}

\begin{Prop} \label{RSSW proposition 2.16}
Suppose $g$ is an Einstein metric on $Y$ such that it has non-negative scalar curvature. Consider a $\text{spin}^c$ structure $\mathfrak{s}_{1/2}$ over $Y$ where $c_1({\mathfrak{s}_{1/2}})$ is torsion in $H^2(Y;\ZN)$. If $(A, \psi)$ is a solution of \eqref{eq:RSSW.1} where $A$ is flat, then $(A,\psi)$ corresponds to a solution of the following degenerate multiple-spinor Seiberg-Witten equations
\begin{equation*}
    \begin{cases}
        D^T_A \psi = 0,\\
        \mu(\psi) = 0, \quad \gamma(\psi)=0.
    \end{cases}
\end{equation*}
\end{Prop}

\begin{proof}
    By Proposition \ref{RSSW proposition 3.1}, we know that if $(A,\psi)$ is a solution and $A$ is flat, then $\psi$ is a Rarita-Schwinger field. This means that $D^T_A \psi = 0$ and $\gamma(\psi) = 0$. Since $A$ is flat, $F_A = 0$. As a result, we also have $\mu(\psi) = 0$. 
\end{proof}

Assuming Theorem \ref{first main theorem} and Corollary \ref{Corollary obstruction} are true, we observe the following consequence regarding the geometry of $Y$.

\begin{Cor}\label{geometry of y}
    Let $\mathfrak{s}_{1/2}$ be a $\text{spin}^c$ structure on $Y$ such that $c_1(\mathfrak{s}_{1/2})$ is torsion in $H^2(Y;\ZN)$. Let $g$ be a Riemannan metric on $Y$. If there is no $g-3/2-$Fueter section associated to the RS-SW equations corresponding to $g$ and $\mathfrak{s}_{1/2}$, then
    \begin{enumerate}
        \item The charged moduli space $\mathcal{M}_k(g)$ is compact for some charge $k$.
        \item Either $g$ is not an Einstein metric with non-negative scalar curvature or $(A,\psi)$ is never a solution where $A$ is a flat $U(1)-$connection on $det\,(\mathfrak{s}_{1/2})$.
    \end{enumerate}
\end{Cor}

\begin{proof}
    Suppose that $g$ is an Einstein metric with non-negative scalar curvature, and $(A,\psi)$ is a solution where $A$ is flat. Take a sequence of solutions $\{(A_n,\psi_n)\}$ where $A_n$ is flat. By Theorem \ref{first main theorem}, $(A_n, \psi_n)$ may converge weakly to a solution $(A_\infty,\psi_\infty)$ in $L^2_{1,loc}$ and $L^2_{2,loc}$, respectively. By definition, $(A_\infty,\psi_\infty)$ is a $g-3/2-$Fueter section. If that is not the case, then from Proposition \ref{RSSW proposition 2.16} we may take $(A,\psi)$ itself, which corresponds to a $g-3/2-$Fueter section.  
\end{proof}

\section{The first compactness theorem}
Recall that the RS-SW equations in dimension 3 read as follows
\begin{equation}\label{eq:first compactness.1}
    \begin{cases}
        Q_A\phi = 0,\\
        F_A = \mu(\phi).
    \end{cases}
\end{equation}
Here $\phi \in \Gamma(W_{\mathfrak{s}_{3/2}}\otimes \mathscr{L})$, $A$ is a $U(1)-$connection on the line bundle $\mathscr{L}\to Y$ and $\mu$ is the quadratic map defined on $\Gamma(T^{*}Y \otimes W_{\mathfrak{s}_{1/2}}\otimes \mathscr{L})= \Gamma(Hom\,(TY, W_{\mathfrak{s}_{1/2}}\otimes \mathscr{L}))$ given by $\mu(\phi) = \phi \phi^* - \dfrac{1}{2}|\phi|^2\cdot 1$. Modding out by the gauge group $\mathcal{G} = Maps(Y \to U(1))$, the moduli space of solutions of \eqref{eq:first compactness.1} could be non-compact. That is, there could be a sequence of solutions $\{(A_n, \phi_n)\}$ to \eqref{eq:first compactness.1} such that $\norm{\phi_n}_{L^4} \to \infty$. For this reason, we blow up our equation by introducing an extra unknown to the equations. Suppose $(A,\phi)$ is a solution to \eqref{eq:first compactness.1}, denote 
$$\psi = \epsilon\phi,\, \text{where } \, \epsilon = \frac{1}{\norm{\phi}_{L^4}}.$$
Note that then $\norm{\psi}_{L^4} = 1$, and $Q_A \psi = \epsilon Q_{A} \phi = 0$. Furthermore,
$$F_A = \mu(\phi) = \mu (\psi / \epsilon) \Rightarrow \epsilon^2 F_A = \mu(\psi).$$
As a result, with the new unknown $\epsilon \in (0,\infty)$, the equations \eqref{eq:first compactness.1} now read as following
\begin{equation}\label{eq:first compactness.2}
    \begin{cases}
        \norm{\psi}_{L^4} = 1,\\
        Q_A \psi = 0,\\
        \epsilon^2 F_A = \mu(\psi).
    \end{cases}
\end{equation}
Solutions of \eqref{eq:first compactness.2} are triples of the form $(A,\psi, \epsilon) \in \mathcal{A}(\mathscr{L}) \times \Gamma(W_{\mathfrak{s}_{3/2}} \otimes \mathscr{L}) \times (0, \infty)$. Similar to \eqref{eq:first compactness.1}, the equations \eqref{eq:first compactness.2} is $\mathcal{G}-$invariant. Of course, when we fix a referenced connection $A_0$ on $\mathscr{L}$, any other $U(1)$ connection can be written as $A = A_0 + a$, where $a \in i\Omega^1(Y)$. Then \eqref{eq:first compactness.2} can also be equivalently re-written as
\begin{equation}\label{eq:first compactness.3}
    \begin{cases}
        \norm{\psi}_{L^4} = 1\\
        Q_{A_0}\psi + \pi(a \cdot \psi) = 0\\
        \epsilon^2(F_{A_0} + da) = \mu(\psi).
    \end{cases}
\end{equation}
The first compactness theorem we will prove regarding the moduli space of solutions of \eqref{eq:first compactness.2} is the following.

\begin{Th}\label{first compactness theorem}
    Let $\{(A_n, \psi_n, \epsilon_n)\} \subset \mathcal{A}(\mathscr{L}) \times \Gamma(W_{\mathfrak{s}_{3/2}} \otimes \mathscr{L}) \times (0, \infty)$ be a sequence of solutions of \eqref{eq:first compactness.2}. There exists a constant $c$ depending only on the Riemannian metric $g$ on $Y$, $A_0$ such that if $\limsup \epsilon_n > c$, then after passing through a subsequence and up to gauge transformations $(A_n, \psi_n, \epsilon_n)$ converges to a solution $(A,\psi, \epsilon)$ in the $C^\infty$ topology.
\end{Th}

We begin the proof of Theorem \ref{first compactness theorem}. Firstly, recall that we have the Weitzenb\"ock formula for the twisted Rarita-Schwinger operator.

\begin{Lemma}[\textbf{Weitzenb\"ock formula of $Q_A$}]\label{weitzenbock formula}
    For all $(A,\psi) \in \mathcal{A}(\mathscr{L}) \times \Gamma(W_{\mathfrak{s}_{3/2}} \otimes \mathscr{L})$, we have
    $$Q^2_A \psi + \dfrac{4}{3} P_A P_A^* \psi = \nabla^*_A \nabla_A \psi + \dfrac{s}{4} \psi + \pi(F_A \psi) - \pi(1 \otimes Ric) \psi,$$
    where $P_A = \pi \circ \nabla_A$ is the  twisted Penrose operator, $s$ is the scalar curvature of the Riemannian metric $g$ on $Y$, and $Ric$ is the associated Ricci curvature. \qed
\end{Lemma}
Since $P_A^* = (\pi \circ \nabla_A)^* = \nabla_A^* \circ \pi^*$, for any $\psi \in \Gamma(W_{\mathfrak{s}_{3/2}} \otimes \mathscr{L})$, we have $\pi^* \psi = \pi(\psi) = \psi$. As a result, $P_A^* \psi = \nabla^*_A \psi$. With respect to the referenced connection $A_0$, we re-write the divergence of $\psi$ as $\nabla^*_A \psi = \nabla^*_{A_0}\psi + a^*\psi$.

\begin{Lemma}\label{first compactness lemma 1.3}
    For any $(a, \psi) \in i\Omega^1(Y) \times \Gamma(W_{\mathfrak{s}_{3/2}}\otimes \mathscr{L})$, we have $\norm{a^*\psi}_{L^2} \leq \sqrt{3}\norm{a \cdot \psi}_{L^2}$.
\end{Lemma}

\begin{proof}
    Note that $a^*\psi$ is defined as following: Since $\psi$ can be viewed as a spinor-valued $1-$form, in local frame $\{e_i\}_{i=1}^{3}$, it can be written as $\psi = \sum_{i=1}^{3}\psi_i \otimes e^i$, where $\psi_i \in \Gamma(W_{\mathfrak{s}_{1/2}}\otimes \mathscr{L})$. Then simply, $a^* \psi = \sum_{i=1}^{3} a(e_i)\psi_i$. By triangle inequality and the Cauchy-Schwarz inequality, we have
    $$|a^* \psi|^2 \leq \left( \sum_{i=1}^{3} |a(e_i)\psi_i| \right)^2 \leq 3 \sum_{i=1}^{3} |a(e_i)|^2|\psi_i|^2 \leq 3 |a|^2 |\psi|^2.$$
    At the same time, recall that the Clifford multiplication by $a \in i\Omega^1(Y)$ on $\Gamma(T^* Y \otimes W_{\mathfrak{s}_{1/2}} \otimes \mathscr{L})$ gives us a skew-adjoint linear map. Thus,
    $$|a\cdot \psi|^2 = \la a \cdot \psi, a \cdot \psi\ra = - \la a^2 \cdot \psi, \psi \ra= |a|^2 |\psi|^2.$$
    As a result, we have 
    $$\norm{a^*\psi}^2_{L^2} = \int_{Y} |a^*\psi|^2 \leq 3 \int_{Y} |a \cdot \psi|^2 = 3\norm{a \cdot \psi}^2_{L^2}.$$
    After taking the square root, we have the desired estimate.
\end{proof}

The following $U(1)-$gauge fixing lemma is well-known.

\begin{Lemma}[\textbf{Gauge fixing lemma for $U(1)-$bundle}]\label{first compactness lemma 1.4}
    Suppose $\mathscr{L}$ is a complex line bundle over $Y$ equipped with a Hermitian metric. Fix a $C^{\infty}$ unitary connection $A_0$ on $\mathscr{L}$. For any $k\geq 0$, there exists constants $c_1(g,k), c_2(g, A_0, k)$ such that for $L^2_k$ unitary connection $A$ on $\mathscr{L}$, there is an $h \in \mathcal{G}_{L^2_{k+1}}$ such that $h \cdot A = A_0 + a$, where $a \in L^2_k(iT^* Y)$ and 
    $$d^* a = 0,\, \, \, \norm{a}^2_{L^2_k} \leq c_1(g,k) \norm{F_A}^2_{L^2_{k-1}} + c_2(g, A_0, k).$$ 
\end{Lemma}

\begin{Lemma}\label{first compactness lemma 1.5}
    For any $(A, \psi, \epsilon)$ that solves \eqref{eq:first compactness.3}, there are  constants $c_3, c_4$ depending only the Riemannian metric $g$ of $Y$ and the fixed reference connection $A_0$ such that 
    $$\int_{Y} |P^*_A \psi|^2 \leq c_3 \int_{Y} \epsilon^{-4} |\mu(\psi)|^2 + c_4.$$
\end{Lemma}

\begin{proof}
    Firstly, based on the above observation and the Cauchy-Schwarz inequality, we have
    \begin{align}\label{eq:first compactness.4}
        \int_{Y} |P^*_A \psi|^2 &= \int_{Y} |\nabla^*_A \psi|^2 \leq \int_{Y} (|\nabla^*_{A_0}\psi| + |a^*\psi|)^2 \leq 2 \int_Y (|\nabla^*_{A_0} \psi|^2 + | a^* \psi|^2 ).
    \end{align}
    Since $\nabla_{A_0}$ is a first-order differential operator, $\nabla_{A_0}^*$ is also a first-order differential operator which extends to a bounded operator from $L^2_1 (W_{\mathfrak{s}_{3/2}} \otimes \mathscr{L}) \to L^2(W_{\mathfrak{s}_{1/2}} \otimes \mathscr{L})$. Thus, there is a constant $c(g, A_0)\geq 0$ such that 
    \begin{align}\label{eq:first compactness.5}
        \norm{\nabla^*_{A_0} \psi}^2_{L^2} \leq c(g, A_0) \norm{\psi}^2_{L^2_1}.
    \end{align}
    On the other hand, by the elliptic estimate and the Cauchy-Schwarz inequality, there is a constant $c'(g, A_0)$ such that
    \begin{align}
        \norm{\psi}^2_{L^2_1} &\leq c'(g, A_0) ( \norm{Q_{A_0}\psi}^2_{L^2} + \norm{\psi}^2_{L^2})\nonumber \\
        & \leq c'(g, A_0)( \norm{\pi(a\cdot \psi)}^2_{L^2}+vol^{1/2}_g(Y)\norm{\psi}^2_{L^4}) \label{eq:first compactness.6}
    \end{align}
    Since $\pi$ is an orthogonal projection, $\norm{\pi(a\cdot \psi)}_{L^2} \leq \norm{a \cdot \psi}_{L^2}$. Hence, \eqref{eq:first compactness.6} can be further estimated by
    \begin{align} \label{eq:first compactness.7}
        \norm{\psi}^2_{L^2_1}\leq c'(g, A_0)( \norm{a\cdot \psi}^2_{L^2} + 1).
    \end{align}
    Combine \eqref{eq:first compactness.5} and \eqref{eq:first compactness.7}, there is a non-negative constant which we shall denote by $c$ again that depends only on $g$ and $A_0$ satisfying
    \begin{equation}\label{eq:first compactness.8}
        \norm{\nabla^*_{A_0}\psi}^2_{L^2}\leq c\norm{a \cdot \psi}^2_{L^2} + c.
    \end{equation}
    On the other hand, by Lemma \ref{first compactness lemma 1.3}, we already have $\norm{a^*\psi}^2_{L^2} \leq 3 \norm{a \cdot \psi}^2_{L^2}$. Thus, combine \eqref{eq:first compactness.4} , \eqref{eq:first compactness.8}, and Lemma \ref{first compactness lemma 1.3}, we have the following estimate
    \begin{equation}\label{eq:first compactness.9}
        \int_{Y}|P^*_A \psi|^2 \leq 2(c+3)\norm{a \cdot \psi}^2_{L^2}+ 2c
    \end{equation}
    Now, by the Sobolev multiplication theorem, there is a constant $c''(g,A_0) \geq 0$ such that $\norm{a\cdot \psi}^2_{L^2} \leq c''(g,A_0) \norm{a}^{2}_{L^2_1} \norm{\psi}^2_{L^4} = c''(g,A_0) \norm{a}^{2}_{L^2_1}$. Then the gauge-fixing lemma (Lemma \ref{first compactness lemma 1.4}) tells us that up to a gauge transformation, we have
    \begin{align}
        \norm{a\cdot \psi}^2_{L^2} &\leq c''c_1(g)\norm{F_A}^2_{L^2} + c''c_2(g,A_0)\nonumber \\
        & = c''c_1(g) \int_{Y} \epsilon^{-4}|\mu(\psi)|^2 + c''c_2(g,A_0). \label{eq:first compactness.10}
    \end{align}
    Combine \eqref{eq:first compactness.9} and \eqref{eq:first compactness.10}, there exists non-negative constants $c_3$ and $c_4$ depending only $g$ and $A_0$ such that
    we have the desired estimate for $\norm{P^*_A\psi}^2_{L^2}$ as claimed.
\end{proof}

\begin{Lemma}\label{first compactness lemma 1.6}
    Suppose $(A, \psi, \epsilon) \in \mathcal{A}(\mathscr{L}) \times \Gamma(W_{\mathfrak{s}_{3/2}}\otimes \mathscr{L}) \times (0,\infty)$ is a solution of \eqref{eq:first compactness.2} (or \eqref{eq:first compactness.3}). Let $f$ be any smooth function on $Y$, and $U$ is a closed subset of $Y$ with a smooth boundary. Denote $\nu$ by the outward pointing normal vector field of $U$. Then we have
    \begin{align}
        \int_U \Delta f \cdot |\psi|^2 &+ f \left( \dfrac{-8}{3}|P^*_A\psi|^2 + \dfrac{s}{2}|\psi|^2 + 2 \epsilon^{-2}|\mu(\psi)|^2 - 2\la (1\otimes Ric)\psi, \psi \ra + 2 |\nabla_A \psi|^2\right) \nonumber \\
        &= \int_{\partial U} f\cdot \partial_{\nu}|\psi|^2 - \partial_{\nu}f \cdot |\psi|^2.\nonumber
    \end{align}
\end{Lemma}
\begin{proof}
    Since for any $f_1, f_2$ smooth function on $Y$, we already have 
    $$\int_{U} \Delta f_1 \cdot f_2 - f_1 \cdot \Delta f_2 = \int_{\partial U} f_1 \cdot \partial_{\nu} f_2 - \partial_{\nu}f_1 \cdot f_2.$$
    The identity in the lemma will follow as soon as we apply the above for $f_1:= f, f_2= |\psi|^2,$ and compute $\Delta |\psi|^2$. Recall that for each $x \in Y$, we have the following point-wise formula
    \begin{equation}
        \dfrac{1}{2} \Delta |\psi|^2 = \la \nabla^*_{A}\nabla_A \psi, \psi\ra - |\nabla_A \psi|^2. \label{eq:first compactness.11}
    \end{equation}
    By the Weitzenb\"ock formula of the twisted Rarita-Schwinger operator (Lemma \ref{weitzenbock formula}), we can expand the RHS of \eqref{eq:first compactness.11} further as
    \begin{align}
        \dfrac{1}{2}\Delta |\psi|^2 & = \left\la Q_A^2 \psi + \dfrac{4}{3}P_A P^*_A \psi - \dfrac{s}{4}\psi -\pi(F_A \psi) + \pi(1\otimes Ric)\psi, \psi \right \ra - |\nabla_A \psi|^2 \nonumber \\
        & = \dfrac{4}{3}\la P_A P^*_A \psi, \psi \ra -\dfrac{s}{4}|\psi|^2 -\la\pi(F_A \psi),\psi\ra +\la \pi(1\otimes Ric)\psi, \psi\ra - |\nabla_A \psi|^2. \label{eq:first compactness.12}
    \end{align}
    Note that for any $\psi_1 \in \Gamma(T^* Y \otimes W_{\mathfrak{s}_{1/2}}\otimes \mathscr{L})$ and $\psi_2 \in \Gamma(W_{\mathfrak{s}_{3/2}}\otimes \mathscr{L})$, point-wise we have $\la \pi(\psi_1), \psi_2\ra = \la \psi_1, \pi(\psi_2)\ra = \la \psi_1, \psi_2\ra$. Furthermore, $\la \mu(\psi)\psi, \psi\ra = |\mu(\psi)|^2$ for any twisted $3/2-$spinor $\psi$. Hence, we can re-write \eqref{eq:first compactness.12} as following
    \begin{equation} \label{eq:first compactness.13}
        -\Delta |\psi|^2 = -\dfrac{8}{3}\la P_A P^*_A \psi, \psi \ra + \dfrac{s}{2}|\psi|^2 +2\epsilon^{-2}|\mu(\psi)|^2 - 2\la (1\otimes Ric)\psi, \psi\ra + 2|\nabla_A \psi|^2.
    \end{equation}
    Multiply both sides of \eqref{eq:first compactness.13} by $f$ and integrate over $U$, we obtain
    \begin{align}
        \int_{U}-f\cdot\Delta|\psi|^2 
        & = \int_{U} f\left( -\frac{8}{3}|P^*_A \psi|^2 + \dfrac{s}{2}|\psi|^2 +2\epsilon^{-2}|\mu(\psi)|^2 - 2\la (1\otimes Ric)\psi, \psi\ra + 2|\nabla_A \psi|^2\right)\nonumber
    \end{align}
    The last equality is true as long as $\psi$ is compactly supported in $U$. As a result, we have the desired identity.
\end{proof}
\noindent
\textit{Proof of Theorem \ref{first compactness theorem}.} Firstly, we apply Lemma \ref{first compactness lemma 1.6} when $U:= Y$ and $f\equiv 1$ to obtain the following
$$\int_Y \left( -\frac{8}{3}|P^*_A \psi|^2 + \dfrac{s}{2}|\psi|^2 +2\epsilon^{-2}|\mu(\psi)|^2 - 2\la (1\otimes Ric)\psi, \psi\ra + 2|\nabla_A \psi|^2\right) = 0.$$
This immediately implies that
\begin{equation}
    \int_{Y} |\nabla_A \psi|^2 = \dfrac{4}{3}\int_Y |P^*_A\psi|^2 + \int_{Y} -\dfrac{s}{4}|\psi|^2-\int_{Y}\epsilon^{-2}|\mu(\psi)|^2+\int_{Y} \la(1\otimes Ric)\psi,\psi\ra. \label{eq:first compactness.14}
\end{equation}
Now, by the Cauchy-Schwarz inequality, there is a non-negative constant $c_5(g)$ depending only on $g$ such that
$$\int_{Y} \la (1\otimes Ric)\psi, \psi\ra \leq \int_{Y} c_5(g) |\psi|^2 \leq c_5(g)\cdot vol_g^{1/2}(Y) \norm{\psi}^2_{L^4}.$$
Of course, it is not hard to see that
$$\int_{Y} -\dfrac{s}{4} |\psi|^2 \leq \int_{Y} \max_{Y} \dfrac{|s|}{4} |\psi|^2 \leq \max_{Y} \frac{|s|}{4}\cdot vol_g^{1/2}(Y)\norm{\psi}^2_{L^4}.$$
The two estimates above combined with Lemma \ref{first compactness lemma 1.5}, we estimate \eqref{eq:first compactness.14} further by
\begin{align}
    \int_{Y} |\nabla_A \psi|^2 &\leq \dfrac{4}{3}\left(c_3\int_{Y} \epsilon^{-4}|\mu(\psi)|^2 + c_4 \right) - \int_{Y} \epsilon^{-2}|\mu(\psi)|^2 + \int_{Y} \left(c_5(g) + \max_{Y} \dfrac{|s|}{4}\right) |\psi|^2\nonumber \\
    & = \int_{Y} \left(\dfrac{4c_3}{3}\epsilon^{-4}- \epsilon^{-2}\right)|\mu(\psi)|^2 + \int_{Y} \left(c_5(g) + \max_{Y} \dfrac{|s|}{4}\right) |\psi|^2 + \dfrac{4c_3c_4}{3}\,vol_g(Y) \label{eq:first compactness.15}
\end{align}
Since $\limsup_{\epsilon \to \infty} (4c_3 \epsilon^{-4}/3 - \epsilon^{-2}) = 0$, there is a positive constant $c$ such that if $\epsilon > c >0$, then $4c_3\epsilon^{-4}/3 - \epsilon^{-2} \leq 0$. As a result with all such $\epsilon$, we estimate the RHS of \eqref{eq:first compactness.15} by
$$\int_{Y} |\nabla_A \psi|^2 \leq \int_{Y} \left(c_5(g) + \max_{Y} \dfrac{|s|}{4}\right) |\psi|^2 + \dfrac{4c_3c_4}{3}\,vol_g(Y) =O(1). $$
Apply Kato's inequality and the Sobolev embedding theorem for $L^2_1 \hookrightarrow L^6$, we have $\norm{\psi}_{L^6} = O(1)$.

We still assume the condition that $\epsilon > c >0$. Next, we apply Lemma \ref{first compactness lemma 1.6} when $U := Y\setminus B(x,\delta)$ and $f := G(x,y)$, which is the positive Green's function of the invertible operator $\Delta + 1$. Passing to the limit as $\delta \to 0$ and apply Lemma \ref{first compactness lemma 1.5}, we have 
\begin{align}
    \dfrac{1}{2}|\psi|^2(x) &+ \int_{Y} G(x, \bullet)(\epsilon^{-2}|\mu(x)|^2+|\nabla_A \psi|^2) \lesssim \int_{Y}
G(x,\bullet)|\psi|^2 + \int_{Y} G(x, \bullet)|P^*_A \psi|^2\nonumber \\
& \lesssim \int_{Y} G(x,\bullet)|\psi|^2 + \int_{Y}  G(x,\bullet)c_3\epsilon^{-4}|\mu(\psi)|^2 + \max_{Y}G \cdot c_4 \, vol_g(Y). \label{eq:first compactness.16}
\end{align}
As a result of re-arranging \eqref{eq:first compactness.16}, we obtain
$$\dfrac{1}{2}|\psi|^2(x)+ \int_{Y}G(x,\bullet)(\epsilon^{-2}-c_3\epsilon^{-4})|\mu(\psi)|^2 + \int_{Y} G(x,\bullet)|\nabla_A \psi|^2 \lesssim  \int_{Y} G(x,\bullet)|\psi|^2.$$
Since $\norm{\psi}_{L^6} = O(1)$, the RHS of the above estimate is also $O(1)$. Therefore, when taking supremum of the LHS overall $x \in Y$, we obtain
\begin{align}
\norm{\psi}_{L^{\infty}} &= O(1), \label{eq:first compactness.17}\\
\norm{\nabla_{A}\psi}_{L^2(B(x,r))} &= O(r^{1/2}), \label{eq:first compactness.18}\\
\norm{\mu(\psi)}_{L^2(B(x,r))} &= O(r^{1/2}(\epsilon^{-2}-c_3\epsilon^{-4})^{-1/2}). \label{eq:first compactness.19}
\end{align}
Because of the above estimates, if $\{(A_n,\psi_n,\epsilon_n)\}$ be a sequence of solutions of \eqref{eq:first compactness.2} with $\limsup \epsilon_n > c > 0$ for some universal constant c, then the standard elliptic boot-strapping technique will give us a sub-sequence $(A_n, \psi_n, \epsilon_n)$ up to some gauge transformations that converges to a solution $(A,\psi, \epsilon)$ in $C^{\infty}-$sense as desired. \qed

\section{Curvature controls the $3/2-$spinors}
From the previous section, we already see that if $\limsup \epsilon_n$ is bounded away from zero by a universal positive constant, then the sequence of solutions $\{(A_n, \psi_n, \epsilon_n)\}$ of \eqref{eq:first compactness.2} convergers to a $C^{\infty}$ solution up to a subsequence and gauge transformations. Now we deal with the case when $\limsup \epsilon_n = 0$. The main idea is to use the \textit{frequency function method} to analyze the convergence behavior of a sequence of solutions of such type. To put ourselves in the set-up where we could apply this method, we need to establish some more local estimates.  Note that the estimates \eqref{eq:first compactness.17},\eqref{eq:first compactness.18},\eqref{eq:first compactness.19} do not hold anymore. However, one can obtain similar estimates for connections whose curvature densities are controlled by a small constant on a small geodesic ball of $Y$. All $3/2-$spinors in this section are considered to have compact support inside the interior of a geodesic ball around a point unless otherwise stated.

We recall Uhlenbeck's local slice theorem for connections with $L^p-$bounds on curvatures. 

\begin{Th}[\textbf{Local slice theorem}]\label{local slice theorem}\cite{MR0648356}
Let $B(x,r)$ be a small geodesic ball in $Y$ where $x$ is any point on $Y$ and $r>0$. There are constants $c, k_0 >0$ (possibly depending on $x,r,g$) such that for any connection $A= A_0 + a$ and $\norm{F_A}^2_{L^2(B(x,r))} \leq k_0$, there exists a gauge transformation $h \in L^2_2(B(x,r), U(1))$ such that for $b = a + hdh^{-1}$, we have
$$d^*b = 0, \, \, \star b|_{\partial B(x,r)} = 0, \, \, \, \norm{b}_{L^2_1(B(x,r))}^2 \leq c\norm{F_A}_{L^2(B(x,r))}^2.$$
Here $A_0$ is the connection from the trivialization of $\mathscr{L}$ on $B(x,r)$ and $a$ is an $L^2_1$ purely imaginary valued $1$-form on $B(x,r)$. 
\end{Th}

With the local slice theorem (Theorem \ref{local slice theorem}), on a small geodesic ball $B(x,r)$, various estimates established in the previous section still hold (e.g, \eqref{eq:first compactness.9}, \eqref{eq:first compactness.10}, etc.). Hence, for any $(A, \psi, \epsilon) \in V_{k_0}(B(x,r)) \times \Gamma(W_{\mathfrak{s}_{3/2}}\otimes \mathscr{L}|_{B(x,r)}) \times (0,\infty)$ that is a solution of \eqref{eq:first compactness.2}, we have
$$\int_{B(x,r)} |P^*_A \psi|^2 \leq c\left(\int_{B(x,r)} |F_A|^2 + 1\right) \leq c(k_0 +1).$$
Then a similar argument using the integration by parts formula (cf. Lemma \ref{first compactness lemma 1.6}) for $f \equiv 1$ and $U = B(x,r)$ yields for us 
$$\int_{B(x,r)} |\nabla_A \psi|^2 \leq \dfrac{4}{3}c(k_0+1) + c' \norm{\psi}^2_{L^4} = O(1).$$
As a result, by Kato's inequality and the fact that $L^2_1 \hookrightarrow L^6$ is a compact embedding, $\norm{\psi}_{L^6(B(x,r))} = O(1)$.

Again, we apply the integration by parts formula (cf. Lemma \ref{first compactness lemma 1.6}) for $f := G(x,y)-$the positive Green's function of $\Delta + 1$ and $U:= B(x,r)\setminus B(x,\delta)$ where $0\leq \delta \leq r$ and let $\delta \to 0$ to obtain
\begin{align}
\norm{\psi}_{L^{\infty}(B(x,r))} &= O(1), \label{eq:curvature controls.20}\\
\norm{\nabla_{A}\psi}_{L^2(B(x,r))} &= O(r^{1/2}), \label{eq:curvature controls.21}\\
\norm{\mu(\psi)}_{L^2(B(x,r))} &= O(r^{1/2}\epsilon). \label{eq:curvature controls.22}
\end{align}
We summarize the above discussion in the form of the following lemma.

\begin{Lemma}\label{curvature controls lemma 2.3}
     For any solution $(A,\psi,\epsilon) \in V_{k_0}(B(x,r)) \times \Gamma(W_{\mathfrak{s}_{3/2}}\otimes \mathscr{L}|_{B(x,r)}) \times (0,\infty)$ of \eqref{eq:first compactness.2}, we have 
    $$\norm{\psi}_{L^\infty (B(x,r))} = O(1), \, \, \, \norm{\nabla_A \psi}_{L^2(B(x,r))} = O(r^{1/2}), \, \, \, \norm{\mu(\psi)}_{L^2(B(x,r))} = O(r^{1/2}\epsilon).$$
    Here $B(x,r)$ is a small geodesic ball on $Y$, and $k_0$ is a constant as in Theorem \ref{local slice theorem}. \qed
\end{Lemma}

Next, we control $\norm{\nabla_A^2 \psi}_{L^2}$ in terms of the critical radius of a connection $A \in V_{k_0}(B(x,r))$. For us to achieve this, we need to establish some more local estimates. Firstly, we recall the following Weitzenb\"ock-type formula that relates $P_A, Q_A,$ and $D_A$ on a three-manifold. 

\begin{Lemma}\label{curvature controls lemma 2.4}
    For any $(A,\psi) \in \mathcal{A}(\mathscr{L}) \times \Gamma(W_{\mathfrak{s}_{3/2}}\otimes \mathscr{L})$, we have
    $$-\dfrac{1}{3}P_A D_A\psi + Q_A P_A\psi = \dfrac{1}{2}\left(Ric - \frac{s}{3}\right)\psi + \pi\pi_2 F_A\psi.$$
    Here $\pi_2$ is the map that takes spinor-valued $2-$forms to spinor-valued $1-$forms. \qed
\end{Lemma}

\begin{Lemma}\label{curvature controls lemma 2.5}
    For any $(A,\psi, \epsilon) \in V_{k_0}(B(x,r)) \times \Gamma(W_{\mathfrak{s}_{3/2}}\otimes \mathscr{L}|_{B(x,r)}) \times (0,\infty)$ that solves \eqref{eq:first compactness.2}, we have
    $$\norm{P^*_A \psi}^2_{L^2_1(B(x,r))} \leq c\left(  \norm{\psi}^2_{L^2(B(x,r))} + \int_{B(x,r)} |F_A|^2 \cdot |\psi|^2 + k_0 \sqrt{k_0^2+1} + k_0 + 1 \right),$$
    where $B(x,r)$ is a small geodesic ball on $Y$ and $k_0$ is as in Theorem \ref{local slice theorem}.
\end{Lemma}

\begin{proof}
    We take the adjoint of the Weitzenb\"ock-type formula in Lemma \ref{curvature controls lemma 2.4} to obtain
    $$-\dfrac{1}{3}D_A P_A^* \psi = \frac{1}{2}\left(Ric-\frac{s}{3}\right)^* \psi + F_A \pi_2^* \pi(\psi) = \frac{1}{2}\left(Ric - \frac{s}{3}\right)^* \psi + F_A \pi_2^* \psi.$$
    If we write $A = A_0 + a$, then we can re-arrange the above identity as
    \begin{equation}\label{eq:curvature controls.23}
        D_{A_0}P^*_A \psi = -a \cdot P^*_A \psi - \frac{3}{2}\left(Ric - \frac{s}{3} \right)^* \psi - 3 F_A \pi^* \psi. 
    \end{equation}
    Combine \refeq{eq:curvature controls.23} with the elliptic estimate and the Cauchy-Schwarz inequality, we have
    \begin{align} \label{eq:curvature controls.24}
        &\norm{P^*_A \psi}^2_{L^2_1(B(x,r))}  \leq c\left(\norm{D_{A_0}P^*_A \psi}^2_{L^2(B(x,r))}+ \norm{P^*_A \psi}^2_{L^2(B(x,r))}\right) \nonumber \\
        &\leq c \left( \norm{a \cdot P^*_A \psi}^2_{L^2} + \int_{B(x,r)}\left | \left(Ric - \frac{s}{3}\right)^*\psi \right|^2 + \int_{B(x,r)} |F_A \pi_2^* \psi|^2 + k_0 + 1 \right).
    \end{align}
    It is not hard to see that 
    \begin{align}
        \int_{B(x,r)}\left | \left(Ric - \frac{s}{3}\right)^*\psi \right|^2 &\leq c_1 \int_{B(x,r)} |\psi|^2 \nonumber \\
        \int_{B(x,r)} |F_A \pi^*_2 \psi|^2 &\leq c_2 \int_{B(x,r)} |F_A|^2 \cdot |\psi|^2.\nonumber
    \end{align}
    Thus, \eqref{eq:curvature controls.24} can be further estimated by
    \begin{equation} \label{eq:curvature controls.25}
        \norm{P^*_A \psi}^2_{L^2_1} \leq c \left(\norm{\psi}^2_{L^2(B(x,r))} + \int_{B(x,r)}|F_A|^2  |\psi|^2 + \norm{a \cdot P^*_A \psi}^2_{L^2} + k_0 + 1 \right).
    \end{equation}
    It remains for us to control $\norm{a \cdot P^*_A \psi}^2_{L^2(B(x,r))}$. The Sobolev multiplication theorem gives us 
    $$\norm{a \cdot P^*_A \psi}^2_{L^2} \leq c_3 \norm{a}^2_{L^2_1} \norm{P^*_A \psi}^2_{L^4}.$$ 
    And the local slice theorem (Theorem \ref{local slice theorem}) already tells us that 
    $$\norm{a}^2_{L^2_1} \leq c_4 \int_{B(x,r)} |F_A|^2 \leq c_4k_0.$$
    Thus, we further control $\norm{a \cdot P^*_A \psi}^2_{L^2(B(x,r))}$ by controlling $\norm{P^*_A \psi}_{L^4}$. To achieve this, firstly, by applying the Minkowski inequality, we have
    \begin{equation} \label{eq:curvature controls.26}
        \int_{B(x,r)} |P^*_A \psi|^4 \leq \int_{B(x,r)} |\nabla^*_{A_0}\psi + a \cdot \psi|^4 \leq 2^3 \int_{B(x,r)}(|\nabla^*_{A_0}\psi|^4 + |a\cdot \psi|^4).
    \end{equation}
    Since $\nabla^*_{A_0}$ is a first-order differential operator, it extends to a bounded operator from $L^4_1 \to L^4$. Hence, the elliptic estimate obtains for us
    \begin{align}
        \norm{\nabla^*_{A_0}\psi}^4_{L^4} & \leq c_5 \norm{\psi}^4_{L^4_1(B(x,r))} \nonumber \\
        & \leq c_6 (\norm{Q_{A_0}\psi}^4_{L^4}+\norm{\psi}^4_{L^4})\leq c_7 (\norm{Q_{A_0}\psi}^4_{L^4}+1). \nonumber
    \end{align}
    Since $(A,\psi, \epsilon)$ is a solution of \eqref{eq:first compactness.2}, the above estimate implies $\norm{\nabla^*_{A_0}\psi}^4_{L^4} \leq c_8 (\norm{a\cdot \psi}^4_{L^4} + 1)$. Hence, combine with \eqref{eq:curvature controls.26}, we get that
    \begin{equation} \label{eq:curvature controls.27}
        \int_{B(x,r)} |P^*_{A}\psi|^4 \leq c_9\left(\int_{B(x,r)} |a \cdot \psi|^4 + 1 \right). 
    \end{equation}
    By the Sobolev multiplication theorem, for any $p>12$, we have
    \begin{align} \label{eq:curvature controls.28}
        \norm{a\cdot \psi}^4_{L^4} &\leq c_{10} \norm{a}^4_{L^2_1} \norm {\psi}^4_{L^p} \nonumber \\
        & \leq c_{11} \norm{a}^4_{L^2_1}\cdot \norm{\psi}^4_{L^\infty (B(x,r))} \leq c_{12} \norm{F_A}^4_{L^2(B(x,r))}.
    \end{align}
    The last estimate of \eqref{eq:curvature controls.28} is given by Uhlenbeck's local slice theorem (Theorem \ref{local slice theorem}) and Lemma \ref{curvature controls lemma 2.3} (cf. \eqref{eq:curvature controls.20}). We use \eqref{eq:curvature controls.28} to further estimate the RHS of \eqref{eq:curvature controls.27} as following
    \begin{equation} \label{eq:curvature controls.29}
        \int_{B(x,r)}|P^*_A \psi|^4 \leq c_{13}\left(\norm{F_A}^4_{L^2(B(x,r))} + 1 \right) = c_{14}(k_0^2 + 1).
    \end{equation}
    Combine \eqref{eq:curvature controls.29} and the Sobolev multiplication estimate for $\norm{a \cdot P^*_A \psi}^2_{L^2}$ earlier, we have $\norm{a \cdot P^*_{A} \psi}^2_{L^2}$ $\leq c_3c_4c_{15} k_0 \sqrt{k_0^2 + 1}$. Apply this estimate to the RHS of \eqref{eq:curvature controls.25}, we immediately obtain our desired estimate.
\end{proof}

We need one more useful identity before getting to the main point of this section, which is the curvature-control of $\norm{\nabla^2_{A} \psi}_{L^2_{loc}}$. 

\begin{Lemma}\label{curvature controls lemma 2.6}
    Let $\psi \in \Gamma(W_{\mathfrak{s}_{3/2}}\otimes \mathscr{L})$ such that $Q_A \psi = 0$ and and arbitrary $f \in C^{\infty}(Y)$. Then one has the following
    $$\nabla^*_A \nabla_A (f\psi) = -\frac{s}{4} f\psi - \pi(F_A(f\psi)) + \pi(1\otimes Ric) f\psi + \frac{4}{3}f P_A P^*_A \psi - 2\pi \nabla^A_{\nabla f} \psi + (\Delta f) \psi.$$
\end{Lemma}

\begin{proof}
    Firstly, it is not difficult to see that 
    \begin{align}
        D^{T \, 2}_A (f\psi) &= D^T_A (D^T_A(f\psi)) = D^T_A (fD^T_A \psi + df \cdot \psi) \nonumber \\
        &= fD^{T\, 2}_{A} \psi + df\cdot D^T_A \psi - df \cdot D^T_{A} \psi - 2\nabla^A_{\nabla f} \psi + (\Delta f) \psi \nonumber \\
        &= f D^{T\, 2}_A \psi - 2\nabla^A_{\nabla f}\psi + (\Delta f) \psi. \nonumber
    \end{align}
    Thus, when applying the orthogonal projection $\pi$ onto the $3/2-$spinor bundle to both sides of the above identity, we have
    $$Q^2_A (f\psi) + \frac{4}{3}P_A P^*_A (f\psi) = f\left(Q^2_A \psi + \frac{4}{3}P_A P^*_A \psi\right) - 2\pi \nabla^A_{\nabla f} \psi + (\Delta f) \psi.$$
    Remember that $Q_A \psi = 0$ so that we can re-arrange the above formula as
    \begin{equation} \label{eq:curvature controls.30}
    Q^2_A (f\psi) = -\frac{4}{3} P_A P^*_A (f\psi) + \frac{4}{3} f P_A P^*_A \psi - 2\pi \nabla^A_{\nabla f} \psi + (\Delta f)\psi.
    \end{equation}
    On the other hand, the Weitzenb\"ock formula for $Q^2_A$ (Lemma \ref{weitzenbock formula}) gives us
    \begin{equation} \label{eq:curvature controls.31}
        Q^2_A (f\psi) = -\frac{4}{3}P_A P^*_A (f\psi) + \nabla^*_A \nabla_A (f\psi) + \frac{s}{4}f\psi + \pi(F_A (f\psi)) - \pi(1\otimes Ric)f\psi.
    \end{equation}
    Compare \eqref{eq:curvature controls.30} and \eqref{eq:curvature controls.31}, we have our formula.
\end{proof}

We are now ready to state and prove the main result of this section. 

\begin{Prop}\label{curvature controls proposition 2.7}
    Suppose $(A,\psi, \epsilon) \in V_{k_0}(B(x,r)) \times \Gamma(W_{\mathfrak{s}_{3/2}}\otimes \mathscr{L}|_{B(x,r)}) \times (0,\infty)$. For any $\delta \in (0,1]$, we have
    \begin{align}
    r^{1/2}\norm{\nabla^2_A \psi}_{L^2(B(x,r(1-\delta)))} \lesssim_{\delta}\,\,\, & r^{-3/2}\norm{\psi}_{L^2(B(x,r))} + r^{-1/2}\norm{\nabla_A \psi}_{L^2(B(x,r))} \nonumber\\
    & + r^{1/2}\norm{F_A}_{L^2(B(x,r))}\norm{\psi}_{L^\infty (B(x,r))} + O(1).\nonumber
    \end{align}
    Here, $B(x,r)$ is a small geodesic ball around $x$ on $Y$, and $k_0$ is a constant as in Theorem \ref{local slice theorem}.
\end{Prop}

\begin{proof}
    Since the statement is scale-invariant, without loss of generality, we shall work on the geodesic ball $B(x,1)$ of radius $1$. Let $\chi$ be a cut-off function such that $supp\, \chi \subseteq B(x, 1- \delta/2)$ and $\chi$ equals to $1$ on $B(x, 1- \delta)$. By integration by parts, we have
    \begin{equation} \label{eq:curvature controls.32}
        \int |\nabla^2_A (\chi \psi)|^2 \lesssim \int \left( |\nabla^*_A \nabla_A (\chi \psi)|^2 + |F_A| \cdot |\nabla_A (\chi \psi)|^2 + |F_A| \cdot | \psi| \cdot |\nabla^2_A (\chi \psi)| \right).
    \end{equation}
    Using the identity in Lemma \ref{curvature controls lemma 2.6}, we can estimate
    \begin{equation} \label{eq:curvature controls.33}
        \int |\nabla^*_A \nabla_A (\chi \psi)|^2 \lesssim_{\delta}\int\left(|\psi|^2 + |F_A|^2 |\chi \psi|^2  + |\nabla_A \psi|^2 + |P_A P^*_A \psi|^2 \right).
    \end{equation}
    Keeping the first three terms of the RHS of \eqref{eq:curvature controls.33} unchanged, we estimate the RHS of \eqref{eq:curvature controls.33} further as follows. By the Cauchy-Schwarz inequality, we have
    \begin{equation} \label{eq:curvature controls.34}
        \int |P_A P_A^* \psi|^2 \leq \int |\nabla_{A_0}P^*_A \psi + a\cdot P^*_A \psi|^2 \lesssim \int (|\nabla_{A_0}P^*_A\psi|^2 + |a\cdot P^*_A \psi|^2).
    \end{equation}
    Since $\nabla_{A_0}: L^2_1 \to L^2$ is bounded, $\norm{\nabla_{A_0}P^*_A \psi}^2_{L^2} \lesssim \norm{P^*_A\psi}^2_{L^2_1}$. From this, by Lemma \ref{curvature controls lemma 2.5}, we then have
    $$\norm{P^*_A \psi}^2_{L^2_1(B(x,1))} \leq c\left(  \norm{\psi}^2_{L^2(B(x,1))} + \int_{B(x,1)} |F_A|^2 \cdot |\psi|^2 + k_0 \sqrt{k_0^2+1} + k_0 + 1 \right).$$
    Combine with the fact that somewhere in the proof of Lemma \ref{curvature controls lemma 2.5}, we also established that $\norm{a \cdot P^*_A \psi}^2_{L^2}$ $\leq c_3 c_4 c_{15} k_0 \sqrt{k_0^2+1}$ to obtain a further estimate of \eqref{eq:curvature controls.34}
    $$\int |P_A P^*_A \psi|^2 \lesssim c \left( \int_{B(x,1)} |\psi|^2 + \int_{B(x,1)} |F_A|^2|\psi|^2 + 2k_0\sqrt{k_0^2+1} + k_0 + 1 \right).$$
    This means that we have the following estimate for \eqref{eq:curvature controls.33}
    $$\int |\nabla^*_A \nabla_A (\chi \psi)|^2 \lesssim_{\delta} \int\left( |\psi|^2 + |F_A|^2|\psi|^2 + |\nabla_A \psi|^2 \right) + O(1).$$
    In turn, we obtain an estimate for the LHS of \eqref{eq:curvature controls.32}
    \begin{align} \label{eq:curvature controls.35}
        \int |\nabla^2_A (\chi\psi)|^2 \lesssim_{\delta} \, \, \, & \norm{\psi}^2_{L^2}
     + \norm{\nabla_A \psi}^2_{L^2} + \int |F_A|^2 |\psi|^2 \nonumber \\
    & + \int \left(|F_A|\cdot |\nabla_A (\chi \psi)|^2 + |F_A| \cdot |\chi \psi| \cdot |\nabla^2_A (\chi \psi)| \right) + O(1).
    \end{align}
    Keeping the first three terms of the RHS of \eqref{eq:curvature controls.35} the same, we give estimates for the last two terms in the integrand of the RHS of \eqref{eq:curvature controls.35}. By the Peter-Paul version of the Cauchy-Schwarz inequality, for all $\tau >0$, we have
    \begin{align}\label{eq:curvature controls.36}
        \int |F_A| \cdot |\chi \psi| \cdot |\nabla^2_A (\chi \psi) | & \leq \int\left( \tau^{-1}|F_A|^2 \cdot |\chi \psi|^2 + \tau |\nabla^2_A (\chi \psi)|^2 \right) \nonumber \\
        & \leq \tau^{-1}\norm{\psi}^2_{L^\infty} \norm{F_A}^2_{L^2} + \tau \norm{\nabla^2_A (\chi \psi)}^2_{L^2}.
    \end{align}
    If $\tau$ is small enough, the second term of the RHS of \eqref{eq:curvature controls.36} can be moved to the LHS of \eqref{eq:curvature controls.35}. Now for the fourth term in the integrand of the RHS of \eqref{eq:curvature controls.35}, we use the Cauchy-Schwarz inequality to obtain
    \begin{align} \label{eq:curvature controls.37}
        \int |F_A| \cdot |\nabla_A (\chi\psi)|^2 &\leq \left(\int |F_A|^2\right)^{1/2} \left(\int |\nabla_A (\chi \psi)|^4 \right)^{1/2} \nonumber \\
        &= \norm{F_A}_{L^2}\cdot \norm{\nabla_A (\chi \psi)}^2_{L^4} \leq \sqrt{k_0}\cdot \norm{\nabla_A (\chi \psi)}^2_{L^4}.
    \end{align}
    By the Gagliardo-Nirenberg interpolation inequality
        $$\norm{f}_{L^4} \lesssim \norm{\nabla f}^{3/4}_{L^2} \norm{f}^{1/4}_{L^2}$$
    the second factor on the RHS of \eqref{eq:curvature controls.37} is estimated by
    \begin{equation} \label{eq:curvature controls.38}
        \norm{\nabla_A (\chi \psi)}^2_{L^4} \lesssim \norm{\nabla_A^2 (\chi \psi)}^{3/2}_{L^2} \norm{\nabla_A (\chi \psi)}^{1/2}_{L^2}.
    \end{equation}
    Finally, by the Peter-Paul version of Young's inequality, for all $u, v \geq 0$ , $\tau >0$, and $1/p + 1/q = 1$ with $p,q >1$, 
    $$uv \leq \frac{1}{p} \tau^{p} u^p + \frac{1}{q} \tau^{-q} v^q$$
    applied in the situation when $u = \norm{\nabla^2_A (\chi \psi)}^{3/2}_{L^2}$, $v = \norm{\nabla_A (\chi \psi)}^{1/2}_{L^2}$, and $(p,q) = (4/3, 4)$, we have 
    \begin{equation} \label{eq:curvature controls.39}
        \norm{\nabla_A(\chi \psi)}^2_{L^4} \lesssim \frac{3}{4}\tau^{4/3}\norm{\nabla^2_A (\chi \psi)}^2_{L^2} + \frac{1}{4}\tau^{-4}\norm{\nabla_A(\chi \psi)}^2_{L^2}.
    \end{equation}
    When $\tau$ is sufficiently small, the first term on the RHS of \eqref{eq:curvature controls.39} can be moved to the LHS of \eqref{eq:curvature controls.35}. Therefore, with all of the above estimates combined, we obtained the desired estimate of the proposition. 
\end{proof}

\begin{Def}\label{first compactness defintion 2.8}
    The \textit{critical radius $\rho(x)$} of a connection $A \in \mathcal{A}(\mathscr{L})$ is
    $$\rho(x) = \sup \left\{ r \in (0, r_0]: r^{1/2} \norm{F_A}_{L^2(B(x,r))} \leq 1 \right\}.$$
\end{Def}

We combine Lemma \ref{curvature controls lemma 2.3} (cf. \eqref{eq:curvature controls.20}, \eqref{eq:curvature controls.21}, \eqref{eq:curvature controls.22}) with Proposition \ref{curvature controls proposition 2.7} to immediate obtain the following corollary.

\begin{Cor}\label{curvature controls corollary 2.9}
    For all $(A,\psi, \epsilon) \in V_{k_0}(B(x,\rho(x)/2)) \times \Gamma(W_{\mathfrak{s}_{3/2}}\otimes \mathscr{L}|_{B(x,\rho(x)/2)}) \times (0,\infty)$ that are solutions of \eqref{eq:first compactness.2}, we have 
    $$\rho(x)^{1/2}\norm{\nabla^2_A \psi}_{L^2(B(x,\rho(x)/2))} = O(1).$$
    Here, as before, $B(x,\rho(x)/2)$ is a geodesic ball around $x$ on $Y$, and $k_0$ is its associated constant as in Theorem \ref{local slice theorem}. \qed
\end{Cor}

\section{A frequency function for the RS-SW equations}
We continue with our analysis of the sequence of solutions $(A_n, \psi_n, \epsilon_n)$ of \eqref{eq:first compactness.2} when $\limsup \epsilon_n = 0$. Roughly speaking, we expect that up to a subsequence and gauge transformations, such a sequence of solutions of \eqref{eq:first compactness.2} will converge to a solution away from a certain singular set on $Y$. To proceed, we first define a \textit{frequency function} for the RS-SW equations. The notion of frequency function was first introduced by Almgren in the study of critical sets of elliptic partial differential equations \cite{MR0574247}. It was also used to analyze the moduli space of solutions of the multiple-spinor Seiberg-Witten equations. In fact, the frequency function for the RS-SW equations is an adaptation of the ones used in \cite{haydys2015compactness} and \cite{taubes2012psl}.

\begin{Def}\label{frequency function definition 3.1}
    Let $x$ be any point on $Y$ and $0<r\leq r_0 <<1$ where $r_0$ is the injective radius of $Y$. For any $(A,\psi, \epsilon) \in \mathcal{A}(\mathscr{L}) \times \Gamma(W_{\mathfrak{s}_{3/2}} \otimes \mathscr{L}) \times (0,\infty)$, we define
    $$H_x(r) = \int_{B(x,r)} \left(|\nabla_A \psi|^2 + \epsilon^{-2} |\mu(\psi)|^2 - \dfrac{4}{3}\la P_A P^*_A \psi, \psi \ra \right), \quad h_x(r) = \int_{\partial B(x,r)} |\psi|^2.$$
    The \textit{frequency function} $N_x: (0, r_0] \to (0, r_0]$ of \eqref{eq:first compactness.2} is given by 
    $$N_x(r) = \dfrac{rH_x(r)}{h_x(r)}.$$
    When the context is clear, sometimes we will ignore the subscript of the base point $x$ and write $N, H, h$ instead.
\end{Def}

    For the rest of this section, we will assume that $x$ is a fixed base-point on $Y$, and $(A, \psi, \epsilon) \in V_{k_0}(B(x,r)) \times \Gamma(W_{\mathfrak{s}_{3/2}}\otimes \mathscr{L}|_{B(x,r)}) \times (0,\infty)$ is a solution of \eqref{eq:first compactness.2}, where $k_0$ is the constant that appears in Theorem \ref{local slice theorem}. We will study the dependence of $N$ on the base point (cf. Proposition \ref{frequency function proposition 3.14}). Furthermore, we will establish the following important properties of $N$:
    \begin{enumerate}
        \item $N$ is almost monotone in $r$. (cf. Corollary \ref{frequency function corollary 3.11})
        \item $N$ controls the growth of $h$. (cf. Corollary \ref{frequency function corollary 3.13})
        \item $|\psi|(x)$ controls $N$. (cf. Proposition \ref{frequency function proposition 3.16})
    \end{enumerate}

Versions of Proposition \ref{frequency function proposition 3.14}, Corollary \ref{frequency function corollary 3.11}, Corollary \ref{frequency function corollary 3.13}, and Proposition \ref{frequency function proposition 3.16} appear in the original context of \cite{MR0252808} and \cite{MR0574247}. They also appear in the gauge-theoretic setting of \cite{haydys2015compactness}, \cite{taubes2012psl} and \cite{Taubes:2016voz}. The novelty of these statements in the current context is the appearance of the Penrose operator in our definition of the frequency function. Once a control of $\la P_A P^*_A \psi, \psi\ra$ is well understood, we demonstrate that standard technique can also be applied to deduce the expected properties of $N$ listed above. Closely following \cite{haydys2015compactness}, we include the proofs of these statements in the next few subsections.

Before we get into the analysis of the frequency function $N$, we need to establish some preliminary observations about $3/2-$spinors on the boundary of a three-ball. These facts might be well-known, but we will nevertheless record them in the following subsection for the sake of self-containment and clarity.

\subsection{$3/2-$spinors on the boundary} Consider the punctured ball $\dot{B}(x,r)$ by removing the base-point $x$.  It is foliated by the surfaces $\partial B(x,r)$ with the normal vector field $\partial_r$ in the radial direction. The restriction of the spinor bundle on $\dot{B}(x,r)$ to the boundary is exactly the spinor bundle on $\partial B(x,r)$, which we will denote by $S$ and $\tilde{S}$, respectively.  Let $\tilde{\gamma}, \tilde{\nabla},$ and $\tilde{\mathbf{D}}$ be the Clifford multiplication, compatible connection, and the Dirac operator associated to the bundle of spinor bundle on $\partial B(x,r)$. For any vector $v \in T \partial B(x,r)$, we have
\begin{equation}\label{eq:40}
    \gamma(v) = -\gamma(\partial_r)\tilde{\gamma}(v), \quad \nabla_v = \tilde{\nabla}_v + \dfrac{e^{O(r^2)}}{2r}\tilde{\gamma}(v), \quad \mathbf{D} = \gamma(\partial_r)\left(\nabla_r + \dfrac{e^{O(r^2)}}{r} - \tilde{\mathbf{D}}\right).
\end{equation}
Denote $\mathbf{D}^{TB}$ by the Dirac operator associated with the bundle of spinor-valued $1-$forms. Let $\mathbf{P}$ be the Penrose operator and $\mathbf{Q}$ be the Rarita-Schwinger operator on $B(x,r)$. We also consider $\mathbf{D}^{T\partial B}$, $\mathbf{\tilde{P}}$, $\mathbf{\tilde{Q}}$ to be the Dirac operators on $T\partial B(x,r) \otimes \tilde{S}$, the Penrose operator and the Rarita-Schwinger operator on $\partial B(x,r)$, respectively. Finally, we write $W$, $\tilde{W}$ as the $3/2-$spinor bundles on $B(x,r)$ and $\partial B(x,r)$, respectively. The objective of this subsection is to derive an analog of the last formula of \eqref{eq:40} for the Rarita-Schwinger operators. 

Suppose $\Phi \in \Gamma(\tilde{W})$ is a $3/2-$spinor on $\partial B(x,r)$. Naturally, we can view it as a $3/2-$spinor on $B(x,r)$ but restricted to the boundary. Note that then we have the following orthogonal decomposition $\Gamma(W|_{\partial B(x,r)}) = \Gamma(\tilde{W}) \oplus \Gamma(\tilde{W}^{\perp})$. First, we would like to describe $\Gamma(\tilde{W}^{\perp})$ explicitly. 

Let $\{\partial_r, e_1, e_2\}$ be orthonormal vector fields on $B(x,r)$ where $\{e_1, e_2\}$ are orthonormal vector fields on $\partial B(x,r)$. Then in local normal coordinate, we can write $\phi$ as $\Phi = \phi_1 \otimes e_1 + \phi_2 \otimes e_2$, where $\phi_i \in \Gamma(\tilde{S})$. By virtue of $\Phi$ being a $3/2-$spinor, it is not hard to see that $\phi_1 = -\partial_r \cdot \phi_2$. As a result, 
$$\Gamma(\tilde{W}) = \{ -\partial_r \cdot \phi \otimes e_1 + \phi \otimes e_2 : \phi \in \Gamma(\tilde{S})\}.$$
Suppose $\Psi$ is any $3/2-$spinor of $B(x,r)$ restricted to the boundary such that $\Psi \in \Gamma(\tilde{W}^{\perp})$. In local normal coordinate, we write $\Psi = \psi \otimes \partial_r + \psi_1 \otimes e_1 + \psi_2 \otimes e_2$, where $\psi, \psi_i \in \Gamma(\tilde{S})$. Then for any $\phi \in \Gamma(\tilde{S})$, we must have 
$$\la -\partial_r \cdot \phi \otimes e_1 + \psi \otimes e_2, \psi \otimes \partial_r + \psi_1 \otimes e_1 + \psi_2 \otimes e_2\ra = 0.$$
Since the Clifford multiplication is skew-adjoint, the above is equivalent to
$$\la\phi, \partial_r \cdot \psi_1 + \psi_2\ra = 0, \quad \quad \forall \phi \in \Gamma(\tilde{S}).$$
As a result, we must have $\psi_1 = \partial_r \cdot \psi_2$. Furthermore, since $\Psi$ is a $3/2-$spinor, we have
$$\partial_r \cdot \psi + e_1 \cdot \psi_1 + e_2 \cdot \psi_2 = 0.$$
Combine with the fact that $\psi_1 = \partial_r \cdot \psi_2 \equiv \partial_r \cdot \psi$, we obtain $\psi = 2 \tilde{\gamma}(e_2)\psi$. Therefore, 
$$\Gamma(\tilde{W}^{\perp}) = \{2\tilde{\gamma}(e_2)\psi \otimes \partial_r + \partial_r \cdot \psi \otimes e_1 + \psi \otimes e_2: \psi \in \Gamma(\tilde{S})\}.$$
By a change of variable, setting $\psi \to e_2\cdot \psi$, then in local normal coordinate any $3/2-$spinor $\Psi \in \Gamma(\tilde{W}^{\perp})$ can be written as
\begin{align}
    \Psi &= -2 \partial_r \cdot \psi \otimes \partial_r + e_1 \cdot \psi \otimes e_1 + e_2 \cdot \psi \otimes e_2 \nonumber \\
    & = -3\iota(\psi)- 3 \partial_r \cdot \psi \otimes \partial_r.\nonumber
\end{align}
Therefore, if we define the embedding $\hat{\iota} : \tilde{S} \to TB(x,r) \otimes S |_{\partial B(x,r)} = T\partial B(x,r) \otimes \tilde{S} \oplus ( \CN\partial_r \otimes \tilde{S})$ to be $\hat{\iota}(\psi) = -3\iota(\psi) - 3\partial_r \cdot \psi \otimes \partial_r$, then $\Gamma(\tilde{W}^{\perp}) = \hat{\iota}(\tilde{S})$. We summarize the above discussion in the following lemma.

\begin{Lemma}\label{frequency function lemma 3.2}
    Let $W$, $\tilde{W}$ be the $3/2-$spinor bundle on $B(x,r)$ and $\partial B(x,r)$, respectively. Let $\tilde{S}$ be the spinor bundle on $\partial B(x,r)$ induced by restricting $S$ to the boundary. We have the following orthogonal decomposition
    $W|_{\partial B(x,r)} = \tilde{W} \oplus \hat{\iota}(\tilde{S})$. \qed
\end{Lemma}

Define $L : \tilde{S} \to \tilde{S} \otimes  \CN \partial_r$ to be a vector bundle isomorphism (in fact, it is also an isometry) given by $L(\phi) = \phi \otimes \partial_r$. Note that then $ \hat{\iota}(\tilde{S}) \oplus \tilde{W} \cong (\tilde{S}\otimes \CN\partial_r)\oplus\tilde{W}$ via
$$T = \begin{pmatrix} L\hat{\iota}^{-1} & 0 \\ 0 & 1 \end{pmatrix}.$$
This fact will be useful later when we perform a coordinate change on the matrix representation of $\mathbf{D}^{TB}$.

Recall that the Levi-Civita connection on the boundary is related to the Levi-Civita connection on the whole $B(x,r)$ as follows
$$\nabla^{LC}_{v} \bullet  = \tilde{\nabla}^{LC}_{v} \bullet + \frac{e^{O(r^2)}}{r}\la v, \bullet\ra \partial_r, \quad \quad \forall v \in \Gamma(T\partial B(x,r)).$$
With this fact, we compute $\mathbf{D}^{TB}|_{\partial B(x,r)}$ in matrix form with respect to the orthogonal decomposition $TB(x,r) \otimes S|_{\partial B(x,r)} = (\tilde{S} \otimes \CN \partial_r)\oplus T\partial B(x,r) \otimes \tilde{S}$. By definition of the Dirac operator, we have
\begin{align*}
   \mathbf{D}^{TB} &= \sum_{i=0}^{2}e_i \cdot (\nabla_{e_i} \otimes 1 + 1 \otimes \nabla^{LC}_{e_i}) = \mathbf{D}\otimes 1 + \sum_{i=1}^{2}e_i \cdot \otimes \nabla^{LC}_{e_i}, \quad \quad \text{ where } e_0 \equiv \partial_r \\
   & = \partial_r\cdot \left(\nabla_r + \frac{e^{O(r^2}}{r} - \mathbf{\tilde{D}}\right)\otimes 1 + \partial_r \cdot \otimes \nabla^{LC}_r -\sum_{i=1}^{2}\partial_r \cdot e_i \tilde{\cdot} \otimes\left( \tilde{\nabla}^{LC}_{e_i} + \frac{e^{O(r^2)}}{r}\la e_i, \bullet\ra \partial_r \right)\\
   & = \partial_r\cdot \left\{\nabla_r + \frac{e^{O(r^2)}}{r} -\left (\mathbf{\tilde{D}}\otimes 1 + \sum_{i=1}^{2}e_i \tilde{\cdot}\otimes \tilde{\nabla}^{LC}_{e_i}\right) - \frac{e^{O(r^2)}}{r}\sum_{i=1}^{2}e_i\tilde{\cdot} \otimes \la e_i, \cdot\ra\partial_r\right\}.
\end{align*}
Hence, when we apply $\mathbf{D}^{TB}$ to $\psi \otimes \partial_r$, we obtain
$$\mathbf{D}^{TB}(\psi \otimes \partial_r) = \partial_r\cdot \left(\nabla_r(\psi \otimes \partial_r) + \frac{e^{O(r^2)}}{r}\psi\otimes \partial_r - \mathbf{\tilde{D}}\psi\otimes \partial_r\right).$$
On the other hand, if we apply $\mathbf{D}^{TB}$ to $\psi_1 \otimes e_i$, we get
$$\mathbf{D}^{TB}(\psi_i\otimes e_i) = \partial_r \cdot \left(\nabla_r(\psi_i \otimes e_i) + \frac{e^{O(r^2)}}{r}\psi_i\otimes e_i - \mathbf{D}^{T\partial B}(\psi_i \otimes e_i) - \frac{e^{O(r^2)}}{r}e_i \tilde{\cdot} \psi_i \otimes \partial_r\right).$$
As a result, we obtain the following proposition. 

\begin{Prop}\label{frequency function proposition 3.3}
    With respect to the orthogonal decomposition $TB(x,r) \otimes S = (S \otimes \CN \partial_r) \oplus T\partial B(x,r) \otimes S$, $\mathbf{D}^{TB}$ has the following matrix form
    $$\displaystyle \mathbf{D}^{TB} = \begin{pmatrix}
        \partial_r \cdot (\nabla_r + e^{O(r^2)}/r - \mathbf{\tilde{D}}\otimes 1) & e^{O(r^2)}/r\,\tilde{\gamma}\otimes \partial_r \\ 0 & \partial_r \cdot (\nabla_r + e^{O(r^2)}/r - \mathbf{D}^{T\partial B})
    \end{pmatrix}.$$
\end{Prop}

\begin{Cor}\label{frequency function corollary 3.4}
    Let $\Phi$ be a $3/2-$spinor on $B(x,r)$.  The restriction of $\Phi$ to  $\partial B(x,r)$ can be written as $\Phi = \hat{\iota}(\phi) + \Psi$, where $\Psi \in \Gamma(\tilde{W})$ and $\phi \in \Gamma(\tilde{S})$. Then $\mathbf{Q}\Phi|_{\partial B(x,r)}$ equals to
    $$  \begin{pmatrix}
        \hat{\iota}L^{-1}\partial_r\cdot\left(\nabla_r + \frac{e^{O(r^2)}}{r}-\mathbf{\tilde{D}}\otimes 1\right)L\hat{\iota}^{-1} & 0 \\ 0 & \tilde{\pi}\partial_r \cdot\left(\nabla_r + \frac{e^{O(r^2}}{r}-\mathbf{D}^{T\partial B}\right) 
        \end{pmatrix}\cdot \begin{pmatrix}
            \hat{\iota}(\phi) \\ \Psi
    \end{pmatrix},$$
    where $\tilde{\pi}$ is the orthogonal projection from $T\partial B(x,r) \otimes \tilde{S}$ onto the $3/2-$spinor bundle $\tilde{W}$.
\end{Cor}

\begin{proof}
    Note that $\tilde{\gamma} \equiv 0$ when restricted to $\Gamma(\tilde{W})$. And a change of basis via $T$ immediately tells us that the matrix representing $\mathbf{Q}|_{\Gamma(\tilde{W})}$ is given by
    $$T^{-1}\begin{pmatrix}
        1 & 0 \\ 0 & \tilde{\pi}
    \end{pmatrix} \cdot \begin{pmatrix}
        \partial_r \cdot (\nabla_r + e^{O(r^2)}/r - \mathbf{\tilde{D}}\otimes 1) & e^{O(r^2)}/r\,\tilde{\gamma}\otimes \partial_r \\ 0 & \partial_r \cdot (\nabla_r + e^{O(r^2)}/r - \mathbf{D}^{T\partial B})
    \end{pmatrix} T.$$
    Then, a straightforward calculation yields the result.
\end{proof}

\begin{Lemma}\label{frequency function lemma 3.5}
    $\nabla_r$ maps $3/2-$spinors to $3/2-$spinors on $\partial B(x,r)$.
\end{Lemma}

\begin{proof}
    Let $\Psi$ be a $3/2-$spinor on $\partial B(x,r)$. In local normal coordinate, we have
    $$\nabla_r \Psi =  \nabla_r \psi_1 \otimes e_1 + \nabla_r \psi_2 \otimes e_2.$$
    Applying $\gamma$ to both sides of the above identity, we get 
    $$\gamma(\nabla_r \psi) =  \gamma(e_1)\nabla_r \psi_1 + \gamma(e_2)\nabla_r \psi_2.$$
    By the compatibility of the connection with the Clifford multiplication and the fact that we are working with normal coordinates, we can re-write the above identity as $\gamma(\nabla_r \psi) = \nabla_r (\gamma(e_1) \psi_1) + \nabla_r (\gamma(e_2)\psi_2) = \nabla_r(\gamma(e_1)\psi_1 + \gamma(e_2)\psi_2) = 0$. Therefore, $\nabla_r \Psi$ is also another $3/2-$spinor as claimed.
\end{proof}

\begin{Lemma}\label{frequency function lemma 3.6}
    $\gamma(\partial_r)$ maps $3/2-$spinors to $3/2-$spinors on $\partial B(x,r)$. 
\end{Lemma}

\begin{proof}
    Let $\Psi$ be $3/2-$spinor on $\partial B(x,r)$. In the local normal coordinate, we have $\psi = \psi_1 \otimes e_1 + \psi_2 \otimes e_2$. Thus, $\gamma(\partial_r)\psi = \gamma(\partial_r)\psi_1 \otimes e_1 + \gamma(\partial_r)\psi_2 \otimes e_2$. We compute
    \begin{align}
        \gamma(\gamma(\partial_r)\psi) & = \gamma(e_1)\gamma(\partial_r) \psi_1 + \gamma(e_2)\gamma(\partial_r)\psi_2 \nonumber \\
        & = -\gamma(\partial_r)\gamma(e_1)\psi_1  - \gamma(\partial_r)\gamma(e_2)\psi_2 = -\gamma(\partial_r)(\gamma(e_1)\psi_1 + \gamma(e_2)\psi_2). \nonumber
    \end{align}
    Clearly, this implies that $\gamma(\partial_r)\Psi$ must also be a $3/2-$spinor. 
\end{proof}

\begin{Cor}\label{frequency function corollary 3.7}
    Let $\Phi$ be a harmonic $3/2-$spinor on $B(x,r)$, i.e, $\mathbf{Q}\Phi = 0$. Consider the restriction of $\Phi$ to $\partial B(x,r)$ that is written as $\Phi = \hat{\iota}(\phi) + \Psi$, where $\phi \in \Gamma(\tilde{S})$ and $\Psi \in \Gamma(\tilde{W})$. Then we must have
    \begin{align}
    \mathbf{\tilde{D}}\phi \otimes \partial_r &= \nabla_r \phi \otimes \partial_r + \frac{e^{O(r^2)}}{r}\phi \otimes \partial_r \label{eq:frequency function.41} \\
    \mathbf{\tilde{Q}}\Psi &= \nabla_r \Psi + \frac{e^{O(r^2)}}{r}\Psi. \label{eq:frequency function.42}
    \end{align}
\end{Cor}

\begin{proof}
    \eqref{eq:frequency function.41} follows directly from Corollary \ref{frequency function corollary 3.4}. It remains for us to prove \eqref{eq:frequency function.42}. From Corollary \ref{frequency function corollary 3.4}, Lemma \ref{frequency function lemma 3.5} and Lemma \ref{frequency function lemma 3.6}, we have that
    $$\partial_r\cdot \nabla_r \Psi + \partial_r \cdot \frac{e^{O(r^2)}}{r}\Psi - \tilde{\pi}(\partial_r \cdot \mathbf{D}^{T\partial B}\Psi) = 0.$$
    It is not hard to see that 
    $$\tilde{\pi}(\partial_r \cdot \mathbf{D}^{T\partial B} \Psi) = \partial_r \cdot \mathbf{\tilde{Q}}\Psi - \tilde{\iota}\tilde{\gamma}(\partial_r \cdot \mathbf{D}^{T\partial B}\Psi) + \partial_r \cdot \tilde{\iota}\tilde{\gamma}(\mathbf{D}^{T\partial B}\psi).$$
    To complete the proof, we must calculate the second and third terms on the RHS of the above identity. Note that in local normal coordinate, we have $\mathbf{D}^{T\partial B} \Psi = \mathbf{\tilde{D}}\psi_1 \otimes e_1 + \mathbf{\tilde{D}}\psi_2 \otimes e_2$. As a result, 
    \begin{align*}
        \tilde{\gamma}(\mathbf{D}^{T\partial B}\Psi) & = e_1 \tilde{\cdot}\,\mathbf{\tilde{D}}\psi_1 + e_2 \tilde{\cdot}\,\mathbf{\tilde{D}}\psi_2 \\
        & = -\mathbf{\tilde{D}}(e_1 \tilde{\cdot}\psi_1) - 2\tilde{\nabla}_{e_1}\psi_1 - \mathbf{\tilde{D}}(e_2 \tilde{\cdot}\psi_2) - 2 \tilde{\nabla}_{e_2}\psi_2 = 2 \mathbf{\tilde{P}}^* \Psi.
    \end{align*}
    Applying $\partial_r \cdot \tilde{\iota}$ to both sides of the above equation, we get
    $$\partial_r \cdot \tilde{\iota}(\tilde{\gamma}(\mathbf{D}^{T\partial B}\psi)) = -\partial_r \cdot e_1 \tilde{\cdot} \, \mathbf{\tilde{P}}^*\Psi \otimes e_1 - \partial_r \cdot e_2 \tilde{\cdot} \, \mathbf{\tilde{P}}^*\Psi \otimes e_2.$$
    On the other hand, it is not difficult to see that $\tilde{\gamma}(\partial_r \cdot \mathbf{D}^{T\partial B} \Psi) = -2 \partial_r \cdot \mathbf{\tilde{P}}^* \Psi$. Note that $e_i \tilde{\cdot} \partial_r = - \partial_r \cdot e_i \tilde{\cdot}$. Then we have
    \begin{align*}
        -\tilde{\iota}\tilde{\gamma}(\partial_r \cdot \mathbf{D}^{T\partial B}\Psi) & = 2 \tilde{\iota}(\partial_r \cdot \mathbf{\tilde{P}}^*\Psi) \\
        &= -e_1 \tilde{\cdot} \partial \cdot \mathbf{\tilde{P}}^*\Psi \otimes e_1 - e_2 \tilde{\cdot} \partial \cdot \mathbf{\tilde{P}}^*\Psi \otimes e_2 \\
        & = \partial_r \cdot e_1 \tilde{\cdot} \, \mathbf{\tilde{P}}^*\Psi \otimes e_1 + \partial_r \cdot e_2 \tilde{\cdot} \, \mathbf{\tilde{P}}^*\Psi \otimes e_2.
    \end{align*}
    Thus, $\partial_r\cdot \nabla_r \Psi + \partial_r \cdot \frac{e^{O(r^2)}}{r}\Psi - \partial_r \cdot \mathbf{\tilde{Q}}\Psi = 0$. By applying the Clifford multiplication by $\partial_r$ to both sides, we obtain \eqref{eq:frequency function.42} as claimed.
\end{proof}

We are now ready to prove the properties of $N$ stated at the beginning of the section.

\subsection{Analysis of $N$} Note that the first derivative of the frequency function $N(r)$ is given by
\begin{equation} \label{eq:frequency function.49}
    N'(r) = \frac{H(r)}{h(r)}+\frac{rH'(r)}{h(r)} - \frac{rh'(r)H(r)}{h^2(r)}.
\end{equation}
To understand the monotonicity of $N$, we need to analyze $H', h'$ further.

\begin{Lemma}\label{frequency function lemma 3.8}
    $h(r)$ satisfies the following properties:
    \begin{equation}\label{eq:frequency function.50}
        h'(r) = \frac{2h(r)}{r}+\int_{\partial B(x,r)} \partial_r |\psi|^2 + O(r)h(r).
     \end{equation}
     \begin{equation}\label{eq:frequency function.51}
         h'(r) = \frac{(2+2N(r)+O(r^2))h(r)}{r}.
     \end{equation}
     \begin{equation}\label{eq:frequency function.52}
         \int_{B(x,r)}|\psi|^2\lesssim rh(r).
     \end{equation}
\end{Lemma}

\begin{proof}
    The proof of \eqref{eq:frequency function.50} is a standard computation in PDE. When the metric is flat, note that a change of variable that turns the circle centered at $x$ of radius $r$ to a unit concentric circle and the product rule allows us to write 
    $$h'(r) = \frac{2}{r}\int_{\partial B(x,r)} |\psi|^2 + \int_{\partial B(x,r)} \partial_r |\psi|^2.$$
    If the metric is not flat, there is an additional error term given by $O(r)h(r)$. We move on to the proof of \eqref{eq:frequency function.51} and \eqref{eq:frequency function.52}. Firstly, we make the following claim.

    \textit{Claim 1:} $\norm{\psi}^2_{L^2(B(x,r))} \lesssim (1+N(r))h(r)$. \\
    Indeed, recall the following fact
    \begin{equation}\label{eq:frequency function.53}
        \int_{B(x,r)} d(x,\bullet)^{-2}f^2 \lesssim r^{-1}\int_{\partial B(x,r)}f^2 + \int_{B(x,r)} |df|^2.
    \end{equation}
    We apply \eqref{eq:frequency function.53} to $f := |\psi|$ to obtain
    \begin{align}\label{eq:frequency function.54}
        r^{-2}\int_{B(x,r)} |\psi|^2 &\lesssim r^{-1}\int_{\partial B(x,r)} |\psi|^2 + \int_{B(x,r)} |d|\psi||^2 \nonumber \\
        & = r^{-1}h(r) + \int_{B(x,r)}|d|\psi||^2.
    \end{align}
    The second term on the RHS of \eqref{eq:frequency function.54} can be estimated further by Kato's inequality to arrive at $r^{-2} \norm{\psi}^2_{L^2(B(x,r))} \lesssim rh(r) + \norm{\nabla_A \psi}^2_{L^2(B(x,r))}$. As a result, we do have 
    \begin{equation}\label{eq:frequency function.55}
        \norm{\psi}^2_{L^2(B(x,r))} \lesssim (1+N(r))h(r).
    \end{equation}

    \textit{Claim 2:} $h'(r) > 0$.\\
   We apply the integration by parts formula from Lemma \ref{first compactness lemma 1.6} to $U = B(x,r)$ and $f \equiv 1$ to obtain
   \begin{align}\label{eq:frequency function.56}
       \int_{B(x,r)}& \left(-\frac{8}{3}\la P_A P^*_A \psi, \psi \ra + \frac{s}{2}|\psi|^2 - 2\la(1\otimes Ric)\psi, \psi\ra + 2\epsilon^{-2}|\mu(\psi)|^2 + 2|\nabla_A \psi|^2 \right) \nonumber \\
       & = \int_{\partial B(x,r)} \partial_r |\psi|^2.
   \end{align}
   The LHS of \eqref{eq:frequency function.56} can be re-arranged to become $2H(r) + O(1)\norm{\psi}^2_{L^2(B(x,r))}$. So, we do have 
   \begin{equation}\label{eq:frequency function.57}
       2H(r) + O(1) \norm{\psi}^2_{L^2(B(x,r))} = \int_{\partial B(x,r)} \partial_r |\psi|^2.
   \end{equation}
    Now, we combine \eqref{eq:frequency function.57} with \eqref{eq:frequency function.55} and \eqref{eq:frequency function.50} to immediately obtain
    $$h'(r) = \frac{(1+O(r^2))(2+ 2N(r))h(r)}{r}.$$
    Since $r\leq r_0$, the above implies that $h'(r)>0$ as claimed. Lastly, \eqref{eq:frequency function.51} follows directly from the fact that $h'(r) >0$. And \eqref{eq:frequency function.51} can be verified by combining \eqref{eq:frequency function.57} with \eqref{eq:frequency function.51} and \eqref{eq:frequency function.50}. 
\end{proof}

\begin{Lemma}\label{frequency function lemma 3.9}
    $H$ satisfies the following
    \begin{equation} \label{eq:frequency function.58}
        H'(r) = \frac{H(r)}{r} + \int_{\partial B(x,r)} \left(3 |\nabla^A_r \psi|^2 +\epsilon^{-2}|i(\partial_r)\mu(\psi)|^2\right) + O((1+N(r))h(r)) + O(1).
    \end{equation}
\end{Lemma}

\begin{proof}
    Using Stoke's theorem and a change of variable that turns the ball centered at $x$ of radius $r$ to a unit concentric ball, we already have 
    \begin{equation}\label{eq:frequency function.59}
        H'(r) = \frac{H(r)}{r} + \int_{\partial B(x,r)} \left(|\nabla_A \psi|^2 + \epsilon^{-2} |\mu(\psi)|^2 - \dfrac{4}{3}\la P_A P^*_A \psi, \psi \ra \right).
    \end{equation}
    It remains for us to match up the second terms of the RHS's of \eqref{eq:frequency function.59} and \eqref{eq:frequency function.58}. When restrict $\psi$ to the boundary, we write it as $\psi = \hat{\iota}(\phi) + \xi$. In local normal coordinate, $\xi = \xi_1 \otimes e_1 + \xi_2 \otimes e_2$, where $\xi_i \in \Gamma(\tilde{S})$. Note that for each $e_i$, we have
    \begin{align*}
        |\nabla_{e_i}\psi|^2 &= \left | \tilde{\nabla}_{e_i}\phi \otimes \partial_r + \frac{e^{O(r^2)}}{r}e_i \tilde{\cdot}\phi \otimes \partial_r + \frac{e^{O(r^2)}}{r}\xi_i\otimes \partial_r + \tilde{\nabla}_{e_i}\xi + \frac{e^{O(r^2)}}{2r}e_i \tilde{\cdot}\xi\right|^2 \\
        &= |\tilde{\nabla}_{e_i}T\psi|^2 + \frac{e^{O(r^2)}}{r}\la \tilde{\nabla}_{e_i}\phi \otimes \partial_r, e_i \tilde{\cdot}\phi \otimes \partial_r \ra + \frac{e^{O(r^2)}}{r}\la \tilde{\nabla}_{e_i}\xi, e_i \tilde{\cdot}\xi\ra + \\
        &+ \frac{e^{O(r^2)}}{r^2}\la e_i \tilde{\cdot}\phi, \xi_i\ra + \frac{e^{O(r^2)}}{r^2}|\xi_i|^2 + \frac{e^{O(r^2)}}{4r^2}|\phi|^2+\frac{e^{O(r^2)}}{4r^2}|\xi|^2 + \frac{2e^{O(r^2)}}{r}\la \tilde{\nabla}_{e_i}\phi, \xi_i\ra.
    \end{align*}
    Thus, $|\nabla_{e_1}\psi|^2 + |\nabla_{e_2}\psi|^2$ can be written as
    \begin{align*}
        |\tilde{\nabla}T\psi|^2 &- \frac{e^{O(r^2)}}{r}\la \mathbf{\tilde{D}}\phi, \phi\ra - \frac{e^{O(r^2)}}{r}\la \mathbf{D}^{TB}\xi, \xi \ra +  \frac{3e^{O(r^2)}}{2r^2}|\xi|^2 + \frac{e^{O(r^2)}}{2r^2}|\phi|^2 + \\
        &  + \frac{2e^{O(r^2)}}{r} \la \tilde{\nabla}_{e_1}\phi, \xi_1\ra + \frac{2e^{O(r^2)}}{r} \la \tilde{\nabla}_{e_2}\phi, \xi_2\ra.
    \end{align*}
    When integrating the above expression over $\partial B(x,r)$, we can replace the last two terms with
    $$\int_{\partial B(x,r)} \frac{2e^{O(r^2)}}{r}\la \phi, \mathbf{\tilde{P}}^*\xi\ra.$$
    Therefore, combining with Corollary \ref{frequency function corollary 3.7}, we obtain
    \begin{align}
        \int_{\partial B(x,r)} |\nabla \psi|^2 - |\nabla_r \psi|^2 & = \int_{\partial B(x,r)} |\tilde{\nabla}T\psi|^2 - \frac{e^{O(r^2)}}{r}\la \nabla_r \phi, \phi\ra - \frac{e^{O(r^2)}}{r}\la\nabla_r \xi, \xi \ra + \nonumber \\
        & + \frac{2e^{O(r^2)}}{r}\la \phi, \mathbf{\tilde{P}}\xi\ra - \frac{e^{O(r^2)}}{2r^2}|\phi|^2 + \frac{e^{O(r^2)}}{2r^2}|\xi|^2. \label{eq:frequency function.60}
    \end{align}
    On the other hand, by the Weitzenb\"ock formulae of the Rarita-Schwinger operator and Dirac operator on $\partial B(x,r)$, we have 
    \begin{align}
        \int_{\partial B(x,r)} |\tilde{\nabla}T\psi|^2 & = \int_{\partial B(x,r)} \la \tilde{\nabla}^* \tilde{\nabla}T\psi, T\psi\ra \nonumber \\
        & = \int_{\partial B(x,r)} \left \la \begin{pmatrix}
            \tilde{\nabla}^* \tilde{\nabla} & 0 \\ 0 &\tilde{\nabla}^* \tilde{\nabla}
        \end{pmatrix} \cdot \begin{pmatrix}
            \phi \otimes \partial_r \\ \xi
        \end{pmatrix} , \begin{pmatrix}
            \phi \otimes \partial_r \\ \xi
        \end{pmatrix}\right \ra. \nonumber \\
        &=\int_{\partial B(x,r)} |\mathbf{\tilde{D}}\phi|^2 - \frac{e^{O(r^2)}}{2r^2}|\phi|^2 + | \mathbf{\tilde{Q}}\xi|^2 + 2 |\mathbf{\tilde{P}}^*\xi|^2 - \frac{3 e^{O(r^2)}}{2r^2}|\xi|^2 \label{eq:frequency function.61}
    \end{align}
    Using Corollary \ref{frequency function corollary 3.7} again, we can expend the first and the third term on the RHS of \eqref{eq:frequency function.61} and re-arrange to obtain
    \begin{align}
        \int_{\partial B(x,r)}|\tilde{\nabla}T\psi|^2 &=\int_{\partial B(x,r)} |\nabla_r \phi|^2 + \frac{2e^{O(r^2)}}{r}\la \nabla_r \phi, \phi\ra + \frac{e^{O(r^2)}}{2r^2}|\phi|^2 + \nonumber \\
        &+ |\nabla_r \xi|^2 + \frac{2e^{O(r^2)}}{r} \la \nabla_r \xi, \xi\ra - \frac{e^{O(r^2)}}{2r^2}|\xi|^2 + 2 |\mathbf{\tilde{P}}^*\xi|^2. \label{eq:frequency function.62}
    \end{align}
    Plug in \eqref{eq:frequency function.62} into the first term on the RHS of \eqref{eq:frequency function.60}, we get
    \begin{align}
        \int_{\partial B(x,r)} |\nabla \psi|^2 - |\nabla_r \psi|^2 &= \int_{\partial B(x,r)} |\nabla_r T\psi |^2 + \frac{e^{O(r^2)}}{r}\la\nabla_r T\psi, T\psi\ra + \nonumber \\
        &+ \frac{2e^{O(r^2)}}{r}\la \phi, \mathbf{\tilde{P}}^*\xi\ra + 2 |\mathbf{\tilde{P}}^*\xi|^2. \label{eq:frequency function.63}
    \end{align}
    By the Cauchy-Schwarz inequality and Lemma \ref{curvature controls lemma 2.3}, it is not hard to see that 
    $$\norm{\nabla_r T\psi}^2_{L^2(\partial B(x,r))} \lesssim 2 \norm{T\nabla_r \psi}^2 + \norm{\psi}^2 = 2\norm{\nabla_r \psi}^2 + O(1).$$
    Hence, we can re-write \eqref{eq:frequency function.63} as
    \begin{align*}
        \norm{\nabla\psi}^2_{L^2} = \int_{\partial B(x,r)}3 |\nabla_r \psi|^2 + \frac{e^{O(r^2)}}{r}\la \nabla_r T\psi, T\psi\ra + \frac{2e^{O(r^2)}}{r}\la \phi, \mathbf{\tilde{P}}^*\xi\ra + 2 |\mathbf{\tilde{P}}^*\xi|^2 +O(1).
    \end{align*}
    The above immediately implies that $\norm{\nabla \psi}^2_{L^2(\partial B(x,r))} - \dfrac{4}{3}\norm{\mathbf{P}\psi}^2_{L^2(\partial B(x,r))}$ is exactly equals to
    \begin{align}
        \int_{\partial B(x,r)} 3 |\nabla_r \psi|^2 +\frac{e^{O(r^2)}}{r}\la \nabla_r T\psi, T\psi\ra + \frac{2e^{O(r^2)}}{r}\la \phi, \mathbf{\tilde{P}}^*\xi\ra + 2 |\mathbf{\tilde{P}}^*\xi|^2 - \frac{4}{3}|\mathbf{P}^*\psi|^2 + O(1). \label{eq:frequency function.64}
    \end{align}
    Via the change of basis given by $T$, we can think of the norm of the divergence of $\psi$ as follows
    \begin{align*}
        |\mathbf{P}^*\psi|^2 &= - 2 \la \tilde{\mathbf{P}}^*\xi, \nabla_r \phi\ra + |\tilde{\mathbf{P}}^* \xi|^2 + |\nabla_r \phi|^2 \\
        & = -2 d\,\la \tilde{\mathbf{P}}^*\xi, \phi \ra + \la \nabla_r \tilde{\mathbf{P}}^*\xi, \phi\ra +|\tilde{\mathbf{P}}^* \xi|^2 + |\nabla_r \phi|^2.
    \end{align*}
    As a result, \eqref{eq:frequency function.64} can be re-written as
    \begin{align}
        \int_{\partial B(x,r)} 3|\nabla_r \psi|^2  &+\frac{e^{O(r^2)}}{r}\la \nabla_r T\psi, T\psi\ra + \left\la \phi, \frac{2e^{O(r^2)}}{r}\mathbf{\tilde{P}}^*\xi + \frac{8}{3}\nabla_r \mathbf{\tilde{P}}^*\xi\right\ra \nonumber \\
        &+ \frac{2}{3} |\mathbf{\tilde{P}}^*\xi|^2 - \frac{4}{3}|\nabla_r \phi|^2 + O(1). \label{eq:frequency function.65}
    \end{align}
    But by Cauchy-Schwarz inequality, we have
    \begin{align*}
        \frac{2}{3}|\tilde{\mathbf{P}}^*\xi|^2 &= \frac{2}{3}|\mathbf{P}^*\psi + \nabla_r \phi |^2 \leq \frac{4}{3}(|\mathbf{P}^*\psi|^2 + |\nabla_r \phi|^2).
    \end{align*}
    At the same time by Lemma \ref{curvature controls lemma 2.3},
    \begin{align*}
        \frac{2e^{O(r^2)}}{r}\la\phi, \mathbf{\tilde{P}}^*\xi\ra_{L^2(\partial B(x,r))} & = \frac{2e^{O(r^2)}}{r} \la \mathbf{\tilde{P}}\phi, \xi \ra \leq \frac{2e^{O(r^2)}}{r}\norm{\tilde{\nabla}\phi}^2 \cdot \norm{\xi}^2 = O(1).
    \end{align*}
    Another application of the Cauchy-Schwarz inequality combined with Lemma \ref{curvature controls lemma 2.3} also yields
    \begin{align*}
        \left \la \phi, \frac{8}{3}\nabla_r \tilde{\mathbf{P}}^*\xi\right \ra & \lesssim \frac{8}{3} \norm{\phi}^2 \cdot \norm{\tilde{\mathbf{P}}^*\xi}^2 \lesssim \norm{\phi}^2\cdot (\norm{\mathbf{P}^*\psi}^2 + \norm{\nabla_r \phi}^2) \\
        & \lesssim \norm{\phi}^2\cdot (\norm{\mathbf{P}^*\psi}^2 + \norm{ \phi}^2) = O(1).
    \end{align*}
    Finally, from somewhere at the beginning of Section 2, we have already shown that $\norm{\mathbf{P}^*\psi}^2_{L^2}$ $= O(1)$. Combine all of the above estimates with \eqref{eq:frequency function.57}, we have our desired identity in the case when $A$ is a product connection.
    Lastly, when $A$ is not a product connection, a similar calculation as the above, up to a change in notation, yields a similar result with two additional terms. The first is estimated by $O(1)h(r)$, whereas the second term is the curvature term which can now be written as 
    $$\int_{\partial B(x,r)} \epsilon^{-2}|i(\partial_r)\mu(\psi)|^2.$$
    Therefore, \eqref{eq:frequency function.58} is verified as claimed.
\end{proof}

Now, we are ready to state one of the main propositions of this subsection. With this proposition, we can finally prove the almost monotonicity of $N(r)$. 

\begin{Prop}\label{frequency function proposition 3.10}
    The frequency function for the RS-SW equations satisfies 
    $$N'(r) \geq O(r)(1+N(r)).$$
\end{Prop}

\begin{proof}
    We substitute \eqref{eq:frequency function.51} and \eqref{eq:frequency function.58} into \eqref{eq:frequency function.49} to obtain
    \begin{align}\label{eq:frequency function.66}
        N'(r) & = \frac{2H(r)}{h(r)} + \frac{r}{h(r)}\left(\int_{\partial B(x,r)} 3|\nabla^A_r \psi|^2 + \epsilon^{-2}|i(\partial_r)\mu(\psi)|^2\right) +  \nonumber \\
        &+ rO((1+N(r)) + \frac{rO(1)}{h(r)} - \frac{(2+2N(r) + O(r^2))H(r)}{h(r)} \nonumber \\
        &= \frac{3r}{h(r)}\left(\int_{\partial B(x,r)} |\nabla^A_r \psi|^2 + \frac{\epsilon^{-2}}{3}|i(\partial_r)\mu(\psi)|^2\right)+ \frac{rO(1)}{h(r)}+\nonumber\\
        &- \frac{2r H^2(r)}{h^2(r)}- O(r^2)\frac{H(r)}{h(r)}  + rO((1+N(r))).
    \end{align}
    Now, if we use \eqref{eq:frequency function.51} and \eqref{eq:frequency function.57}, we can re-write the fourth and fifth term on the RHS of \eqref{eq:frequency function.66} as
    \begin{equation} \label{eq:frequency function.67}
        - \frac{2r H^2(r)}{h^2(r)}- O(r^2)\frac{H(r)}{h(r)} = -\frac{2r}{h^2(r)} \left(\int_{\partial B(x,r)} \la \nabla^A_r \psi, \psi \ra \right)^2.
    \end{equation}
    As a result, \eqref{eq:frequency function.67} and \eqref{eq:frequency function.66} yields
    \begin{align}\label{eq:frequency function.68}
        N'(r) = \frac{3r}{h(r)}&\left(\int_{\partial B(x,r)} |\nabla^A_r \psi|^2 + \frac{\epsilon^{-2}}{3}|i(\partial_r)\mu(\psi)|^2\right) + \frac{rO(1)}{h(r)} +  \nonumber \\
        & -\frac{2r}{h^2(r)} \left(\int_{\partial B(x,r)} \la \nabla^A_r \psi, \psi \ra \right)^2 + O(r)(1+N(r)).
    \end{align}
    Note that by the Cauchy-Schwarz inequality,
    \begin{align}
        &\left(\int_{\partial B(x,r)} \la \nabla^A_r \psi, \psi \ra \right)^2 \leq \left( \int_{\partial B(x,r)} |\nabla^A_r \psi|^2 \right)\cdot h(r) \nonumber 
    \end{align}
    As a result, the first term on the RHS of \eqref{eq:frequency function.68} is greater or equal to the fourth term on the RHS of the same equation. Therefore, we have $N'(r) \geq O(r)(1+N(r))$ as claimed.
\end{proof}

\begin{Cor}[\textbf{$N(r)$ is almost monotone in $r$}]\label{frequency function corollary 3.11}
    For all $0<s\leq r$, we have 
    $$N(s) \leq e^{O(r^2-s^2)}N(r) + O(r^2-s^2).$$
\end{Cor}

\begin{proof}
    By Proposition \ref{frequency function proposition 3.10}, we have
    \begin{equation}\label{eq:frequency function.69}
        \frac{d}{dr}\, \log(1+N(r)) = \frac{N'(r)}{1+N(r)} \geq O(r) = -2cr.
    \end{equation}
    We integrate both sides of \eqref{eq:frequency function.69} to obtain
    \begin{equation*}
        \int_{s}^{r}\frac{d}{dt}\, \log(1+N(t))\, dt \geq \int_{s}^{r} -2ct \,dt = -c(r^2-s^2).
    \end{equation*}
    Equivalently, $\log(1+N(r)) - \log(1+N(s)) \geq -c(r^2 -s^2)$, which immediately gives us the desired estimate.
\end{proof}

Our next item on the agenda is to exhibit the fact that $N(r)$ controls the growth of $h(r)$. 

\begin{Prop}\label{frequency function proposition 3.12}
    For any $0<s<r$, we have $h(s) \lesssim (s/r)^2h(r)$. Consequently, positive $h(s)$ implies that $h(r)$ is also positive. Furthermore, $|\psi|^2(x) \lesssim h(r)/r^2$.
\end{Prop}

\begin{proof}
    Firstly, we claim that
    \begin{equation}\label{eq:frequency function.70}
     h(r)=e^{O(r^2)}\left(\dfrac{r}{s}\right)^2\exp \left(2 \int_{s}^{r} \dfrac{N(t)}{t}\,dt\right)h(s).
     \end{equation}
    From \eqref{eq:frequency function.51}, we have 
    \begin{align*}
        \dfrac{d}{dr}\, \log\,h(r) = \dfrac{h'(r)}{h(r)} = \dfrac{2+2N(r)}{r} + O(r).
    \end{align*}
    Integrating both sides of the above equation, we obtain
    \begin{align*}
        \log\, \dfrac{h(r)}{h(s)} &= \int^{r}_{s} \dfrac{2}{t}\, dt + 2\int^{r}_{s} \dfrac{N(t)}{t}\, dt + O(r^2-s^2) \\
        & =  \log\, \left(\dfrac{r}{s}\right)^2 + 2 \int^{r}_{s} \dfrac{N(t)}{t}\, dt + O(r^2)
    \end{align*}
    Exponentiate both sides and re-arrange to get \eqref{eq:frequency function.70}. As a result, the desired estimates follow directly from \eqref{eq:frequency function.70}.
\end{proof}

\begin{Cor}[\textbf{$N(r)$ controls the growth of $h(r)$}]\label{frequency function corollary 3.13} For any $0<s<r$, we have
$$e^{O(r^2)}\left(\dfrac{s}{r}\right)^{e^{O(r^2)}(2+2N(r))}h(r)\leq h(s) \leq e^{O(r^2)}\left(\dfrac{s}{r}\right)^{e^{O(r^2)}(2+2N(r))}h(r).$$  
\end{Cor}

\begin{proof}
    The estimate is a direct consequence of \eqref{eq:frequency function.70} and Corollary \ref{frequency function corollary 3.11}.
\end{proof}

Finally, we will exhibit the dependence of $N$ on the base point.

\begin{Prop}[\textbf{Dependence of $N$ on the base-point}]\label{frequency function proposition 3.14}
    Let $x \in Y$ and $r>0$ such that $N_x(10r) \leq 1$. For any $y\in B(x,r)$, we have $N_y(5r)\lesssim N_x(10r)$.
\end{Prop}

\begin{proof}
    We claim that for any $x,y \in Y$ and $r>0$, 
    \begin{equation}\label{eq:frequency function.71}
        h_x(r) \lesssim \dfrac{2r+d(x,y)}{r}h_y(2r+d(x,y)).
    \end{equation}
    Indeed, by Proposition \ref{frequency function proposition 3.12}, we already have 
    $$h(r) \leq \left(\dfrac{r}{2r}\right)^2 h(2r) = \dfrac{1}{4} \int_{\partial B(x,2r)} |\psi|^2.$$
    Remember that we always take $r<<1$. As a result, 
    $$r h_x(r) \leq \dfrac{r}{4}\int_{\partial B(x,2r)} |\psi|^{2} \lesssim \int_{B(x,2r)} |\psi|^2 \leq \int_{B(y,2r + d(x,y))}|\psi|^2.$$
    The RHS of the above estimate can be estimated further by \eqref{eq:frequency function.51} of Lemma \ref{frequency function lemma 3.8}, which gives us \eqref{eq:frequency function.71} as claimed.

    Now note that since $B(y,5r) \subseteq B(x,10r)$, we obviously have 
    $$5r\,H_y(5r) \leq 5r\, H_x(10r) \leq 10r\,H_x(10r).$$
    Hence, 
    \begin{equation}\label{eq:frequency function.72}
    N_y(5r) \leq \dfrac{10r\,H_x(10r)}{h_y(5r)} = N_x(10r)\, \dfrac{h_x(10r)}{h_y(5r)}.
    \end{equation}
    At the same time, by Corollary \ref{frequency function corollary 3.13} applied to $s :=r$ and $r:= 10r$ and the fact that $N_x(10r)\leq 1$,
    \begin{align}
        &\quad \quad e^{O(r^2)}\left(\dfrac{r}{10r}\right)^{e^O(r^2)(2+2N_x(10r))}h_x(10r)\leq h_x(r)\nonumber \\
        & \Rightarrow h_x(10r) \leq e^{-O(r^2)}10^{e^{O(r^2)}(2+2N_x(10r))}h_x(r) \lesssim h_x(r). \label{eq:frequency function.73}
    \end{align}
    Finally, by \eqref{eq:frequency function.71} and Proposition \ref{frequency function proposition 3.12} applied to $s:=2r+d(x,y)$ and $r:= 5r$,
    \begin{align}
        h_x(r) \lesssim 2\,\left(\dfrac{2r+d(x,y)}{5r}\right)^2 h_y(5r) \lesssim h_y(5r). \label{eq:frequency function.74}
    \end{align}
    Combine \eqref{eq:frequency function.72}, \eqref{eq:frequency function.73} and \eqref{eq:frequency function.74}, we have our estimate as claimed.
\end{proof}

\subsection{Control of the critical radius}
In this subsection, we will show that 

\begin{Prop}\label{frequency function proposition 3.15}
    There exists a constant $\omega >0$ such that for each  $(A,\psi, \epsilon) \in V_{k_0}(B(x,r)) \times \Gamma(W_{\mathfrak{s}_{3/2}}\otimes \mathscr{L}|_{B(x,r)}) \times (0,\infty)$ that is a solution of \eqref{eq:first compactness.2}, we have 
    $$\min\{1, |\psi|^{1/\omega}(x)\} \lesssim \rho(x).$$
\end{Prop}

Proposition \ref{frequency function proposition 3.15} follows from the following two propositions.

\begin{Prop}[\textbf{$|\psi|(x)$ controls the growth of $N$}]\label{frequency function proposition 3.16}
    Suppose we have that $0 < \omega <<1$ and $s\lesssim_{\omega} \min \{1, |\psi|^{1/\omega}(x)\}$. Then we must have $N(s) \lesssim \omega$.
\end{Prop}

\begin{Prop}\label{frequency function proposition 3.17}
    There are constants $\omega>0$ and $\rho_0 >0$ such that for any $(A,\psi,\epsilon) \in V_{k_0}(B(x,r)) \times \Gamma(W_{\mathfrak{s}_{3/2}}\otimes \mathscr{L}|_{B(x,r)}) \times (0,\infty)$ that is a solution of \eqref{eq:first compactness.2}, we have 
    $$N(50r) \leq \omega \Rightarrow \rho \geq \min \{r, \rho_0\}.$$
\end{Prop}
\noindent 
\textit{Proof of Proposition \ref{frequency function proposition 3.15} assuming Proposition \ref{frequency function proposition 3.16} and \ref{frequency function proposition 3.17}.} Let $0<\omega << 1$ and $0<c << 1$
be as in Proposition \ref{frequency function proposition 3.16} so that for
\[s=c\cdot\min\{1,|\psi|^{1/\omega}(x)\}\]
one has
\[N(s)\leq \omega.\]
We may assume $\rho_0$ to be as in Proposition \ref{frequency function proposition 3.17}, and additionally, $\omega$ also satisfies the assumption of Proposition \ref{frequency function proposition 3.17}. Hence, we have inequality 
\[\rho(x) \geq\min\left\{\rho_0,\dfrac{s}{50}\right\}=\min\{\rho_0,c,c\cdot|\psi|^{1/\omega}(x)\}.\]
Since $50$, $c$, and $\rho_0$ are constants, we may rewrite the latter inequality as follows 
$$\rho(x) \gtrsim\min\{1,|\psi|^{1/\omega}(x)\}$$ 
as claimed. \qed

We give a proof of Proposition \ref{frequency function proposition 3.16} first.\\
\noindent
\textit{Proof of Proposition \ref{frequency function proposition 3.16}.} Combine Lemma \ref{curvature controls lemma 2.3} with Proposition \ref{frequency function proposition 3.12} and Corollary \ref{frequency function corollary 3.13} and remember that we always take $0<s<r<<1$ to obtain
$$s^2 |\psi|^2(x) \lesssim h_x(s) \leq c^{2}\left(\dfrac{s}{r}\right)^{e^{O(r^2)}(2+2N_x(s))}r^2.$$
The above implies that
\begin{equation*}
    c^2|\psi|^{-2}(x) \geq s^2 \left(\dfrac{r}{s}\right)^{e^{O(r^2)}(2+2N(s))}r^{-2} = \left(\dfrac{r}{s}\right)^{-2 + 2e^{O(r^2)}+e^{O(r^2)}2N(s)}.
\end{equation*}
Since $2e^{O(r^2)}-2-O(r^2)\geq 0$, we can further esimate the above by
\begin{equation}\label{eq:frequency function.75}
    c^2|\psi|^{-2}(x)\geq \left(\dfrac{r}{s}\right)^{e^{O(r^2)}2N(s) + O(r^2)}.
\end{equation}
Taking the logarithm of \eqref{eq:frequency function.75}, we have 
\begin{align}
    &\quad \quad \quad (e^{O(r^2)}2N(s)+O(r^2))\log\,\left( \dfrac{r}{s}\right) \leq 2 \log\,(c|\psi|^{-1}(s))\nonumber  \\
    &\Rightarrow N(s) \leq \dfrac{\log\,(c|\psi|^{-1}(x))}{\log\,\left(\dfrac{r}{s}\right)}\, e^{-O(r^2)}-O(r^2) \lesssim \dfrac{\log\,(c|\psi|^{-1}(x))}{\log\,\left(\dfrac{r}{s}\right)} + O(r^2).\label{eq:frequency function.76}
\end{align}

\textit{Case 1:} $c|\psi|^{-1}(x) \leq 1$. 

By the above estimate, we have $N(s) \lesssim O(r^2)$ for all $0<s<r$. Since $s \lesssim_{\omega} \min\{1, |\psi|^{1/\omega}\}$, $s \lesssim_{\omega} 1$. So we can pick $r$ to be large enough in terms of $\omega$ to immediately get our desired estimate for $N(s)$.

\textit{Case 2:} $c|\psi|^{-1}(x) > 1$.

Note in this case, we have $c^{-1/\omega}|\psi|^{1/\omega}(x)<1$. Thus, $\omega \, c^{-1/\omega} |\psi|^{1/\omega}(x) < \omega$. Since $s\lesssim_{\omega} |\psi|^{1/\omega}$, we can choose $c$ such that $s \leq \omega\, c^{-1/\omega} |\psi|^{1/\omega}(x) < \omega := r$. As a result, we can use \eqref{eq:frequency function.76} to estimate
$$N(s) \lesssim \omega + O(\omega^2) \lesssim \omega,$$
as claimed. \qed

The rest of this subsection will be devoted to the proof of Proposition \ref{frequency function proposition 3.17}. We start with the following potentially well-known technical lemma.

\begin{Lemma}\label{frequency function lemma 3.18}
    Let $\delta \in (0,1)$. let $f$ be a non-negative function on $B^3(0,1)$ such that $f \in L^2((B^3(x,1-\delta)) \cap L^2(\partial B^3(x,1-\delta))$ (with respect to their natural respective measure). Let $c>0$. Suppose that $f$ satisfies the following
    \begin{itemize}
        \item $\displaystyle \norm {f}_{C^{0,1/4}(B(x,1-\delta))}\lesssim_{\delta} c$
        \item $\displaystyle \int_{\partial B(x,1-\delta)} f^2 \gtrsim \int_{\partial B(x,1)} f^2$
        \item $\displaystyle \norm{f^2}_{L^2(\partial B(x,1))} \neq 0$
    \end{itemize}
    If $c <<_{\delta} 1$, then $f$ never vanishes in $B(x,1-\delta)$. Consequently, there exists $\lambda >0$ such that $f(y) \geq \lambda$ for all $y \in B(x,1-\delta)$.
    \end{Lemma}

    \begin{proof}
        Since the statement is scale-invariant, without loss of generality, we may assume that the $f$ has unit $L^2-$norm on $\partial B(x,1)$. Let $y$ be any point inside $B(x,1-\delta)$. Consider $B(y,r(\delta))$ to be a neighborhood around $x$ inside $B(x,1-\delta)$. Let $\chi_{y,r(\delta)}$ to be the indicating function of $B(y,r(\delta))$. Note that
        $$\int_{B(y,r(\delta))} f^2 = \int_{B(x,1)} \chi_{y,r(\delta)}f^2.$$
        Firstly, we claim that there exists an $n \in \mathbb{N}^*$ such that
        \begin{equation}\label{eq:frequency function.77}
            \int_{\partial B(x,1-\delta)} f^2 \leq n \, \int_{B(x,1-\delta)} f^2.
        \end{equation}
        Indeed, otherwise and since $\norm{f}_{L^2(\partial B(x,1-\delta))} \gtrsim 1$, for all $n \in \mathbb{N}^*$ we have 
        \[\displaystyle \frac{\norm{f}_{L^2(B(x,1-\delta))}}{\norm{f}_{L^2(\partial B(x,1-\delta))}} < \dfrac{1}{n} \Longrightarrow 0 \leq \frac{\norm{f}_{L^2(B(x,1-\delta))}}{\norm{f}_{L^2(\partial B(x,1-\delta))}} <0, \quad \quad \text{ which is a contradiction.} \]
        By the same argument as above, we also have an $m \in \mathbb{N}^*$ such that
        \begin{equation}\label{eq:frequency function.78}
            \int_{B(x,1-\delta)} f^2 \leq m \int_{B(y,r(\delta))} f^2 = m\int_{B(0,1)}\chi_{x,r(\delta)}f^2
        \end{equation}
        With \eqref{eq:frequency function.77} and \eqref{eq:frequency function.78} in mind, by the hypothesis, we have 
        \[\int_{B(x,1)} \chi_{y,r(\delta)}f^2 \gtrsim \int_{\partial B(x,1-\delta)} f^2 \gtrsim 1.\]
        At the same time, as $\delta \to 1$, $\chi_{y,r(\delta)}f^2(\bullet) \to f^2(y)$ almost everywhere. Furthermore, 
        \[ \sup_{B(x,1)} |\chi_{y,r(\delta)}f^2| = \sup_{B(x,1-\delta)}|f|^2 \leq \norm{f}^{2}_{C^{0,1/4}(B(x,1-\delta))} \lesssim_{\delta}c^2 <<_{\delta} 1.\]
        This means that the family of function $\{\chi_{y,r(\delta)}f^2\}_{\delta \in (0,1)}$ is uniformly bounded. As a result, by the bounded convergence theorem, for each $y \in B(0,1-\delta)$ we must have that 
        $$f^2(y)\cdot vol(B(x,1)) \gtrsim 1.$$
        In other words, there is a $\lambda >0$ such that $f \geq \lambda$ inside $B(x,1-\delta)$ as desired. 
    \end{proof}

Now we give an $L^6-$bound for $P_A P_A^* \psi$, where $(A,\psi)$ is a solution to the $3/2-$monopole Seiberg-Witten equations. The proof is a variation of the proof in Lemma \ref{curvature controls lemma 2.5} and Proposition \ref{curvature controls proposition 2.7}. Recall that on a geodesic ball $B(x,r)$, by Uhlenbeck's local slice theorem, there are constants $c,k_0>0$ (possibly depending on $x,r,g$) such that any connection $A = A_0 + a$ and $\norm{F_A}_{L^6(B(x,r))}\leq k_0$, $A$ is gauge equivalent to $A_0 + b$ where 
\[\norm{b}_{L^6_1(B(x,r))} \leq c \norm{F_A}_{L^6(B(x,r))}, \quad \quad d^* b = 0.\]
Assume $r=1$. Denote all such $U(1)-$connections satisfy the above by $V_{k_0,6}$. For any solution $(A,\psi) \in V_{k_0,6} \times \Gamma(W_{\mathfrak{s}_{3/2}}\otimes \mathscr{L}|_{B(x,1)})$, we first claim that

\begin{Lemma}\label{frequency function lemma 3.19}
    $\norm{P^*_A \psi}_{L^6(B(x,1))} = O(1)$.
\end{Lemma}

\begin{proof}
    By the Minkowski inequality, we have
    \begin{equation}\label{eq:frequency function.79}
        \int_{B(x,1)}|P^*_A \psi|^6 = \int_{B(x,1)}|\nabla^*_A \psi|^6 \lesssim \int_{B(x,1)} |\nabla^*_{A_0}\psi|^6 + |a^*\psi|^6.
    \end{equation}
    Since $\nabla^*_{A_0}$ is a first-order differential operator, we have 
    \begin{equation}\label{eq:frequency function.80}
        \int_{B(x,1)}|\nabla^*_{A_0}\psi|^6 \lesssim \norm{\psi}^6_{L^6_1(B(x,1))}.
    \end{equation}
    Note that we have already established that $\norm{\psi}_{L^\infty(B(0,1))} = O(1)$ (cf. Lemma \ref{curvature controls lemma 2.3}). By the elliptic estimate, Minkowski inequality, and the Sobolev multiplication theorem, we have
    \begin{align}\label{eq:frequency function.81}
        \norm{\psi}^6_{L^6_1} &\lesssim \norm{Q_{A_0}\psi}^6_{L^6}+\norm{\psi}^6_{L^6} \nonumber \\
        & \lesssim \norm{a\cdot \psi}^6_{L^6} + O(1) \lesssim \norm{a}^6_{L^6_1}\cdot \norm{\psi}^6_{L^4} + O(1) = O(1).
    \end{align}
    By Lemma \ref{first compactness lemma 1.3} and a similar argument as above, we also have $\norm{a^* \psi}_{L^6(B(0,1))} = O(1)$. As a result, combining \eqref{eq:frequency function.79}, \eqref{eq:frequency function.80} and \eqref{eq:frequency function.81} yields for us the desired estimate as claimed.
\end{proof}

Next, we control $\norm{P^*_A \psi}_{L^6_1(B(x,1))}$. 

\begin{Lemma}\label{frequency function lemma 3.20}
    For any solution $(A,\psi) \in V_{k_0,6} \times \Gamma(W_{\mathfrak{s}_{3/2}}\otimes \mathscr{L}|_{B(x,1)})$, we have
    \[\norm{P^*_A \psi}_{L^6_1(B(x,1))}=O(1).\]
\end{Lemma}

\begin{proof}
    We use the Weitzenb\"ock-type formula in Lemma \ref{curvature controls lemma 2.4}, the elliptic estimate with the Minkowski inequality, and Lemma \ref{frequency function lemma 3.19} to estimate
    \begin{align}\label{eq:frequency function.82}
        \norm{P^*_A \psi}^6_{L^6_1(B(x,1))} & \lesssim \norm{D_{A_0}P^*_A\psi}^6_{L^6}+ \norm{P^*_A\psi}^6_{L^6} \nonumber \\
        & \lesssim \norm{\psi}^6_{L^6} + \norm{\psi}^6_{L^\infty}\norm{F_A}^6_{L^6}+\norm{a\cdot P^*_A\psi}^6_{L^6} + O(1)
    \end{align}
    Since the first two terms on the RHS of \eqref{eq:frequency function.82} are $O(1)$, it remains for us to control the third term. By the Sobolev multiplication theorem, we have
    \[ \norm{a \cdot P^*_A \psi}^6_{L^6} \lesssim \norm{a}^6_{L^6_1} \norm{P^*_A \psi}^6_{L^4}.\]
    From either Lemma \ref{frequency function lemma 3.19} or in the proof of Lemma \ref{curvature controls lemma 2.5}, we have already established the second factor on the RHS of the above estimate is $O(1)$. Therefore, $\norm{P^*_A \psi}_{L^6_1(B(x,1))} = O(1)$ as desired.
\end{proof}

Finally, we establish a control for $\norm{P_AP^*_A\psi}_{L^6(B(x,1))}$.

\begin{Lemma}\label{frequency function lemma 3.21}
    Let $(A,\psi) \in V_{k_0,6} \times \Gamma(W_{\mathfrak{s}_{3/2}}\otimes \mathscr{L}|_{B(x,1)})$ be a solution of the RS-SW equations. Then
    \[\int_{B(x,1)}|P_AP^*_A \psi|^6 = O(1).\]
\end{Lemma}

\begin{proof}
    By the Minkowski inequality and the fact that $\nabla_{A_0}$ is a first order differential operator, we have
    \begin{align*}
        \int_{B(x,1)} |P_A P^*_A \psi|^6 &\lesssim \int_{B(x,1)} |\nabla_{A_0}P^*_A|^6 + |a\cdot P^*_A \psi|^6 \\
        & \lesssim \norm{P^*_A\psi}^6_{L^6_1} + \norm{a \cdot P^*_A \psi}^6_{L^6}.
    \end{align*}
    By Lemma \ref{frequency function lemma 3.20}, the first term on the RHS of the above estimate is $O(1)$. In the proof of Lemma \ref{frequency function lemma 3.20}, the second term is $O(1)$. Hence, we obtain the desired estimate as claimed.
\end{proof}

With the above technical lemmas out of the way, the proof of Proposition \ref{frequency function proposition 3.17} will be divided into several parts.

\textit{Part 1.} In this part, we prove the following proposition.

\begin{Prop}\label{frequency function proposition 3.22}
    Let $B_r$ be a geodesic ball of radius $r>0$ in $Y$. Let $(A,\psi,\epsilon) \in V_{6,k_0} \times \Gamma(W_{\mathfrak{s}_{3/2}}\otimes \mathscr{L}|_{B_r}) \times (0,\infty)$ be a solution to \eqref{eq:first compactness.2}. We denote
    \begin{itemize}
        \item $\mathfrak{c} = \dfrac{r^4 \norm{\nabla_A \psi}^8_{L^2(B_r)}}{\norm{\psi}^8_{L^2(\partial B_r)}}+O(r^2)$
        \item $\tau = \dfrac{\epsilon}{\norm{\psi}_{L^2(\partial B_r)}}$
    \end{itemize}
    Let $\delta \in (0,1)$ and $\kappa > 0$. If we have $r^{1/2}\norm{F_A}_{L^2(B_r)} \leq 1$ and $\mathfrak{c} <<_{\delta, \kappa} 1$ and $\tau <<_{\delta, \kappa} 1$, then $r^{1/2}\norm{F_A}_{L^2(B_{r(1-\delta)})} \leq \kappa$.
\end{Prop}

\begin{proof}
    Without loss of generality, because of scale invariance, it is enough if we can prove the proposition for $r = 1$ and $\norm{\psi}_{L^2(\partial B_r)} = 1$. With respect to the re-normalization, the proof will be divided into several steps.
    
    \textit{Step 1.} We claim that if $\mathfrak{c} \leq 1$ and $\tau \leq 1$, then $|\psi|^2(x)\lesssim d(x,\partial B_1)^{-2}$ for all $x \in B_1$. This implies that there is $\Lambda(\delta) = O(1/\delta) >0$ such that $|\psi| \leq \Lambda(\delta)$ on $B_{1-\delta}$.\\
    \noindent
    \textit{Proof of Step 1.} We apply Lemma \ref{first compactness lemma 1.6} to $f:= G(x,\bullet)-$the Green's function of $\Delta$ and $U:= B_1 \setminus B_\sigma$ and let $\sigma \to 0$ to get
    \begin{align*}
        |\psi|^2(x) \lesssim \int_{B_1} G(x,\bullet) |\psi|^2 &+ \int_{\partial B_1} G(x,\bullet) \partial_r |\psi|^2 + |\nabla G(x,\bullet)| \cdot |\psi|^2+ \\
        & + \int_{B_1}G(x,\bullet)\la P_AP^*_A \psi, \psi \ra.
    \end{align*}
    Since $G(x,\bullet) \lesssim d(x,\bullet)^{-1}$ and $|\nabla G(x,\bullet)| \lesssim d(x,\bullet)^{-2}$, by the Cauchy-Schwarz inequality we can estimate $|\psi|^2(x)$ further by
    \begin{align}\label{eq:frequency function.83}
        |\psi|^2(x) \lesssim \int_{B_1} \dfrac{|\psi|^2}{d(x,\bullet)} + \int_{B_1}\dfrac{|P_A P^*_A \psi|^2}{d(x,\bullet)} + \dfrac{1}{d(x,\partial B_1)} \int_{\partial B_1} \partial_r |\psi|^2 + \dfrac{1}{d(x,\partial B_1)^2}.
    \end{align}
    In the above estimate, via the H\"older inequality, the first and second terms on the RHS can be estimated by
    \begin{align*}
        \int_{B_1} \dfrac{|\psi|^2}{d(x,\bullet)} + \int_{B_1}\dfrac{|P_A P^*_A \psi|^2}{d(x,\bullet)} &\leq \norm{\dfrac{1}{d(x,\bullet)}}_{L^{3/2}(B_1)}\left(\norm{\psi}_{L^6(B_1)}+ \norm{P_A P^*_A \psi}_{L^6(B_1)}\right).
    \end{align*}
    The RHS is clearly $O(1)$ via Lemma \ref{frequency function lemma 3.21} and the fact $L^{3/2}-$norm of $d(x,\bullet)^{-1}$ is also $O(1)$. Since $d(x,\partial B_1)^{-1}=O(1)$, it remains for us to estimate the integral of $\partial_r |\psi|^2$ over the boundary. To do that, we apply Lemma \ref{first compactness lemma 1.6} again to $f:=1$ and $U:= B_1$,
    \begin{align*}
        \int_{\partial B_1 } \partial_r |\psi|^2 \lesssim \int_{B_1}|P_AP^*_A \psi|^2 + |\psi|^2 + |\nabla_A \psi|^2 + \epsilon^{-2}|\mu(\psi)|^2.
    \end{align*}
    The first three terms on the RHS are $O(1)$ because of Lemma \ref{frequency function lemma 3.21}, Lemma \ref{curvature controls lemma 2.3}, and the fact that $\mathfrak{c} \leq 1$. Since $\tau = \epsilon$ due to re-normalization and $\tau \leq 1$, 
    \[\int_{B_1}\epsilon^{-2}|\mu(\psi)|^2 \leq \int_{B_1}  \epsilon^{-4}|\mu(\psi)^2| = \norm{F_A}^2_{L^2(B_1)} \leq 1.\]
    As a result, the entire RHS of the above estimate is $O(1)$. Hence, from \eqref{eq:frequency function.83}, we have the bound on $|\psi|$ as claimed.

    \textit{Step 2.} We claim that $\norm{\psi}_{C^{0,1/4}(B_{1-\delta})}\lesssim_{\delta} \mathfrak{c}^{1/8}$.\\
    \noindent
    \textit{Proof of Step 2.} By Morrey's inequality, we have
    \[ \norm{\psi}_{C^{0,1/4}(B_{1-\delta})} \lesssim \norm{\psi}_{L^4_1(B_{1-\delta})}.\]
    Since $\norm{\psi}_{L^4(B_{1-\delta})} = O(1)$, it is enough for us to estimate $\norm{\nabla_A \psi}_{L^4(B_{1-\delta})}$. We apply the Gagliardo-Nirenberg interpolation inequality
    \[\norm{f}_{L^4(B_{1-\delta})}\lesssim_{\delta} \norm{\nabla f}^{3/4}_{L^2(B_{1-\delta})}\norm{f}^{1/4}_{L^2(B_{1-\delta})} + \norm{f}_{L^2(B_{1-\delta)}}\]
    to $f:=|\nabla_A \psi|$ and combine with the Minkowski inequality and Kato's inequality to obtain
    \begin{align}\label{eq:frequency function.84}
        \norm{\nabla_A \psi}_{L^4}^8 & \lesssim_{\delta} \norm{\nabla|\nabla_A\psi|}^6_{L^2}\norm{\nabla_A\psi}^2_{L^2} + \norm{\nabla_A\psi}^{8}_{L^2} \nonumber \\
        &\lesssim_{\delta}\norm{\nabla_A^2\psi}^6_{L^2}\norm{\nabla_A\psi}^2_{L^2} + \norm{\nabla_A\psi}^{8}_{L^2}.
    \end{align}
    By Proposition \ref{curvature controls proposition 2.7}, we have
    \[\norm{\nabla^2_A \psi}^6_{L^2} \lesssim_{\delta} \norm{\psi}^6_{L^2}+\norm{\nabla_{A}\psi}^6_{L^2}+\norm{F_A}^6_{L^2}\norm{\psi}^6_{L^\infty}+O(1).\]
    As a result, we estimate \eqref{eq:frequency function.84} further by
    \begin{align} \label{eq:frequency function.85}
        \norm{\nabla_A\psi}^8_{L^4}  \lesssim_{\delta} \norm{\psi}^6_{L^2}\norm{\nabla_A\psi}^2_{L^2} &+ \norm{\nabla_A\psi}^8_{L^2}+\norm{F_A}^6_{L^2}\norm{\psi}^6_{L^\infty}\norm{\nabla_A\psi}^2_{L^2} +  \nonumber \\
        & + O(1)\norm{\nabla_A\psi}^2_{L^2}+ \norm{\nabla_A\psi}^8_{L^2} \lesssim_{\delta} \mathfrak{c}.
    \end{align}
    The above gives us the desired bound for the H\"older norm of $\psi$.

    \textit{Step 3.} We claim that if $\mathfrak{c}<<_{\delta}1$, then $\psi$ never vanishes in $B_{1-\delta}$. In particular, there is a $\lambda>0$ such that $|\psi|(x)\geq \lambda$ for all $x\in B_{1-\delta}$.\\
    \noindent
    \textit{Proof of Step 3.} The statement is obtained by a direct application of Lemma \ref{frequency function lemma 3.18} to $f:=|\psi|$.

    \textit{Step 4.} In this step, we claim that if $\mathfrak{c}\leq 1$, then $\norm{\mu(\psi)}_{L^{\infty}(B_{1-\delta})}\lesssim_{\delta}\tau^{1/8}$.\\
    \noindent
    \textit{Proof of Step 4.} The statement is true via the previous steps, \eqref{eq:frequency function.85}, Kato's inequality, the Gagliardo-Nirenberg inequality, and Morrey's inequality.

    \textit{Step 5.} We assemble the ingredients established and give proof to the proposition. Let $\la \bullet, \bullet\ra : T^*Y \otimes T^* Y \to \RN$ be the contraction operator. Then combine with the Weitzenb\"ock formula for the Rarita-Schwinger operator (cf. Lemma \ref{weitzenbock formula}), we have
    \begin{align}\label{eq:frequency function.86}
        \nabla^*\nabla \mu(\psi) &= 2 \mu(\nabla^*_A \nabla_A \psi, \psi) - 2 \la\mu(\nabla_A \psi, \nabla_A \psi)\ra \nonumber \\
        &=2\mu\left(\dfrac{4}{3}P_AP^*_A\psi - \dfrac{s}{4}\psi - \pi(F_A\psi)-\pi(1\otimes Ric)\psi,\psi\right)-2\la\mu(\nabla_A\psi, \nabla_A\psi)\ra. 
    \end{align}
    By the definition of the quadratic map $\mu$ and some straight-forward calculations, we note
    \begin{align*}
        \mu(P_AP^*_A\psi,\psi) &= P_AP^*_A\psi \psi^* - \frac{1}{2}\la P_A P^*_A \psi, \psi \ra 1 \\
        \mu(\pi(\mu(\psi)\psi),\psi) &= \mu(\mu(\psi)\psi,\psi) + 2\mu\left(\iota\left(\sum_i \alpha_i\psi_i\right),\psi\right) \\
        & = \dfrac{1}{2}|\psi|^2\mu(\psi)+\mu(\psi)\circ\mu(\psi)-\frac{1}{2}tr(\mu(\psi)\circ\mu(\psi))1 + \\
        &+2\mu\left(\iota\left(\sum_i \alpha_i\psi_i\right),\psi\right),
    \end{align*}
    where $\mu(\psi) = \alpha_1 I + \alpha_2 J + \alpha_3 K$ with $I, J, K$ being identified with $\gamma(e_1), \gamma(e_2), \gamma(e_3)$, respectively; and $\displaystyle \psi = \sum_i \psi_i \otimes e_i$. As a result, we can rewrite \eqref{eq:frequency function.86} as
    \begin{align}\label{eq:frequency function.87}
        \nabla^*\nabla \mu(\psi)  = \dfrac{8}{3}P_AP^*_A\psi\psi^*&-\dfrac{4}{3}\la P_A P^*_A \psi, \psi \ra 1-\dfrac{s}{2}\mu(\psi) - \epsilon^{-2}|\psi|^2\mu(\psi) +\nonumber \\
        & - 2\epsilon^{-2}\mu(\psi)\circ\mu(\psi)+ \epsilon^{-2}tr(\mu(\psi)\circ\mu(\psi))1 + \nonumber \\
        &-4\epsilon^{-2}\mu\left(\iota\left(\sum_i \alpha_i\psi_i\right),\psi\right)+\nonumber \\
        & - \mu(\pi(1\otimes Ric)\psi,\psi) - 2 \la \mu(\nabla_A \psi, \nabla_A\psi)\ra.
    \end{align}
    Rearranging the RHS of \eqref{eq:frequency function.87} yields for us
    \begin{align} \label{eq:frequency function.88}
        \nabla^*\nabla \mu(\psi) +\left(\epsilon^{-2}|\psi|^2+\dfrac{s}{2}\right)\mu(\psi) & = \dfrac{8}{3}P_AP^*_A \psi \psi^* - \dfrac{4}{3}\la P_A P^*_A \psi, \psi\ra 1 + \nonumber \\
        & - 2\epsilon^{-2}\mu(\psi)\circ\mu(\psi) + \epsilon^{-2}tr(\mu(\psi)\circ\mu(\psi))1+\nonumber \\
        & - 4\epsilon^{-2}\mu\left(\iota\left(\sum_i \alpha_i\psi_i\right),\psi\right)+ \nonumber \\
        & - \mu(\pi(1\otimes Ric)\psi,\psi) - 2 \la \mu(\nabla_A \psi, \nabla_A\psi)\ra.
    \end{align}
    Consider $\chi$ to be the cut-off function whose support is is $B_{1-\delta/2}$ and equals to $1$ in $B_{1-\delta}$. We multiply both sides of \eqref{eq:frequency function.88} by $\chi$ and pair with $\mu(\psi)$ to integrate to obtain
    \begin{align}\label{eq:frequency function.89}
        \int \chi|\nabla \mu(\psi)|^2 &+\chi\left(\epsilon^{-2}|\psi|^2+\dfrac{s}{2}\right)|\mu(\psi)|^2 \lesssim_{\delta} \nonumber \\
        & \lesssim_{\delta} \int |P_A P^*_A \psi|\cdot |\psi| \cdot |\mu(\psi)| + \chi \epsilon^{-2}|\mu(\psi)|^2\cdot |\mu(\psi)| + \nonumber \\
        & + 4\chi \epsilon^{-2}|\psi|^2|\mu(\psi)|^2 + |Ric|\cdot |\psi|^2 |\mu(\psi)| + |\nabla_A \psi|^2 |\mu(\psi)| + \nonumber \\
        &+ |\nabla_A \psi|^2 |\mu(\psi)|.
    \end{align}
    Since $\epsilon = \tau$, using the estimate in Step 4, we can estimate the RHS of \eqref{eq:frequency function.89} further and rearrange the inequality to get
    \begin{align}\label{eq:frequency function.90}
        &\int \chi \left(\epsilon^{-2}|\psi|^2+\dfrac{s}{2}-c_{\delta}\epsilon^{-2+1/8}-4\epsilon^{-2}|\psi|^2\right)|\mu(\psi)|^2 \lesssim_{\delta} \nonumber \\
        &\lesssim_{\delta} \int (|P_AP^*_A \psi|\cdot |\psi| + |Ric|\cdot |\psi|^2 + |\nabla_A \psi|^2 + |\psi| \cdot |\nabla_A \psi|)|\mu(\psi)|
    \end{align}
    Using the estimates established in Step 1, Step 2, and Step 3 to further estimate \eqref{eq:frequency function.90}, we arrive at
    \begin{equation*}
        \int \chi  |F_A|^2 = \int \chi\epsilon^{-4}|\mu(\psi)|^2 \lesssim_{\delta} \dfrac{1}{\lambda^2 + \frac{s}{2}\epsilon^2-c_{\delta}\epsilon^{1/8}-4\Lambda^2(\delta)}(\mathfrak{c}+\mathfrak{c}^{1/8}).
    \end{equation*}
    The above estimate implies that if both $\epsilon = \tau$ and $\mathfrak{c}$ are small, then so is $\norm{F_A}_{L^2(B_{1-\delta})}$. This completes the proof of the claim made in the proposition.
\end{proof}

\textit{Part 2.} In this part, we prove the following corollary. Recall that for a given point $x \in Y$, we denote $\rho(x)$ by the critical radius of a connection around $x$. The corollary which we will prove roughly tells us that if the frequency function $N_x(\rho)$ is small, so is the density of the curvature $F_A$ on $B(x,\rho(x)/2)$.

\begin{Cor}\label{frequency function corollary 3.23}
    Suppose $(A,\psi,\epsilon)\in V_{k_0,6} \times \Gamma(W_{\mathfrak{s}_{3/2}}\otimes \mathscr{L}|_{B(x,\rho(x))}) \times (0,\infty)$ be a solution to \eqref{eq:first compactness.2}. For any $\kappa >0$, if $\rho <<_{\kappa} 1$ and $N(\rho) <<_{\kappa} 1$, then $\rho \norm{F_A}^2_{L^2(B(x,\rho(x)/2))} \leq \kappa$.
\end{Cor}

\begin{proof}
    Apply Proposition \ref{frequency function proposition 3.22} to $r:=\rho$, $\delta :=1/2$ and $\kappa > 0$. If $\rho \norm{F_A}^2_{L^2(B(x,\rho))} \leq \kappa$, then there is nothing to prove. Assume $\rho \norm{F_A}^2_{L^2(B(x,\rho))} > \kappa$. Note that
    \begin{align*}
        \rho \norm{F_A}^2_{L^2(B(x,\rho))} \epsilon^2 \leq \rho H(\rho) + \dfrac{4}{3}\rho \la P_A P^*_A \psi, \psi \ra_{L^2(B(x,\rho))} \lesssim \rho H(\rho) + O(\rho^2).
    \end{align*}
    As a result,
    \begin{align*}
        \tau^2 = \dfrac{\epsilon^2}{\norm{\psi}^2_{L^2(\partial B(x,\rho))}} &\lesssim \left(\rho\norm{F_A}^2_{L^2(B(x,\rho))}\right)^{-1}\left(N(\rho) + \dfrac{O(\rho^2)}{\norm{\psi}^2_{L^2(\partial B(x,\rho))}}\right) \\
        & \lesssim_{\kappa} N(\rho) + \dfrac{O(\rho^2)}{\norm{\psi}^2_{L^2(\partial B(x,\rho))}}.
    \end{align*}
    The above estimate shows that if $\rho$ and $N(\rho)$ are small, then so is $\tau$. Furthermore, $\mathfrak{c}$ is also small. Therefore, we deduce our assertion by Proposition \ref{frequency function proposition 3.22}.
\end{proof}

\textit{Part 3.} In this part, we give a proof of Proposition \ref{frequency function proposition 3.17}.\\
\noindent
\textit{Proof of Proposition \ref{frequency function proposition 3.17}} We argue by contradiction. Suppose that for any arbitrarily small $\kappa >0$, there is a solution $(A,\psi, \epsilon) \in V_{k_0,6}\times \Gamma(W_{\mathfrak{s}_{3/2}}\otimes \mathscr{L}|_{B(x,\rho(x))}) \times (0,\infty)$ such that $N_x(50\rho(x)) \leq \kappa$ but $\rho(x) \leq \kappa$.

\textit{Claim 1:} There is an $x' \in B(x,2\rho(x))$ such that 
\[ \rho(x') \leq \rho(x), \quad \quad \rho(x') \leq 2 \min_{y \in B(x',\rho(x'))} \{\rho(y)\}.\]
We show the existence of $x'$ by an inductive argument. The construction is a toss-pick procedure. To get a feel of the construction, we start with some base cases. Let $x_0 := x$. Obviously, $\rho(x_0) = \rho(x)$. If $\rho(x_0) \leq 2 \min\{ \rho(y) : y \in B(x_0,\rho(x_0))\}$, then we take $x':=x_0$. Otherwise, we pick another point $x_1 \in B(x_0, \rho(x_0))$ such that $\rho(x_1) < \rho(x)/2$. This implies that $\rho(x_1) < \rho(x)$. Now if $\rho(x_1) \leq 2 \min \{\rho(y) : y \in B(x_1,\rho(x_1))\}$, then we take $x':=x_1$. Otherwise, we pick an $x_2 \in B(x_1, \rho(x_1))$ such that $\rho(x_2)<\rho(x_1)/2<\rho(x)/2^2$. This implies that $\rho(x_2) < \rho(x)$. By repeating this construction, we obtain a sequence $\{x_n\}$ where $\rho(x_n) < \rho(x)/2^n$. If this sequence is infinite, then $\rho(x_n) \to 0$ as $n$ approaches infinity. But this is simply impossible because the function $\rho(\bullet)$ is bounded below away from zero for a particular solution $(A,\psi, \epsilon)$. As a result, the construction must stop at some $x_n$ where we have 
\[ \rho(x_n) \leq \rho(x), \quad \quad  \rho(x_n) \leq 2 \min_{y \in B(x_n, \rho(x_n))} \{\rho(y)\}.\]
Furthermore, we have $x_n \in B(x,2\rho(x))$ because
\[d(x_n,x) \leq \sum_{j=0}^{n-1} d(x_{j+1},x_j) \leq  \sum_{j=0}^{n-1} \rho(x_j) <  \sum_{j=0}^{n-1} \dfrac{\rho(x)}{2^j} < 2 \rho(x).\]

\textit{Claim 2:} Let $x'$ be as in Claim 1. For every $y \in B(x',\rho(x'))$, we have 
\[\rho(y) \lesssim \kappa,\quad \quad  N_y(\rho(y)) \lesssim \kappa.\]
Note that the function $r \mapsto r^{1/2}\norm{F_A}_{L^2(B(x,r))}$ is monotone. From Claim 1, we also have $\rho(x')$ small. If $\rho(x')\norm{F_A}^2_{L^2(B(x',\rho(x')))} < 1$, then we can add a small emough amount to $\rho(x')$ so that the newly obtained quantity $r'$ is still within $(0,r_0]$ and $r'\norm{F_A}^{2}_{L^2(B(x',r'))} = 1$. This would lead to a contradiction with the definition of the critical radius $\rho(x')$. Thus, $\rho(x')\norm{F_A}^2_{L^2(B(x',\rho(x'))}=1$. With this being understood, for any $y \in B(x', \rho(x'))$, we clearly have $B(x',\rho(x'))\subset B(y, 2\rho(x'))$. As a result,
\[\int_{B(y,2\rho(x'))}|F_A|^2 \geq \int_{B(x',\rho(x'))} |F_A|^2 = \dfrac{1}{\rho(x')} > \dfrac{1}{2\rho(x')}.\]
By the monotonicity, the above implies that $\rho(y) < 2\rho(x') \leq 2\rho(x) \lesssim \kappa$. At the same time, since $y \in B(x,5\rho(x))$, by Corollary \ref{frequency function corollary 3.11} and Proposition \ref{frequency function proposition 3.14} applied to $r:=5\rho(x)$, we have
\[N_y(\rho(y)) \leq e^{O(\kappa^2)}N_y(25\rho(y))+O(\kappa^2) \lesssim N_x(50\rho(x))+O(\kappa^2) \lesssim \kappa.\]

\textit{Claim 3:} There exists a finite number of points $\{y_1,\cdots, y_n\}\subset B(x',\rho(x'))$, where $n$ is large enough independent of $x'$ such that 
\[B(x',\rho(x'))\subset \bigcup_{i=1}^{n} B(y_i, \rho(y_i)/2).\]
From Claim 1, we know that for all $y \in B(x',\rho(x'))$, $\rho(y)\geq \rho(x')/2$. Consider a cover of $B(x',\rho(x'))$ given by
\[\bigcup_{y \in B(x',\rho(x'))} B(y, \rho(x')/4).\]
Since $B(x',\rho(x'))$ is compact, there is a finite sub-cover
\[\bigcup_{i=1}^{n} B(y_i, \rho(x')/4),\]
where $n$ is independent of $x'$. Obviously for each $i$, we have $B(y_i,\rho(x')/4)\subset B(y_i, \rho(y_i)/2)$. Hence, we obtained a finite cover of $B(x',\rho(x'))$ as desired.

Now, we combine the three claims to get to a contradiction. By Claim 2, Corollary \ref{frequency function corollary 3.23} and Claim 1, for each $y_i$ as in Claim 3, we have
\[\int_{B(y_i, \rho(y_i)/2)} |F_A|^2 \leq \dfrac{\kappa}{\rho(y_i)} \lesssim \dfrac{\kappa}{\rho(x')}.\]
Thus, from Claim 3, we deduce that
\[\int_{B(x',\rho(x'))} |F_A|^2 \leq \sum_{i=1}^{n} \int_{B(y_i,\rho(y_i)/2)}|F_A|^2 \lesssim \dfrac{n \kappa}{\rho(x')}.\]
This implies that $1=\rho(x')\norm{F_A}^2_{L^2(B(x',\rho(x'))} \lesssim \kappa$. But $\kappa$ is arbitrarily small; this is a contradiction.\qed

\section{The second compactness theorem}
In this section, we prove the second compactness theorem regarding the moduli space of solutions of \eqref{eq:first compactness.2}. The theorem concerns the behavior of a sequence of solutions of the original formulation of the three-dimensional RS-SW equations where the $3/2-$spinors become very large, yet the curvatures of the $U(1)-$connections are well-controlled. Roughly, in an appropriate topology, such a sequence converges away from a singular set $Z$ on $Y$. In some sense, one can view the limiting data as a boundary condition of the three-dimensional RS-SW equations. 

\begin{Th}\label{second compactness theorem}
    Let $\{(A_n, \psi_n, \epsilon_n)\} \subset \mathcal{A}(\mathscr{L}) \times \Gamma(W_{\mathfrak{s}_{3/2}}\otimes \mathscr{L}) \times (0,\infty)$ be a sequence of solutions of \eqref{eq:first compactness.2} where $\{F_{A_n}\}$ is uniformly bounded in $L^6-$norm. If $\limsup \epsilon_n = 0$, then
    \begin{enumerate}
        \item There is a closed nowhere-dense subset $Z \subset Y$, a connection $A \in \mathcal{A}(\mathscr{L}|_{Y\setminus Z})$ and a $3/2-$spinor $\psi \in \Gamma(Y\setminus Z, W_{\mathfrak{s}_{3/2}}\otimes \mathscr{L})$ such that $(A,\psi, 0)$ solves \eqref{eq:first compactness.2}. Furthermore, after to passing through a subsequence, $|\psi_n|$ converges to $|\psi|$ in $C^{0,\alpha}-$topology. Specifically, $Z = |\psi|^{-1}(0)$.
        \item On $Y\setminus Z$, up to gauge transformations and after passing through a subsequence, $A_n$ converges weakly to $A$ in $L^{2}_{1,loc}$ and $\psi_n$ converges weakly to $\psi$ in $L^2_{2,loc}$.
    \end{enumerate}
\end{Th}

\begin{proof}
    Note that if $(A,\psi, \epsilon)$ is a solution and if the curvature $F_A$ is universally bounded, from Section $4$ (cf. Lemma \ref{curvature controls lemma 2.3}), the $L^{\infty}-$norm of $\psi$ is universally bounded. In this situation, we claim that there exists an $\alpha > 0$ such that $\norm{\psi}_{C^{0,\alpha}} = O(1)$. Equivalently, we have to show that for $x, y \in Y$ and $x\neq y$, there is an $\alpha > 0$ such that
    \[ \dfrac{||\psi|(x) - |\psi|(y)|}{d(x,y)^{\alpha}}\]
    is uniformly controlled. Indeed, we take $\omega > 0$ as in Proposition \ref{frequency function proposition 3.15}. From the proof of Proposition \ref{frequency function proposition 3.15}, we may assume $0 < \omega << 1$. Furthermore, without loss of generality, we also assume that $0\neq |\psi|(x) \geq |\psi|(y)$ and $d(x,y) \leq \omega$. By Proposition \ref{frequency function proposition 3.15}, we have that 
    \[\rho(x) \gtrsim \min\{1, |\psi|^{1/\omega}(x)\}.\] There are two cases to consider.

    \textit{Case 1:} $d(x,y)^{1/2}> \rho(x)/2$.

    In this scenario, there are two sub-cases. If $|\psi|(x) \geq 1$, then $\rho(x) \gtrsim 1$. As a result, $d(x,y) \gtrsim 1$ also. Then
    \[\dfrac{||\psi|(x) - |\psi|(y)|}{d(x,y)^{1/2}}  \lesssim |\psi|(x) \leq \norm{\psi}_{L^\infty} = O(1).\]
    If $|\psi|(x) < 1$, then we have $|\psi|(y) \leq |\psi|(x) \lesssim \rho(x)^{\omega} \lesssim d(x,y)^{\omega/2}$. As a result, a similar estimate above yields for us
    \[\dfrac{||\psi|(x) - |\psi|(y)|}{d(x,y)^{\omega/2}} = O(1).\]

    \textit{Case 2:} $d(x,y)^{1/2}\leq \rho(x)/2$.

    Since we assume $d(x,y) \leq \omega <<1$, $d(x,y) \leq d(x,y)^{1/2} \leq \rho(x)/2$. Consider a closed ball $B_{\rho(x)/2}$ on $Y$ of radius $\rho(x)/2$. By Kato's inequality, Morrey's inequality, one of the Sobolev embedding theorems, and Proposition \ref{curvature controls proposition 2.7} (or Corollary \ref{curvature controls corollary 2.9}), we have
    \begin{align*}
        \dfrac{||\psi|(x)-|\psi|(y)|}{d(x,y)^{1/2}} &\lesssim \norm{\nabla_A \psi}_{L^6(B_{\rho(x)/2})} \lesssim \norm{\nabla^2_A \psi}_{L^2(B_{\rho(x)/2})}\\
        & \lesssim \rho(x)^{-1/2} \lesssim d(x,y)^{-1/4}.
    \end{align*}
    Hence after re-arranging, we obtain $\norm{\psi}_{C^{0,1/4}} = O(1)$. This means that we do have a uniform control for $\norm{\psi}_{C^{0,\alpha}}$ for $\alpha = \min \{ 1/4, 1/2, \omega/2\}$. We apply this observation to our situation of a sequence of solutions of \eqref{eq:first compactness.2} where the curvatures are uniformly bounded. By the Arzela-Ascoli theorem, after passing through a subsequence and re-indexing, $|\psi_n|$ converges to $|\psi|$ in $C^{0,\alpha}-$topology. By continuity, the nodal set $Z:=|\psi|^{-1}(0)$ is closed. 

    Now we prove the second assertion in the theorem. We prove the weak convergence locally first. Let $x\in Y\setminus Z$. By Proposition \ref{frequency function proposition 3.15}, after passing through a subsequence, the critical radius $\rho(x_n)$ of $(A_n,\psi_n,\epsilon_n)$ is bounded below by a constant $2R$ that depends only on $|\psi|(x)$. From the observation above, if necessary, we can make $R$ smaller so that $|\psi_n|$ is bounded away from zero on $B(x,2R)$. There exists a cover of $B(x,R)$ such that on each member of the covering set, we have $L^2-$bounds for $F_{A_n}$ by Proposition \ref{frequency function proposition 3.22} and Proposition \ref{frequency function proposition 3.16}. Note that we can take this $L^2-$bounds to be the minimum between the constants provided by the aforementioned propositions and the universal bound of the curvatures in the hypothesis. As a result, we can use Proposition \ref{curvature controls proposition 2.7} (or Corollary \ref{curvature controls corollary 2.9}) to get $L^2_{2,A_n}-$bounds $\psi_n$. Putting $A_n$ inside Uhlenbeck gauge on $B(x,R)$, by the Banach-Alaoglu theorem, after passing through a subsequence and re-arranging the index, $(A_n, \psi_n)$ converges weakly to $(A, \psi)$ in $L^2_1$ and $L^2_2$ topology, respectively. A standard argument tells us that we can patch these local gauge transformations together to obtain a global one on $Y\setminus Z$. Note that the limiting solution $(A,\psi)$ must satisfy the following degenerate equations, 
    \[Q_A \psi = 0, \quad \quad \quad \mu(\psi) = 0.\]
    
    Finally, one can follow verbatim the argument in \cite{haydys2015compactness} to achieve the result that $Z$ is nowhere dense in $Y$.
\end{proof}

Recall that to write down the three-dimensional RS-SW equations, besides fixing a $\text{spin}^c$ structure on $Y$, we have to fix an auxiliary choice of a Riemannian metric $g$ on $Y$. Thus, theorem \ref{first compactness theorem} and Theorem \ref{second compactness theorem} are true for a specific a priori fixed $g$. A version of these theorems also holds as we vary the metrics on $Y$. The proof above can be adapted to this more general setting. In particular, we have the following.

\begin{Th}\label{varying compactness theorem}
    Denote $\mathfrak{M}$ by the space of all Riemannian metrics on $Y$. Let $\{g_n\}$ be a sequence of metrics on $Y$ converging to $g \in \mathfrak{M}$. Let $\{(A_n, \psi_n, \epsilon_n)\}$ be a sequence of solutions of the $g_n-$\eqref{eq:first compactness.2} equations (i.e, the blown-up $g_n-$RS-SW equations) such that the $L^6-$norms of $F_{A_n}$ are all uniformly bounded.
    \begin{enumerate}
        \item If $\limsup \epsilon_n > c >0$, then after passing through a subsequence and up to gauge transformations $\{(g_n, (A_n, \psi_n, \epsilon_n))\}$ converges to $(g, (A, \psi, \epsilon))$ in the $C^\infty$ topology. 
        \item If $\limsup \epsilon_n = 0$, then there exists a closed nowhere-dense subset $Z \subset Y$, a connection $A$ on $Y\setminus Z$, a $g-3/2-$spinor $\psi$ on $Y\setminus Z$ such that
            \begin{enumerate}
                \item $Q_A \psi = 0$ and $\mu(\psi) = 0$
                \item $\displaystyle \int_{Y\setminus Z} |\psi|^4 = 1$ and $\displaystyle \int_{Y\setminus Z} |\nabla_A \psi|^2 < \infty$
                \item $|\psi|$ extends to a $C^{0,\alpha}-$H\"older continuous function on $Y$ where $Z= |\psi|^{-1}(0)$;
            \end{enumerate}
            furthermore, $A_n$ converges weakly to $A$ in $L^2_{1,loc}$ and $\psi_n$ converges weakly to $\psi$ in $L^2_{2,loc}$ on $Y\setminus Z$.
    \end{enumerate}
\end{Th}

Motivated by the above theorem, we make the following definition regarding the limiting objects that appear in the moduli space of charge $k$ of the three-dimensional RS-SW equations. 

\begin{Def}\label{charged moduli space}
    Let $g$ be a Riemannian metric on $Y$. We say a moduli space $\mathcal{M}_k(g)$ of \textit{charge $k$} of the three-dimensional RS-SW equations \eqref{eq:first compactness.1} to be the space of gauge equivalence classes of solutions $(A,\psi)$ where $L^6-$norm of $F_A$ is bounded by $k$. Specifically,
    \[\mathcal{M}_k(g) = \left\{ (A, \psi) :  \begin{array}{l}
    (A,\psi) \text{ solves \eqref{eq:first compactness.1}} \\
    \norm{F_A}_{L^6(Y)} \leq k
  \end{array}\right\}\bigg/ \mathcal{G}.\]
\end{Def}

\begin{Def}\label{3/2 fueter section}
    Suppose $g$ is a Riemannian metric on $Y$. Let $Z \subset Y$ be a closed proper subset of $Y$ and $(A, \psi) \in \mathcal{A}(\mathscr{L}|_{Y\setminus Z})\times \Gamma(Y\setminus Z, W_{\mathfrak{s}_{3/2}}\otimes \mathscr{L})$. We call a triple $(A,\psi, Z)$ a $g-3/2-$\textit{Fueter section along singular set $Z$} if it satisfies the conditions $2-(a), (b), (c)$ stated in Theorem \ref{varying compactness theorem}.
\end{Def}

It turns out that the existence of $g-3/2-$Fueter sections on $Y$ provides an obstruction for the compactness of $\mathcal{M}_k(g)$. 

\begin{Cor}\label{obstruction for compactness}
    Let $g \in \mathfrak{M}$. If there is no $g-3/2-$Fueter section, then $\mathcal{M}_k$ is always compact on some neighborhood of $g$ in $\mathfrak{M}$.
\end{Cor}

\begin{proof}
    Since there is no $g-3/2-$Fueter section, by Theorem \ref{varying compactness theorem} (or Theorem \ref{first compactness theorem} and Theorem \ref{second compactness theorem}), the set $\{ \norm{\psi}_{L^4(Y)} : [A,\psi] \in \mathcal{M}_k(g)\}$ must have finite supremum, say, $\nu$. Suppose that there is a sequence of metric $\{g_n\}$ converging to $g$ such that for each $n$ we have
    \[\sup\{\norm{\psi}_{L^4(Y)} : [A,\psi] \in \mathcal{M}_k(g_n)\} \geq \nu + 1.\]
    This means that for each $n$ (also $g_n$), there is a corresponding $(A_n, \psi_n)$ such that $\norm{\psi_n}_{L^4} \geq \nu + 1$. By Theorem \ref{varying compactness theorem}, if the first scenario holds for this sequence $\{(g_n,(A_n, \psi_n))\}$, then the sequence converges to $(g,(A,\psi))$, where $(A,\psi) \in \mathcal{M}_k(g)$. This means that $\norm{\psi}_{L^4} \geq \nu + 1 > \nu$, which contradicts with the observation in the beginning. If the second scenario holds instead, this is also not possible because otherwise, we would obtain the existence of a $g-3/2-$Fueter section. Therefore, there must be a a neighborhood around $g$ in $\mathfrak{M}$ such that for all $g'$ in such neighborhood,
    \[\sup \{\norm{\psi}_{L^4(Y)} : [A,\psi] \in \mathcal{M}_k(g')\}<\nu + 1.\]
    By Theorem \ref{first compactness theorem}, we must have $\mathcal{M}_k(g')$ is compact in the $C^\infty$ topology.
\end{proof}

\section{Discussion}
The non-compactness of the moduli space is a hurdle to derive the monopole invariants. The results about the compactness of the moduli space of the RS-SW equations in this paper serve as the first step in the program of defining a Seiberg-Witten type invariant using the Rarita-Schwinger operator. Corollary \ref{obstruction for compactness} provides an obstruction for compactness, i.e., the non-existence of these so-called $3/2$-Fueter sections will ensure the compactness of the moduli space. Conjecturally, we expect that there should never be a $3/2$-Fueter section. A few indicators are pointing to this. 

Firstly, note that if $(A,\psi)$ is a solution where $\psi$ is a Rarita-Schwinger field, then $(A,\psi)$ is also a solution to a certain $3-$spinor Seiberg-Witten equations. Although the existence of limiting object in the moduli space of $2-$spinor Seiberg-Witten equations has been established by Doan and Walpuski \cite{doan2021existence}, such an analogous statement does not exist for $n-$spinor where $n>2$. Furthermore, the recent work of B\"ar and Mazzeo \cite{MR4252883} suggests that there is a sequence of spin manifolds that have an arbitrarily large number of Rarita-Schwinger fields; however, these manifolds have dimensions at least 4. In the case of multiple-spinor Seiberg-Witten equations, Haydys showed that there is a correspondence between the limiting objects in the moduli space with (non-trivial) solutions to a certain non-linear Dirac-type operator defined on a hyper-K\"ahler fiber bundle over $Y$ \cite{MR2980921}. Via Hitchin's \textit{hyper-K\"ahler quotient construction} \cite{MR0877637}, the hyper-K\"ahler fiber bundle over $Y$ has fibers given by $\mu^{-1}(0)/U(1)$. Roughly, the non-linear Dirac operator is a Dirac-type operator $\mathfrak{F}$ defined on the vertical component of the hyper-K\"ahler fiber bundle. Such an operator is called the \textit{Fueter operator}, and its solutions are called \textit{Fueter sections}. The study of the existence of limiting objects in the moduli space turns into the study of the existence of non-trivial (singular) Fueter sections. 

Morally, whenever there is a Dirac operator, there is a corresponding Rarita-Schwinger operator. Thus, it seems natural to ask whether there is an analogous correspondence in the setting of the Rarita-Schwinger operator and $3/2-$spinors. Inspired by the work of Haydys in \cite{MR2980921}, we propose that $g-$3/2-Fueter sections defined above correspond to non-trivial solutions of a certain non-linear Raritia-Schwinger operator associated with the Fueter operator. Note that, in the setting of $3/2-$spinors, the hyper-K\"ahler fiber bundle is cut down by a linear space $\gamma^{-1}(0)$. This more restrictive condition may lend insights into why one should not expect any 3/2-Fueter sections at all.

We expect the dimension of the moduli space to be zero (at least when $b_1(Y)>1$). Similar to the classical Seiberg-Witten equations, one hopes the invariants are obtained as a ``signed'' count of 
the solutions \cite{MR1781619}.  Conjecturally, this requires dealing with the ``transversality'' of the monopole equation and showing that generic perturbations of the RS-SW equation have only irreducible solutions. This is the next goal of the authors.

Since solutions of the RS-SW equations correspond to the critical points of the modified Chern-Simon-Dirac functional $\mathcal{L}^{RS}$ (cf. Proposition \ref{chern-simon-dirac}), a compactness and transversality result of the moduli space naturally leads to the following questions

\begin{?}
    Is there a homological invariant associated with $\mathcal{L}^{RS}$ via the Floer theory program? If so, does such a Floer homology associated with the RS-SW equations categorify the conjectural numerical invariants described above?
\end{?}

Various gauge-theoretic Floer homology theories turn out to have a formal $(3+1)-$TQFT picture, see \cite{donaldson_furuta_kotschick_2002}, \cite{kronheimer_mrowka_2007}. If the answer to the above question is positive, it is possible to develop a similar TQFT picture for a Floer homology associated with the RS-SW equations. Thus, one should expect that the homological invariant via functoriality should let us define (indirectly) numerical invariants associated with the RS-SW equations for closed $4-$manifolds. Such a supposed numerical invariant is difficult to achieve directly due to the non-compactness of the moduli space of the $4-$D RS-SW equations \cite{nguyen2023pin}.

To implement the Floer theory program in this setting, one has to understand the gradient flow line equations of $\mathcal{L}^{RS}$. It can be shown that the negative gradient flow lines of $\mathcal{L}^{RS}$ satisfy the four-dimensional version of the RS-SW equations defined on $Y\times \RN$
\begin{equation}\label{eq:4dRSSW}
    \begin{cases}
        Q^+_{B}\phi = 0,\\
        F^+_B = \gamma^{-1}(\mu(\phi)).
    \end{cases}
\end{equation}
A differential of a Floer complex should be defined by the count of solutions of \eqref{eq:4dRSSW} interpolating between critical points of $\mathcal{L}^{RS}$. Thus, it is natural for us to investigate the compactification problem of the moduli space of \eqref{eq:4dRSSW}.

\begin{Problem}
    Develop a gauge-theoretic frequency function that can be applied to the analysis of four-dimensional RS-SW equations defined on four-manifolds with cylindrical ends. 
\end{Problem}

We hope to explore these directions in our future works.

\bibliographystyle{amsplain}

\providecommand{\bysame}{\leavevmode\hbox to3em{\hrulefill}\thinspace}
\providecommand{\MR}{\relax\ifhmode\unskip\space\fi MR }

\providecommand{\MRhref}[2]{%
  \href{http://www.ams.org/mathscinet-getitem?mr=#1}{#2}
}

\bibliography{Reference}

\end{document}